\documentclass{article}
\usepackage{amsmath,amssymb,amsthm}
\usepackage{mathrsfs} 
\usepackage[colorlinks=true]{hyperref}
\usepackage{pdfsync}
\usepackage{enumitem}

\def\R{\mathbb{R}}
\def\N{\mathbb{N}}
\def\Q{\mathbb{Q}}
\def\Z{\mathbb{Z}}
\def\C{\mathbb{C}}

\newcommand{\Op}{\mathrm{Op}}

\newcommand{\Lie}{\mathcal{L}}

\newcommand{\trianglesR}{\triangle_{\mathrm{sR}}}
\newcommand{\dsR}{\mathrm{d}_{\mathrm{sR}}}
\newcommand{\BsR}{\mathrm{B}_{\mathrm{sR}}}
\newcommand{\hatdsR}{\widehat{\mathrm{d}}_{\mathrm{sR}}}
\newcommand{\hatBsR}{\widehat{\mathrm{B}}_{\mathrm{sR}}}
\newcommand{\supp}{\mathrm{supp}}
\newcommand{\Cst}{\mathrm{Cst}}

\renewcommand{\geq}{\geqslant}
\renewcommand{\leq}{\leqslant}

\newtheorem{theorem}{Theorem}[section]

\newtheorem{proposition}{Proposition}[subsection]
\newtheorem{corollary}{Corollary}[subsection]
\newtheorem{definition}{Definition}[subsection]
\newtheorem{lemma}{Lemma}[subsection]
\theoremstyle{definition}\newtheorem{example}{Example}[subsection]
\theoremstyle{definition}\newtheorem{remark}{Remark}[subsection]

\newtheorem{innercustomthm}{Theorem}
\newenvironment{customthm}[1]
{\renewcommand\theinnercustomthm{#1}\innercustomthm}
{\endinnercustomthm}

\newcommand{\Qeq}{\mathcal{Q}_\mathrm{eq}}

\title{Small-time asymptotics of hypoelliptic heat kernels near the diagonal, nilpotentization and related results}

\author{Yves Colin de Verdi\`ere\footnote{Institut Fourier, Universit\'e Grenoble Alpes, 100 rue des Math\'ematiques, 38610 Gi\`eres, France (\texttt{yves.colin-de-verdiere@univ-grenoble-alpes.fr}).}
\and
Luc Hillairet\footnote{Institut Denis Poisson, Universit\'e d'Orl\'eans, route de Chartres, 45067 Orl\'eans Cedex 2, France (\texttt{luc.hillairet@univ-orleans.fr}).}
\and
Emmanuel Tr\'elat\footnote{Sorbonne Universit\'e, CNRS, Universit\'e de Paris, Inria, Laboratoire Jacques-Louis Lions (LJLL), F-75005 Paris, France (\texttt{emmanuel.trelat@sorbonne-universite.fr}).}
}

\begin{document}

\maketitle

\begin{abstract}
We establish small-time asymptotic expansions for heat kernels of hypoelliptic H\"ormander operators in a neighborhood of the diagonal, generalizing former results obtained in particular by M\'etivier and by Ben Arous. The coefficients of our expansions are identified in terms of the nilpotentization of the underlying sub-Riemannian structure. Our approach is purely analytic and relies in particular on local and global subelliptic estimates as well as on the local nature of small-time asymptotics of heat kernels. The fact that our expansions are valid not only along the diagonal but in an asymptotic neighborhood of the diagonal is the main novelty, useful in view of deriving Weyl laws for subelliptic Laplacians.
In turn, we establish a number of other results on hypoelliptic heat kernels that are interesting in themselves, such as Kac's principle of not feeling the boundary, asymptotic results for singular perturbations of hypoelliptic operators, global smoothing properties for selfadjoint heat semigroups. 
\end{abstract}

\newpage
\tableofcontents
\newpage


\section{Introduction and main result}
\subsection{Framework}\label{sec_intro}
Let $n$ and $m$ be nontrivial integers. Let $M$ be a smooth connected manifold of dimension $n$. 
Let $X_0,X_1,\ldots,X_m$ be smooth vector fields on $M$ and let $\mathbb{V}$ be a smooth function (potential) on $M$ that is bounded below.
Setting $X=(X_0,X_1,\ldots,X_m)$, we define the H\"ormander operator
\begin{equation}\label{def_DeltaX}
\triangle = \sum_{i=1}^m X_i^2 + X_0 - \mathbb{V} 
\end{equation}
where $X_i$ is seen as a derivation operator and $\mathbb{V}$ is the multiplication by the potential.
In view of involving the case of magnetic fields, the drift vector field $X_0$ can even be assumed to take complex values.\footnote{This requires some obvious slight changes for instance when considering a scalar product.}

Let $\mu$ be an arbitrary smooth (Borel) measure on $M$. 
We assume that the operator $\triangle$ on $L^2(M,\mu)$, of domain $D(\triangle)\subset L^2(M,\mu)$ encoding some possible boundary conditions whenever $M$ has a boundary,
%
generates a strongly continuous semigroup $(e^{t\triangle})_{t\geq 0}$ on $L^2(M,\mu)$ (see Lemma \ref{lem_semigroup} in  Section \ref{sec_kernel} for some sufficient conditions).
We denote by $e_{\triangle,\mu}$ the corresponding heat kernel defined on $(0,+\infty)\times M\times M$, associated with the operator $\triangle$ and with the measure $\mu$ (see Appendix \ref{appendix_Schwartz}).

We set $D=\mathrm{Span}(X_1,\ldots,X_m)$. 
Under the strong H\"ormander condition
\begin{equation}\label{Hormander_condition}
\mathrm{Lie}_q(D)=\mathrm{Lie}_q(X_1,\ldots,X_m)=T_qM\qquad\forall q\in M
\end{equation}
the operator $\triangle$ is subelliptic\footnote{Actually, the weaker assumption $\mathrm{Lie}_q(X_0,X_1,\ldots,X_m)=T_qM$ for every $q\in M$ (called weak H\"ormander condition) is sufficient to ensure subellipticity. The stronger assumption \eqref{Hormander_condition} is however required to derive our main result.} (see \cite{Ho-67}) and the heat kernel $e_{\triangle,\mu}$ is smooth on $(0,+\infty)\times M\times M$.

The objective of this paper is to establish a small-time asymptotic expansion of the heat kernel $e_{\triangle,\mu}$ at any order \emph{near} the diagonal.
Our study is in the line of well known results \cite{BenArous_AIF1989, Metivier1976} establishing such expansions \emph{along} the diagonal. The first main novelty here is that our expansion is valid, not only along the diagonal, but in an asymptotic neighborhood of the diagonal. This fact is actually instrumental in view of deriving local and microlocal Weyl laws for general subelliptic Laplacians, which will be done in the forthcoming papers \cite{CHT-II,CHT-III}.

The second main novelty is that we identify the functions in the small-time asymptotic expansion of the heat kernel in terms of the so-called nilpotentization of the sub-Riemannian structure $(M,D,g)$ where $g$ is a metric on $D$ defined thanks to the vector fields $X_1,\ldots,X_m$.

Compared with the approach of \cite{BenArous_AIF1989} which is probabilistic (Malliavin calculus), our proof (done in Part \ref{part_complete}) is purely analytic and relies in particular on local and global subelliptic estimates, on local and global smoothing properties of heat kernels, on the finite propagation speed property for sR waves and on the Kannai transform, and on the local nature of hypoelliptic heat kernels (Kac's principle). Our paper is entirely selfcontained: even for the several tools (such as uniform local subelliptic estimates) that are straightforward extensions of known results, we provide at least a sketch of proof.

\paragraph{Sub-Riemannian structure.}
Complete reminders on sub-Riemannian (sR) geometry are given in Section \ref{sec_sR}. 
Attached with the $m$-tuple of vector fields $(X_1,\ldots,X_m)$, there is a canonical sR structure $(M,D,g)$, where $D=\mathrm{Span}(X_1,\ldots,X_m)$ and $g$ is a positive definite quadratic form on $D$. Given any $q\in M$, the sR flag at $q$ is the sequence of nested vector subspaces $\{0\}=D^0_q\subset D_q=D^1_q\subset D^2_q\subset \ldots\subset D^{r(q)-1}_q\subsetneq D^{r(q)}_q=T_qM$ defined in terms of successive Lie brackets, and $r(q)$ is the degree of nonholonomy at $q$. Setting $n_i(q) = \dim D^i_q$, the integer $\mathcal{Q}(q) = \sum_{i=1}^r i ( n_i(q)-n_{i-1}(q) )$ coincides with the Hausdorff dimension when $q$ is regular, i.e., when the integers $n_i(\cdot)$ are constant in an open neighborhood of $q$. The point $q$ is said singular when it is not regular.

The nilpotentization of $(M,D,g)$ at $q$ is the sR structure $(\widehat{M}^{q},\widehat{D}^{q},\widehat{g}^{q})$ defined as the metric tangent space of $M$ (endowed with the sR distance) in the sense of Gromov-Hausdorff. In a local chart of privileged coordinates around $q$, $\widehat{M}^{q}$ is a homogeneous space identified to $\R^n$ (with a sR isometry), endowed with dilations $\delta_\varepsilon(x) = \left( \varepsilon^{w_1(q)} x_1,\ldots, \varepsilon^{w_n(q)} x_n \right)$, for $\varepsilon\in\R$ and $x\in\R^n$, where $w_i(q)=i(n_i(q)-n_{i-1}(q))$ (sR weights). 
We have $\widehat{D}^{q}=\mathrm{Span}(\widehat{X}_1^{q},\ldots,\widehat{X}_m^{q})$ where $\widehat{X}_i^{q}$ is the limit of $\varepsilon\delta_\varepsilon^*X_i$ in $C^\infty$ topology as $\varepsilon\rightarrow 0$. We also define the (constant) nilpotentized measure $\widehat{\mu}^{q}$ on $\R^n$ as the limit of $\frac{1}{\vert\varepsilon\vert^{\mathcal{Q}(q)}} \delta_\varepsilon^* \mu$ as $\varepsilon\rightarrow 0$. Finally, we define the sR Laplacian $\widehat{\triangle}^{q} = \sum_{i=1}^m (\widehat{X}^{q}_i)^2$ and we denote by $\widehat{e}^{q} = e_{\widehat{\triangle}^{q},\widehat{\mu}^{q}}$ the heat kernel generated by $\widehat{\triangle}^{q}:D(\widehat{\triangle}^{q})\rightarrow L^2(\widehat M^{q},\widehat\mu^{q})$.

\paragraph{$C^\infty$ topology.}
We recall that the set $C^\infty(\R^n)$ of smooth functions on $\R^n$ is a Montel space, i.e., a Fr\'echet space enjoying the Heine-Borel property (meaning that closed bounded subsets are compact), for the topology defined by the seminorms $p_{i,j}(f) = \max \{ \vert \partial^\alpha f(x)\vert\ \mid\ \vert x\vert\leq i, \vert\alpha\vert\leq j\}$, $(i,j)\in\N^2$. We speak of the $C^\infty(\R^n)$ topology.

Let $W$ be an arbitrary open subset of the manifold $M$. Covering $W$ with charts, the set $C^\infty(W)$ of smooth functions on $W$ is endowed with the $C^\infty$ topology (in charts), making it a Montel space.

We say that a sequence $(f_k)_{k\in\N}$ in $C^\infty(W)$ converges to $0$ in $C^\infty$ topology if $f_k$ converges uniformly to $0$ on any compact subset of $W$, as well as all its derivatives. A sequence $(X_k)_{k\in\N}$ of smooth vector fields on $W$ converges to $0$ in $C^\infty$ topology if all its coefficients (in charts) converge to $0$ in $C^\infty$ topology.

Throughout, we denote by $C_c^\infty(W)$ the set of smooth functions of compact support on $W$.

\paragraph{Notations.}
Throughout the paper, we use in the various estimates the notation $\Cst(\cdot)$, standing for a generic positive constant depending on the parameters indicated in the parenthesis.

The integral of an integrable function $f$ on $M$ with respect to the smooth measure $\mu$ is denoted by $\int_M f(q)\, d\mu(q)$.

\subsection{Main result}

\begin{customthm}{A}\label{main_thm}
{\it 
Let $q\in M$ be arbitrary (regular or not). 
Let $\psi_{q}:U\rightarrow V$ be a chart of privileged coordinates at $q$ such that $\psi_{q}(q)=0$, where $U$ is an open connected neighborhood of $q$ in $M$ and $V$ is an open neighborhood of $0$ in $\R^n$.
We assume that $X_0$ is a smooth section of $D$ over $M$. 
We also assume that $\sup_{q'\in U} r(q')<+\infty$ (this is always satisfied if $\overline U$ is compact).

Then, given any $N\in\N^*$, in the chart\footnote{This means that the left-hand side of \eqref{complete_expansion} is $\vert\varepsilon\vert^{\mathcal{Q}(q)}\, e_{\triangle,\mu}(\varepsilon^2 \tau,\psi_{q}^{-1}(\delta_\varepsilon(x)),\psi_{q}^{-1}(\delta_\varepsilon(x')))$.} we have the asymptotic expansion in $C^\infty((0,+\infty)\times V\times V)$
\begin{equation}\label{complete_expansion}
\vert\varepsilon\vert^{\mathcal{Q}(q)}\, e_{\triangle,\mu}(\varepsilon^2 \tau,\delta_\varepsilon(x),\delta_\varepsilon(x')) \\
= \widehat{e}^{q}(\tau,x,x') + \sum_{i=1}^N \varepsilon^i f^{q}_i(\tau,x,x') + \mathrm{o}\big(\vert\varepsilon\vert^N\big) 
\end{equation}
as $\varepsilon\rightarrow 0$, $\varepsilon\neq 0$, 
where the functions $f^{q}_i$ are smooth and satisfy the homogeneity property
$$
f^{q}_i(\tau,x,x') = \varepsilon^{-i} \vert\varepsilon\vert^{\mathcal{Q}(q)} f^{q}_i(\varepsilon^2 \tau,\delta_\varepsilon(x),\delta_\varepsilon(x')) 
$$
for all $(\tau,x,x')\in(0,+\infty)\times\R^n\times\R^n$ and for every $\varepsilon\neq 0$.

Taking $\tau=1$, $\varepsilon=\sqrt{t}$ and setting $a_i^{q}(x,x')=f^{q}_i(1,x,x')$, it follows that, given any $N\in\N$, in the chart we have the asymptotic expansion in $C^\infty(V\times V)$
\begin{equation}\label{complete_expansion_1}
t^{\mathcal{Q}(q)/2}\, e_{\triangle,\mu}(t,\delta_{\sqrt{t}}(x),\delta_{\sqrt{t}}(x')) \\
= \widehat{e}^{q}(1,x,x') + \sum_{i=1}^N t^{i/2} a_i^{q}(x,x')  + \mathrm{o}(t^{N/2})  
\end{equation}
as $t\rightarrow 0$, $t>0$, where the functions $a_i^{q}$ are smooth and satisfy $a_{2j-1}^{q}(0,0)=0$ for every $j\in\N^*$.

Moreover, if $q$ is regular, then the above convergence and asymptotic expansion are also locally uniform with respect to $q$, and the functions $\widehat{e}^{q}$, $f^{q}_i$ and $a_i^{q}$ depend smoothly (in $C^\infty$ topology) on $q$ in any open neighborhood of $q$ consisting of regular points. If the manifold $M$ is Whitney stratifiable with strata defined according to the sR flag (i.e., the growth vector $(n_1(q),\ldots, n_{r(q)}(q))$ is constant along each stratum) then the latter property is satisfied along strata.
}
\end{customthm}

The fact that the asymptotic expansion \eqref{complete_expansion} is in $C^\infty((0,+\infty)\times V\times V)$ means that the asymptotics are uniform with respect to $(\tau,x,x')$ on any compact subset of $(0,+\infty)\times V\times V$ as well as for all derivatives. In particular, for all $k\in\N$, $\alpha=(\alpha_1,\ldots,\alpha_n)\in\N^n$ and $\beta=(\beta_1,\ldots,\beta_n)\in\N^n$, we have
\begin{equation*}
\lim_{\varepsilon\rightarrow 0\atop \varepsilon\neq 0} \vert\varepsilon\vert^{\mathcal{Q}(q)+2k} \varepsilon^{\sum_{i=1}^n (\alpha_i+\beta_i)w_i(q)}\, (\partial_1^k \partial_2^\alpha \partial_{3}^\beta \, e_{\triangle,\mu}) 
(\varepsilon^2 \tau,\delta_\varepsilon(x),\delta_\varepsilon(x')) 
= (\partial_1^k \partial_2^\alpha \partial_{3}^\beta \, \widehat{e}^{q})  (\tau,x,x')  
\end{equation*}
uniformly with respect to $(\tau,x,x')$ on any compact subset of $(0,+\infty)\times V\times V$.
Here, given a function $e$ depending on three variables $(\tau,y,y')$, the notation $\partial_1$ (resp., $\partial_2$, $\partial_3$) denotes the partial derivative with respect to $\tau$ (resp., to $y$, to $y'$).

\medskip

As a particular case, take $x=x'=0$ in the expansion \eqref{complete_expansion_1} given in Theorem \ref{main_thm} and set $c_j(q) = a_{2j}^{q}(0,0)$.
Since $a_{2j-1}^{q}(0,0)=0$ and $\psi_{q}^{-1}(0)=q$, 
we obtain the following corollary.

\begin{corollary}
Given any $N\in\N$, for every $q\in M$,
$$
t^{\mathcal{Q}(q)/2}\, e_{\triangle,\mu}(t,q,q) \\
= \widehat{e}^{q}(1,0,0) + c_1(q) t + \cdots + c_N(q) t^N + \mathrm{o}(t^{N})  
$$
as $t\rightarrow 0$, $t>0$. Moreover if $q$ is regular then the functions $c_j$ are smooth locally around $q$.
\end{corollary}

We thus recover the main result of \cite{BenArous_AIF1989} (see also \cite{Metivier1976}), which is a small-time expansion of the heat kernel along the diagonal. Here, additionally to those well known results, we provide a geometric interpretation of the coefficients of this expansion, in function of the nilpotentization at $q$: the main coefficient is $\widehat{e}^{q}(1,0,0) >0$, but the other coefficients $c_j(q)$ are also given by convolutions of the heat kernel $\widehat{e}^{q}=e_{\widehat{\triangle}^{q},\widehat{\mu}^{q}}$, as made precise in the proof of the theorem in Part \ref{part_complete} (see in particular Proposition \ref{prop_expansion_semigroup_epsgam} in Section \ref{sec:asymptotic_expansion_epsilongamma}, and see Section \ref{taking_Skernels}).

Moreover, the expansion stated in Theorem \ref{main_thm} is established in an asymptotic neighborhood of the diagonal, which is instrumental to derive the microlocal Weyl law for general equiregular sR structures, or to establish local Weyl laws for singular sR structures (see Section \ref{sec_useful}). 

\begin{remark}\label{rem_X0_length2}
In Theorem \ref{main_thm}, we have assumed that $X_0$ is a smooth section of $D$ over $M$. If $X_0(q)\in D_q$ for every $q\in M$ but cannot be written as a combination of the vector fields $X_i$, $i=1,\ldots,m$, with smooth functions $a_i$, then the asymptotics of the heat kernel in small time along the diagonal may degenerate and be exponentially decreasing (see \cite{BenArousLeandre_PTRF1991_1,BenArousLeandre_PTRF1991_2}).

We have the following generalization if $X_0$ is not a smooth section of $D$.
We assume that $X_0$ is a smooth section of $D^2$ over $M$, meaning that
\begin{equation}\label{hyp_X0_length2}
X_0(q) = \sum_{i=1}^m a_i(q) X_i(q) + \sum_{i,j=1}^m b_{ij}(q) [X_i,X_j](q)\qquad \forall q\in M
\end{equation}
where the $a_i$ and $b_{ij}$ are smooth functions on $M$. In local privileged coordinates around $q$, it is then possible to write $X_0 = X_0^{(-2)} + X_0^{(-1)} + X_0^{(0)} + \cdots$, where $X_0^{(k)}$ is homogeneous of degree $k$ (see Section \ref{sec_nilp_smoothsection}). Therefore, defining on $\R^n$ the vector field $\widehat{X}_0^{q}=X_0^{(-2)}$ (in the local coordinates), which is homogeneous of order $-2$ (and not of order $-1$), we have $\varepsilon^2\delta_\varepsilon^* X_0 = \widehat{X}_0^{q} + \varepsilon X_0^{(-1)} + \varepsilon^2 X_0^{(0)} + \cdots + \varepsilon^NX_0^{(N-2)} + \mathrm{o}(\vert\varepsilon\vert^N)$ at any order $N$, in $C^\infty$ topology.
Then Theorem \ref{main_thm} is still valid provided that $\widehat{\triangle}^{q}$ be replaced with $\widehat{\triangle}^{q}+\widehat{X}_0^{q}$.
In other words, under the assumption \eqref{hyp_X0_length2}, the asymptotics of the heat kernel of $\triangle$ is given by the heat kernel of the operator $\widehat{\triangle}^{q}+\widehat{X}_0^{q}$, which is a first-order perturbation of $\widehat{\triangle}^{q}$, homogeneous of order $-2$.
This is in accordance with results of \cite{BenArousLeandre_PTRF1991_2}.

In contrast, if $X_0$ is not a smooth section of $D^2$ 
then the asymptotics in small time may be completely different, and the heat kernel along the diagonal may even decrease exponentially as $\varepsilon\rightarrow 0$ (see \cite{BenArousLeandre_PTRF1991_2}).
\end{remark}

\begin{remark}
We have stated Theorem \ref{main_thm} in terms of the heat kernels $e_{\triangle,\mu}$ and $\widehat{e}^{q} = e_{\widehat{\triangle}^{q},\widehat{\mu}^{q}} $. 
But, as explained in Appendix \ref{appendix_Schwartz}, the smooth measure $e_{\triangle,\mu}(t,q,q')\, d\mu(q')$ on $M$ does not depend on $\mu$, but only on the operator $\triangle$. The same remark holds for the smooth measure $\widehat{e}^{q}(t,x,x')\, d\widehat{\mu}^{q}(x')$ on $\widehat{M}^q\sim\R^n$.
It would therefore be more natural to express Theorem \ref{main_thm} in terms of Schwartz kernels.
We keep however the statement in this form, because the concept of heat kernel is familiar and is probably the most standard in the literature.

Anyway, it is useful to note that using Schwartz kernels would avoid nilpotentizations of measures.
Moreover, this explains why small-time expansions of heat kernels along the diagonal have no interesting meaning in singular sR cases: what has to be considered, there, is the small-time asymptotics of $e_{\triangle,\mu}(t,q,q)\, d\mu(q)$, which is related to the trace of $e^{t\triangle}f$ (see Appendix \ref{appendix_Schwartz}) and does not depend on $\mu$ nor on the nilpotentization of $\mu$ at the point $q$.
\end{remark}

\subsection{Using Theorem \ref{main_thm} to obtain Weyl laws}\label{sec_useful}
In the forthcoming papers \cite{CHT-II,CHT-III}, we will establish local and microlocal Weyl laws for sR Laplacians in regular and in singular cases.
Let us give a flavor of these results, thus explaining how Theorem \ref{main_thm} can be used to reach this objective.

Considering the general framework of Section \ref{sec_intro}, we assume here moreover that $M$ is compact without boundary
and that $\triangle = \triangle_{sR}$ is selfadjoint (it is a sR Laplacian: see Sections \ref{sec_sRLaplacian} and \ref{sec_rem_Horm}).

Since $\mathrm{Lie}(D) = TM$, the operator $\triangle$ is subelliptic, has a compact resolvent and thus has a discrete spectrum $0=\lambda_1\leq\lambda_2\leq\cdots\leq\lambda_k\cdots\rightarrow+\infty$.
Let $(\phi_k)_{k\in\N^*}$ be an orthonormal eigenbasis of $L^2(M,\mu)$ corresponding to these ordered eigenvalues. The spectral counting function is defined by $N(\lambda ) = \# \{k\in\N^* \mid \lambda_k \leq \lambda \}$ for every $\lambda\in\R$.

\paragraph{Equiregular cases.} We prove in \cite{CHT-II} that, if $(M,D,g)$ is equiregular, i.e., 
if every point of $M$ is regular, then
$$
\int_M e_{\triangle,\mu}(t,q,q)f(q) \, d\mu (q) = \frac{\int _M \widehat{e}^{q}(1,0,0) f(q) \, d\mu(q)}{t^{\mathcal{Q}/2}}  + \mathrm{O}\left( \frac{1}{t^{\mathcal{Q}/2-1}} \right)  
$$
as $t\rightarrow 0$, $t>0$, for every continuous function $f$ on $M$, where 
$\mathcal{Q}$ is the Hausdorff dimension of $M$ (local Weyl law).
Actually, we establish a small-time expansion at any order, by using the complete statement of Theorem \ref{main_thm} (asymptotic expansion at any order).
In particular, the spectral counting function has the asymptotics
$$
N(\lambda) \underset{\ \lambda\rightarrow+\infty}{\sim} \frac{\int_M \widehat{e}^{q}(1,0,0) \, d\mu(q)}{\Gamma(\mathcal{Q}/2+1)} \lambda^{\mathcal{Q}/2} .
$$
The above limits are an easy consequence of Theorem \ref{main_thm} along the diagonal (and thus,  already follow from \cite{BenArous_AIF1989,Metivier1976}). Indeed, taking $x=x'=0$ and $\varepsilon=\sqrt{t}$ in Theorem \ref{main_thm}, we obtain 
\begin{equation*}
e_{\triangle,\mu}(t,q,q)\, d\mu(q) \underset{\ t\rightarrow 0\atop t>0}{\sim} \frac{\widehat{e}^{q}(1,0,0)}{t^{\mathcal{Q}(q)/2}}\, d\mu(q) 
=  \widehat{e}^{q}(t,0,0) \, d\mu(q)  \qquad \forall q\in M .
\end{equation*}
The result follows by dominated convergence.

In turn, this result puts in evidence an intrinsic sR measure that we call the \emph{Weyl measure}, of which there exists a local and a microlocal version. The \emph{local Weyl measure} $w_{\triangle}$ is the probability measure on $M$ defined by
$$
\int _M f \, dw_{\triangle} = \lim _{\lambda \rightarrow +\infty }\frac{1}{N(\lambda)}\sum _{\lambda_k \leq \lambda } \int_M f |\phi_k |^2 \, d\mu \qquad \forall f\in C^0(M)
$$
whenever the limit exists for all continuous functions $f$, i.e., $w_{\triangle}$ is the weak limit of the sequence of probability measures $\frac{1}{N(\lambda)}\sum _{\lambda_k \leq \lambda } |\phi_k |^2 \, \mu $ (Ces\`aro mean) as $\lambda \rightarrow +\infty$.
The above argument shows that, in the equiregular case, the local Weyl measure exists, does not depend on $\mu$, is a smooth measure on $M$ and its density with respect to $\mu$ is 
$$
dw_{\triangle}(q) = \frac{\widehat{e}^{q}(1,0,0)}{\int_M \widehat{e}^{q'}(1,0,0)\, d\mu(q')} \, d\mu(q) .
$$
Accordingly, the \emph{microlocal Weyl law} $W_\triangle$ is the probability measure defined on the co-sphere bundle $S^\star M$ by
$$
\int _{S^\star M} a \, dW_{\triangle} = \lim _{\lambda \rightarrow +\infty }\frac{1}{N(\lambda)}\sum _{\lambda_k \leq \lambda } \langle \Op (a)\phi_k , \phi_k \rangle_{L^2(M,\mu)} 
$$
for every symbol $a$ of order $0$, whenever the limit exists for all symbols $a$ of order $0$ (here, $\Op$ denotes any quantization operator).
We prove in \cite{CHT-II} that, in the equiregular case, the microlocal Weyl law exists and we provide its explicit expression, showing in particular that $W_\triangle$ is supported on $S(D^{r-1})^\perp$ where $r$ is the degree of nonholonomy.
This generalizes to equiregular cases a result obtained in \cite{CHT-I} in the three-dimensional contact case.

Establishing this result instrumentally relies on the fact that, taking $x'=0$ in Theorem \ref{main_thm}, we obtain that, for every $q\in M$, in the chart where $\psi_{q}(q)=0$,
\begin{equation*}
\vert\varepsilon\vert^{\mathcal{Q}(q)}\, e_{\triangle,\mu}(\varepsilon^2,\delta_\varepsilon(x),0) \, d\widehat{\mu}^{q}(x) 
\underset{\varepsilon\rightarrow 0\atop \varepsilon\neq 0}{\sim} \widehat{e}^{q}(1,x,0) \, d\widehat{\mu}^{q}(x) 
\end{equation*}
uniformly with respect to $x$ on any compact subset of $\psi_{q}(U)$.

\paragraph{Singular cases.}
Famous singular sR structures are given by the Grushin case in dimension two or by the Martinet case in dimension three. They are chiefly studied in \cite{CHT-II}. One of the main tools is the fact that, taking $x=x'$ in Theorem \ref{main_thm}, we obtain that, for every $q\in M$, in the chart where $\psi_{q}(q)=0$,
\begin{equation*}
\lim_{\varepsilon\rightarrow 0\atop \varepsilon\neq 0} \vert\varepsilon\vert^{\mathcal{Q}(q)}\, e_{\triangle,\mu}(\varepsilon^2,\delta_\varepsilon(x),\delta_\varepsilon(x))
= \widehat{e}^{q}(1,x,x) 
\end{equation*}
uniformly with respect to $x$ on any compact subset of $\psi_{q}(U)$.
For instance, in the Grushin (resp., Martinet) case the asymptotics of the spectral counting function is $\Cst\, \lambda\ln\lambda$ (resp., $\Cst\,\lambda^2\ln\lambda$) as $\lambda\rightarrow+\infty$; this fact is known for the Grushin case (see \cite{Me-Sj-78}) but is new for the Martinet case.
We even obtain two-terms asymptotic expansions of the local Weyl law, with intrinsic coefficients, by using the complete statement of Theorem \ref{main_thm} (asymptotic expansion). 
In \cite{CHT-III} we prove that, in some sense, the occurence of a (power of a) logarithm in the spectral counting function is the highest possible complexity. More precisely, we establish that, given any sR structure whose singular set is Whitney stratifiable, 
$$
\int_M e_{\triangle,\mu}(t,q,q)f(q) \, d\mu (q) \sim \mathrm{Cst} \frac{\vert\ln^k t\vert}{t^\gamma} 
$$
as $t\rightarrow 0$, $t>0$, for every continuous function $f$ on $M$, with $k\in\{0,1,\ldots,n\}$ and $\gamma\in\Q$ such that $\gamma\geq \frac{\Qeq}{2}$, where $\Qeq$ is the Hausdorff dimension of the equiregular region of $M$. Moreover if $k=n$ then $\gamma=\frac{\Qeq}{2}$.
As a consequence the asymptotics of $N(\lambda)$ is $\Cst\,\lambda^\gamma\ln^k\lambda$ as $\lambda\rightarrow +\infty$: this gives the maximal complexity, for instance no term $\ln\ln\lambda$ appears in the asymptotics.

\subsection{Structure of the paper}
The paper is structured as follows.

\medskip

We provide in Section \ref{sec_sR} some reminders in sub-Riemannian (sR) geometry. In particular, we recall the instrumental concept of nilpotentization, much relying on sR dilations that are used in our main result. Also, for the sake of completeness, in Section \ref{sec_kernel} we recall sufficient assumptions ensuring existence of the hypoelliptic heat kernel.

\medskip

The paper is then split into two parts. The reason is the following.

While our main result, Theorem \ref{main_thm}, states a small-time asymptotic expansion of the heat kernel near the diagonal at any order, obtaining only the limit, i.e., only the first term, is much less difficult than obtaining the complete expansion.

Actually, the mathematical techniques and results that are required to obtain the limit are purely of a \emph{local} nature, and thus, do not require lengthy developments.
Since the result is already interesting (in particular, for obtaining Weyl laws), in Part \ref{part_limit} we establish Theorem \ref{main_thm_weak}, which is Theorem \ref{main_thm} at the order zero.
Part \ref{part_complete} is then devoted to establishing the complete statement of Theorem \ref{main_thm}, which is surprisingly much more difficult and requires results of a \emph{global} nature, as explained in detail at the beginning of Part \ref{part_complete}.

The two parts, as well as the sections therein, are redacted independently enough one from each other, in order to allow several possible levels of reading. For instance, the reader only interested in the limit result may read only Part \ref{part_limit} (and Appendix \ref{sec:local_estimates} on local subelliptic estimates).
Besides, a reader, even though not interested in the complete asymptotic expansion, may however find in Part \ref{part_complete} (and in Appendices \ref{sec:global_estimates} and \ref{sec:global_sR} on global subelliptic estimates) a number of tools of interest to deal with global issues.

\paragraph{Structure of Part \ref{part_limit}.}
This part is concerned with \emph{local} issues, sufficient to prove Theorem \ref{main_thm_weak}.

Section \ref{sec_facts} in Part \ref{part_limit} is devoted to establishing some general facts on hypoelliptic heat kernels:
\begin{itemize}
\item In Section \ref{sec_Hormander_kernel}, we gather some remarks on H\"ormander operators (intrinsic formula of integration by parts on a domain with boundary, symmetry properties, with Dirichlet or Neumann boundary conditions), useful in order to state a general result for existence of semigroups (Lemma \ref{lem_semigroup}).
\item In Section \ref{sec_two_general_results}, we establish two general results for parameter-dependent hypoelliptic heat kernels:
\begin{itemize}
\item In Section \ref{sec_localheat}, we prove that the small time asymptotics of hypoelliptic heat kernels is purely local: this reflects the famous Kac's principle of not feeling the boundary. Our version (hypoelliptic Kac's principle, Theorem \ref{theo:localheat}) is moreover uniform with respect to parameters.
\item In Section \ref{sec_convergence_heat_kernels}, we give a general result for singular perturbations of hypoelliptic operators (Theorem \ref{thm_CV_general1}): assuming that the H\"ormander operator depends continuously on some parameter $\tau$, we prove that the corresponding heat kernel depends as well continuously on $\tau$.
\end{itemize}
\end{itemize}

\medskip

Section \ref{sec:proof_main_thm_prelim} is dedicated to proving Theorem \ref{main_thm_weak}, which corresponds to proving the limit in Theorem \ref{main_thm}. As said before, this proof only requires to use local tools.
As sketched at the beginning of Part \ref{part_limit}, the argument starts by applying the Trotter-Kato theorem to the operator $\triangle^\varepsilon = \varepsilon^2 \delta_\varepsilon^*\triangle(\delta_\varepsilon)_*$, which converges to $\widehat{\triangle}^{q}$ in $C^\infty$ topology, before using uniform local subellipticity (Appendix \ref{sec:local_estimates}) to infer strong convergence properties.

\paragraph{Structure of Part \ref{part_complete}.} This part is devoted to establishing the complete statement of Theorem \ref{main_thm}. In contrast to the proof of Theorem \ref{main_thm_weak}, the proof of Theorem \ref{main_thm} requires \emph{global} smoothing properties. 

Since the proof is quite lengthy, in Section \ref{sec_ideaoftheproof} we give the idea of our approach to the proof and we list a number of properties that are required. We point out the main difficulties, in order to motivate the developments done in Section \ref{sec:proof_main_thm} and in Appendix.

Theorem \ref{main_thm} is proved in full detail in Section \ref{sec:proof_main_thm}. 
The proof starts by applying iteratively the Duhamel formula but, as explained in Section \ref{sec_ideaoftheproof}, getting the complete asymptotic expansion out of it raises serious difficulties and requires global considerations.
To facilitate the reading we have organized this section as follows:
\begin{itemize}
\item The operator $\triangle^{\varepsilon,\gamma}$, adequate modification of the operator $\triangle^\varepsilon$ with a ``damping" parameter $\gamma>0$ so as to be (uniformly) at most polynomial at infinity, is defined and analyzed in Section \ref{sec_construction_deltaepsiloneta}. In particular, the key lemmas \ref{lem_epsgamma} and \ref{lem_unif_poly_Hormander} state strong global convergence properties of $\triangle^{\varepsilon,\gamma}$ to $\widehat{\triangle}^{q}$
that are required in our proof (and which do not hold for $\triangle^\varepsilon$ that converges to $\widehat{\triangle}^{q}$ only in $C^\infty$ topology).
\item In Section \ref{sec:asymptotic_expansion_epsilongamma}, and more precisely, in Proposition \ref{prop_expansion_semigroup_epsgam}, we prove that $e^{t\triangle^{\varepsilon,\gamma}}$ has a small-time asymptotic expansion with respect to $\varepsilon$ at any order, in the sense of (uniformly) smoothing operators. The proof of this proposition is delicate and uses in an instrumental way, as explained and motivated in Section \ref{sec_ideaoftheproof}, the global smoothing properties established for $e^{t\widehat{\triangle}^{q}}$ in Appendix \ref{sec:global_sR}, and those established for $e^{t\triangle^{\varepsilon,\gamma}}$ in Appendix \ref{sec:global_estimates} thanks to (uniform) global subelliptic estimates.
\item Taking Schwartz kernels in Section \ref{taking_Skernels}, we obtain the asymptotic expansion of the heat kernel.
\item The end of the proof, in Section \ref{end_of_the_proof_main_thm}, consists of applying the localization theorem (hypoelliptic Kac's principle).
\end{itemize}
As noted in Section \ref{sec_idea_eps}, in the particular case where $M=\R^n$, all vector fields $X_i$ are polynomial and $\triangle$ is selfadjoint, it is not necessary to resort to the modified operator $\triangle^{\varepsilon,\gamma}$, and the global results established in Appendix \ref{sec:global_sR} are sufficient to achieve the proof of Theorem \ref{main_thm}.

\paragraph{Structure of Part \ref{part_appendix} (appendix).} We have gathered in the appendix the following material.

In Appendix \ref{appendix_Schwartz}, we recall some known facts on Schwartz and heat kernels.
In particular, we point out the meaningful fact that Schwartz kernels do not depend on the measure while heat kernels do.

\medskip

In Appendix \ref{sec:uniform_subelliptic_estimates}, we establish subelliptic estimates and smoothing properties for hypoelliptic heat semigroups:
\begin{itemize}
\item Local estimates in Appendix \ref{sec:local_estimates}, uniform with respect to some parameters: although these are standard (adding dependence with respect to parameters is straightforward), we give sketches of proofs, in order to prepare the reader to global estimates.
\item Global estimates in Appendix \ref{sec:global_estimates}, uniform with respect to some parameters, established for parameter-dependent H\"ormander operators whose growth at infinity is at most polynomial, satisfying a uniform polynomial H\"ormander condition.
\end{itemize}

\medskip

In Appendix \ref{sec:global_sR}, we derive a number of more precise and stronger global smoothing properties in the case where the operator $\triangle$ is a (selfadjoint) sR Laplacian, in Sobolev spaces or iterated domains with polynomial weight. Our arguments are based on the Kannai transform combined with the finite speed propagation property for sR waves and on upper exponential estimates for the sR heat kernel.

%
%
%

\section{Reminders: sub-Riemannian (sR) structure}\label{sec_sR}
Attached with the $m$-tuple of vector fields $(X_1,\ldots,X_m)$, there is a canonical sub-Riemannian (sR) structure.
This section consists of reminders in sR geometry (see the textbooks \cite{AgrachevBarilariBoscain_book2019, Bellaiche, Gromov, Jean_2014, LeDonne_book2017, Mo-02, Rifford_2014}), which are useful in our analysis.

\subsection{Sub-Riemannian metric}
The sR metric $g$ associated with $(X_1,\ldots,X_m)$ is defined as follows: given any $q\in M$ and any $v\in D_q=\mathrm{Span}(X_1(q),\ldots,X_m(q))$, we define the positive definite quadratic form $g_q$ on $D_q$ by
$$
g_q(v) = \inf \left\{ \sum_{i=1}^m u_i^2 \ \ \Big\vert \ \ v=\sum_{i=1}^mu_iX_i(q) \right\} .
$$
The triple $(M,D,g)$ is called a sub-Riemannian structure (see \cite{Bellaiche,Rifford_2014}). 
When $D$ has constant rank $m$ on $M$ with $m\leq n$, $D$ is a subbundle of $TM$ endowed with the Riemannian metric $g$, and the frame $(X_1,\ldots,X_m)$ is $g$-orthonormal. But the rank of $D$ may vary (i.e., $D$ is a subsheaf of $TM$) and the above definition encompasses the so-called almost-Riemannian case, for which $m\geq n$ and $\mathrm{rank}(D)<n$ at some singular points.

More formally, a sR structure on $M$ can be defined by an Euclidean vector bundle $E$ over $M$ and a smooth vector bundle morphism $\sigma:E\rightarrow TM$, with $D_q=\sigma(E_q)$ and $g_q(V)=\inf\{\Vert u\Vert_{E_q}^2\ \mid\ u\in E_q,\ \sigma(u)=V\}$ for every $q\in M$. When $E=M \times\R^m$ and $\sigma(x,u)=\sum_{i=1}^m u_i X_i(x)$, we recover the definition of a sR structure attached with the $m$ vector fields $X_1,\ldots,X_m$.

A horizontal path is, by definition, an absolutely continuous path $q(\cdot):[0,1]\rightarrow M$ for which there exist $m$ functions $u_i\in L^1(0,1)$ such that $\dot q(t) = \sum_{i=1}^m u_i(t)X_i(q(t))$ for almost every $t\in[0,1]$.
The metric $g$ induces a length on the set of horizontal paths, and thus a distance $\dsR$ on $M$ that is called the sR distance.

The cometric $g^*$ associated with $(X_1,\ldots,X_m)$ is the nonnegative quadratic form on $T^*M$ defined as follows: given any $q\in M$, $g^*_q$ is the nonnegative quadratic form defined on $T_q^*M$ by $g^*_q(\xi)=\sum_{i=1}^m\langle\xi,X_i(q)\rangle^2$.
Note that $\frac{1}{2} g_q(v) = \sup_{\xi\in T^*_qM} \left( \langle\xi,v\rangle - \frac{1}{2}g^*_q(v) \right)$ (Legendre transform).

Given any smooth function $f$ on $M$, the horizontal gradient $\nabla_gf$ of $f$ is the smooth section of $D$ defined by $g(\nabla_gf,X)=df.X$ for every smooth section $X$ of $D$. We have $\nabla_gf = \sum_{i=1}^m (X_if)X_i$.

\subsection{Sub-Riemannian Laplacian}\label{sec_sRLaplacian}
We denote by $\mathrm{div}_\mu$ the divergence operator associated with the smooth measure $\mu$ on $M$, defined by $\Lie_X \mu=\mathrm{div}_\mu(X)\, \mu$ for every smooth vector field $X$ on $M$. Here, $\Lie_X$ is the Lie derivative along $X$.
The sR Laplacian $\trianglesR$ is defined as the differential operator
\begin{equation}\label{def_DeltasR}
\trianglesR f = \mathrm{div}_\mu(\nabla_g f) = \sum_{i=1}^m \left( X_i^2f + \mathrm{div}_\mu(X_i)X_if \right) \qquad \forall f\in C^\infty(M).
\end{equation}
Its principal symbol is the cometric $g^*$.
The sR Laplacian is a particular instance of a H\"ormander operator: we have $\trianglesR=\triangle$ with $X_0 = \sum_{i=1}^m \mathrm{div}_\mu(X_i)X_i$ and $\mathbb{V}=0$.

Let $\Omega$ be an open subset of $M$.
By integration by parts, we have
$$
\langle \trianglesR u,v\rangle_{L^2(\Omega,\mu)} = - \int_\Omega g(\nabla_g u,\nabla_g v)\, d\mu + \int_{\partial\Omega} v\, d\left(\iota_{\nabla_gu}\mu\right)  \qquad \forall u,v\in C^\infty(M)
$$
where $\iota_{\nabla_gu}$ is the interior product of $\mu$ and $\nabla_gu$.
We infer the sR Green formula
$$
\langle \trianglesR u,v\rangle_{L^2(\Omega,\mu)} = \langle u,\trianglesR v\rangle_{L^2(\Omega,\mu)} + \int_{\partial\Omega} v\, d\left(\iota_{\nabla_gu}\mu\right) - \int_{\partial\Omega} u\, d\left(\iota_{\nabla_gv} \mu\right) \qquad \forall u,v\in C^\infty(M).
$$
Hence, $\trianglesR$ is symmetric and dissipative on $C^\infty(\Omega)$ in the two following cases:
\begin{itemize}
\item Dirichlet case: $f=0$ along $\partial\Omega$;
\item Neumann case: $\iota_{\nabla_gf}\mu=0$ along $\partial\Omega$.
\end{itemize}
In these two cases, $\trianglesR$ has selfadjoint extensions; moreover, if the manifold $\Omega$ endowed with the induced sR distance is complete then $\trianglesR$ is essentially selfadjoint (see \cite{Strichartz_JDG1986}) and thus has a unique selfadjoint extension. We speak then of the Dirichlet (resp., Neumann) sR Laplacian, which is defined on the maximal domain that is the completion in $L^2(\Omega,\mu)$ of the subset of $f\in C^\infty(\Omega)$ satisfying $f=0$ (resp., $\iota_{\nabla_gf}\mu=0$) along $\partial\Omega$.

Note that, if $\Omega=M$ is compact or if $\Omega=M=\R^n$ or if one considers Dirichlet boundary conditions then the adjoint of $X_i$ in $L^2(\Omega,\mu)$ is $X_i^*=-X_i-\mathrm{div}_\mu(X_i)$ 
and then $\trianglesR=-\sum_{i=1}^mX_i^*X_i$.

\subsection{Sub-Riemannian flag}\label{sec:sRflag}
We define the sequence of subsheafs $D^k$ of $TM$ by $D^0=\{0\}$, $D^1=D=\mathrm{Span}(X_1,\ldots,X_m)$ and $D^{k+1}=D^k+[D,D^k]$ for $k\geq 1$. Under the strong H\"ormander condition \eqref{Hormander_condition}, given any point $q\in M$, we have the flag
$$
\{0\}=D^0_q\subset D_q=D^1_q\subset D^2_q\subset \ldots\subset D^{r(q)-1}_q\subsetneq D^{r(q)}_q=T_qM 
$$
where $r(q)$ is called the \emph{degree of nonholonomy} at $q$. We set $n_i(q) = \dim D^i_q$. The $r(q)$-tuple of integers $(n_1(q),\ldots,n_{r(q)}(q))$ is called the \emph{growth vector} at $q$, and we have $n_{r(q)}(q)=n=\dim M$. By convention, we set $n_0(q)=0$.
We define the nondecreasing sequence of weights $w_i(q)$ as follows:
given any $i\in\{1,\ldots,n\}$, there exists a unique $j\in\{1,\ldots,n\}$ such that $n_{j-1}(q)+1\leq i\leq n_j(q)$, and we set $w_i(q)=j$.
By definition, we have $w_1(q)=\cdots=w_{n_1}(q)=1$, and $w_{n_{j-1}+1}(q)=\cdots=w_{n_j}(q)=j$ when $n_j(q)>n_{j-1}(q)$. We also have $w_{n_{r-1}+1}(q)=\cdots=w_{n_r}(q)=r(q)$.
Note that $n_{w_{n_j}}(q) = n_j(q)$ for $j=1,\ldots,r$ and that $w_{n_j}(q) = j$ if (and only if) $n_j(q)>n_{j-1}(q)$.

Given any $q\in M$, we set
$$
\mathcal{Q}(q) = \sum_{i=1}^r i ( n_i(q)-n_{i-1}(q) ) =\sum_{i=1}^n w_i(q).
$$
If $q$ is regular then $\mathcal{Q}(q)$ is the Hausdorff dimension of a small ball in $M$ containing $q$ endowed with the induced corresponding sR distance (see \cite{Gromov}).

A point $q\in M$ is said to be \emph{regular} if the growth vector is constant in a neighborhood of $q$; otherwise it is said to be \emph{singular}. The sR structure is said to be \emph{equiregular} if all points of $M$ are regular; in this case, the weights and the Hausdorff dimension are constant as well on $M$.

We recall that, at a point $q$ that is regular, with degree of nonholonomy $r(q)$, $\triangle$ is locally hypoelliptic and even subelliptic with a gain of regularity $2/r(q)$, meaning that if $\triangle f=g$ with $g$ of Sobolev class $H^s$ locally at $q$ then $f$ is (at least) of Sobolev class $H^{s+2/r(q)}$ locally at $q$ (see \cite{Ho-67}). 



\subsection{Sub-Riemannian isometries}
Given two sR structures $(M_1,D_1,g_1)$ and $(M_2,D_2,g_2)$, of respective cometrics $g^*_1$ and $g^*_2$, a (local) sR isometry $\phi:M_1\rightarrow M_2$ is a (local) smooth diffeomorphism mapping $g^*_1$ to $g^*_2$. 



Note that, if $M$ is a Lie group equipped with a left-invariant sR structure, then the left action is a sR isometry on $M$.

\subsection{Nilpotentization of the sub-Riemannian structure}\label{sec:nilp}
\subsubsection{First definition}
Let $q\in M$ be arbitrary. The nilpotentization of the sR structure $(M,D,g)$ at $q$ is defined as the metric tangent space of $M$ (endowed with its sR distance) in the sense of Gromov-Hausdorff (see \cite{Bellaiche,Gromov}). It is identified, with a sR isometry, to the sR structure $(\widehat{M}^{q},\widehat{D}^{q},\widehat{g}^{q})$ defined hereafter, where $\widehat{M}^{q}$ is a smooth connected manifold of dimension $n$ (as a topological space, $\widehat{M}^{q}$ is the usual tangent space to $M$ at $q$), $\widehat{D}^{q}=\mathrm{Span}(\widehat{X}^{q}_1,\ldots,\widehat{X}^{q}_m)$ with smooth vector fields $\widehat{X}^{q}_1,\ldots,\widehat{X}^{q}_m$ on $M$ (defined hereafter) called nilpotentizations at $q$ of the vector fields $X_1,\ldots,X_m$ at $q$, and the sR metric $\widehat{g}^{q}$ is defined, accordingly, by
$$
\widehat{g}^{q}_x(v) = \inf \left\{ \sum_{i=1}^m u_i^2 \ \ \Big\vert \ \ v = \sum_{i=1}^m u_i \widehat{X}^{q}_i(x) \right\}\qquad \forall x\in \widehat{M}^{q}\qquad \forall v\in \widehat{D}^{q}_x .
$$
The metric $\widehat{g}^{q}$ induces 
a distance $\hatdsR^{q}$ on $\widehat{M}^{q}$.

To define $(\widehat{M}^{q},\widehat{D}^{q},\widehat{g}^{q})$, it suffices to define one of the elements of the equivalence class under the action of sR isometries.
A standard description consists of using charts in $M$ of so-called \textit{privileged coordinates}, and then to identify $\widehat{M}^{q}\simeq\R^n$ with a sR isometry, as follows.

\subsubsection{Privileged coordinates}
We first recall the notion of nonholonomic order (see \cite{Bellaiche,Jean_2014,Mo-02} for details).
Given a germ $f$ of a smooth function at $q$, given $k\in\N$ and integers $j_1,\ldots,j_k$ in $\{1,\ldots,m\}$, the Lie derivative $\Lie_{X_{j_1}}\cdots\Lie_{X_{j_k}}f(q)$ is called a nonholonomic derivative of order $k$. By definition, the nonholonomic order of $f$ at $q$, denoted by $\mathrm{ord}_{q}(f)$, is the smallest integer $k$ for which at least one nonholonomic derivative of $f$ of order $k$ at $q$ is not equal to zero.
Given a germ $Y$ of a smooth vector field at $q$, the nonholonomic order of $Y$ at $q$ is the largest integer $k$ such that $\mathrm{ord}_{q}(\Lie_Yf)\geq k+\mathrm{ord}_{q}(f)$, for every germ $f$ at $q$.

A family $(Y_1,\ldots,Y_n)$ of $n$ vector fields is said to be adapted to the flag at $q$ if it is a frame of $T_{q}M$ at $q$ and if $Y_i(q)\in D_{q}^{w_i(q)}$, for every $i\in\{1,\ldots,n\}$.

A system of privileged coordinates at $q$ is a system of local coordinates $(x_1,\ldots,x_n)$ at $q$ such that $\mathrm{ord}_{q}(x_j)=w_j(q)$, for every $i\in\{1,\ldots,n\}$.
Note that we must have $dx_i(D_{q}^{w_i(q)})\neq 0$ and $dx_i(D_{q}^{w_i(q)-1})= 0$, meaning that $\partial_{x_i}\in D_{q}^{w_i(q)}\setminus D_{q}^{w_i(q)-1}$ at $q$ (i.e., privileged coordinates are always adapted to the flag).

An example of privileged coordinates at $q$ is given by 
\begin{equation}\label{example_privileged}
(x_1,\ldots,x_n)\mapsto\exp \left( \sum_{i=1}^nx_i Z^q_i \right) (q)
\end{equation}
where $(Z^q_i)_{1\leq i\leq n}$ is a frame of vector fields that is adapted to the flag at $q$.

Privileged coordinates can be obtained from any system of adapted coordinates by a triangular change of variables (see \cite{Jean_2014}).

\subsubsection{Dilations and nilpotentization of smooth sections of $D$}\label{sec_nilp_smoothsection}
We consider a chart of privileged coordinates at $q$, that is a smooth mapping $\psi_{q}:U\mapsto\R^n$, where $U$ is a neighborhood of $q$ in $M$, with $\psi_{q}(q)=0$, inducing local coordinates $x=\psi_{q}(q)$ in which the vector fields $(\psi_{q})_*X_i$ ($i=1,\ldots,m$) have a nilpotent approximation $\widehat{X}^{q}_i$ with the following precise meaning. 
In these local coordinates, for every $\varepsilon\in\R$, the dilation $\delta_\varepsilon$ is defined in $\R^n$, according to the flag at $q$, by 
\begin{equation*}
\delta_\varepsilon(x) = \left( \varepsilon^{w_1(q)} x_1,\ldots, \varepsilon^{w_n(q)} x_n \right) \qquad\forall x=(x_1,\ldots,x_n)\in\R^n .
\end{equation*}
Note that, denoting by $m$ the Lebesgue measure on $\R^n$ (given by $dm=dx_1\cdots dx_n$), we have $\delta_\varepsilon^*m=\vert\varepsilon\vert^{\mathcal{Q}(q)}m$ for every $\varepsilon\neq 0$.

Given any vector field $X$ on $M$ that is a smooth section%
\footnote{Note that we consider a smooth section of the subsheaf $D$, otherwise there are some difficulties: take $M=\R^2$, $D$ spanned by $X_1=\partial_x$ and $X_2=x^6\partial_y$, and the vector field $X=x^2\partial_y$.}
of $D$ (i.e., $X(q) = \sum_{i=1}^m a_i(q) X_i(q)$ at any point $q$, with smooth functions $a_i$), the nilpotentization $\widehat{X}^{q}$ at $q$ of $X$ is the (nilpotent and complete) vector field on $\R^n$ defined in the chart by 
$$
\widehat{X}^{q} = \lim_{\varepsilon\rightarrow 0\atop \varepsilon\neq 0}\varepsilon \delta_\varepsilon^* X . 
$$
Actually this convergence is valid in $C^\infty$ topology (uniform convergence of all derivatives on compact subsets of $\R^n$).
Note that $\widehat{X}^{q}$ is homogeneous of order $-1$ with respect to dilations, i.e., $\lambda\delta_\lambda^*\widehat{X}^{q}=\widehat{X}^{q}$ for every $\lambda\neq 0$,
and that the nonholonomic order of $X-\widehat{X}^{q}$ at $q$ is nonnegative.
Actually, setting $X^\varepsilon=\varepsilon\delta_\varepsilon^*X$ and writing in $C^\infty$ topology the Taylor expansion $X=X^{(-1)}+X^{(0)}+X^{(1)}+\cdots$ around $0$, where $X^{(k)}$ is polynomial and homogeneous of degree $k$ (with respect to dilations), we get that $X^\varepsilon$ has a Taylor expansion at any order $N$ with respect to $\varepsilon$, in $C^\infty$ topology:
\begin{equation}\label{expansion_Xeps}
X^\varepsilon = \varepsilon \delta_\varepsilon^* X = \widehat{X}^{q} + \varepsilon X^{(0)} + \varepsilon^2 X^{(1)} + \cdots + \varepsilon^N X^{(N-1)} + \mathrm{o}\big(\vert\varepsilon\vert^N\big)
\end{equation}
with $\widehat{X}^{q} = X^{(-1)}$ (see also \cite[Lemma 1]{Ba-13}), i.e., setting $X^0=\widehat{X}^{q}$ for $\varepsilon=0$, $X^\varepsilon$ depends smoothly on $\varepsilon$ in $C^\infty$ topology. We also have
\begin{equation}\label{expansion_XZeps}
X^\varepsilon = \widehat{X}^{q} + \varepsilon Z^\varepsilon
\end{equation}
for every $\varepsilon\in\R$ with $\vert\varepsilon\vert$ small enough so that we are in the chart, where $Z^\varepsilon$ is a smooth vector field depending smoothly on $\varepsilon$ in $C^\infty$ topology.

\subsubsection{Definition of the nilpotentization of the sR structure}\label{sec_def_nilpotentization}
In the above chart, we define $\widehat{M}^{q}\simeq\R^n$, endowed with the sR structure (denoted by $(\widehat{M}^{q},\widehat{D}^{q},\widehat{g}^{q})$) induced by the vector fields $\widehat{X}^{q}_i$, $i=1,\ldots,m$.
This definition does not depend on the choice of privileged coordinates at $q$ because two sets of such coordinates produce two sR-isometric sR structures. 
This is due to the fact that, since transition maps of charts of privileged coordinates are triangular with respect to the flag, the nilpotentization of any transition map is a sR isometry (see \cite[Proposition 5.20]{Bellaiche}).
Note that the nilpotent sR structure $(\widehat{M}^{q},\widehat{D}^{q},\widehat{g}^{q})$ is homogeneous with respect to the above dilations and that the corresponding sR distance is homogeneous of order $1$. Moreover, the growth vector of $\widehat{D}^{q}$ coincides with that of $D$ at $q$, and $\mathrm{Lie}(\widehat{X}^{q}_1,\ldots,\widehat{X}^{q}_m)$ is a nilpotent Lie algebra of step $r(q)$.

It follows from the definition of the sR metric that 
$$
\widehat{g}^{q} = \lim_{\varepsilon\rightarrow 0\atop \varepsilon\neq 0} \varepsilon^{-2} \delta_\varepsilon^* g
\qquad \textrm{and}\qquad
\widehat{g}^{q}_x(\widehat{X}^{q}(x),\widehat{Y}^{q}(x)) = g_{q}(X(q),Y(q))
$$
for every $x\in\R^n$, for all vector fields $X$ and $Y$ on $M$ that are smooth sections of $D$. 

Another useful geometric identification of $(\widehat{M}^{q},\widehat{D}^{q},\widehat{g}^{q})$ is the following.
Fixing a chart of privileged coordinates at $q$, let $\mathcal{G}_{q}$ be the (nilpotent) Lie group of diffeomorphims of $\R^n$ generated by $\exp(t\widehat{X}^{q}_i)$, for $t\in\R$ and $i=1,\ldots,m$. Its Lie algebra is 
$$
\mathfrak{g}_{q} = \mathrm{Lie}(\widehat{X}^{q}_1,\ldots,\widehat{X}^{q}_m)
= \bigoplus_{i=1}^{r(q)} \big( \widehat{D}^{q} \big)^i / \big( \widehat{D}^{q} \big)^{i-1} 
$$
it is nilpotent, graded, and generated by its first component $\widehat{D}^{q}$. In other words, $\mathcal{G}_{q}$ is a Carnot group (see \cite{Mo-02}).
Under the strong H\"ormander condition \eqref{Hormander_condition}, $\mathcal{G}_{q}$ acts transitively on $\R^n$.
Defining the isotropy group $H_{q}=\{\varphi\in \mathcal{G}_{q}\ \mid\ \varphi(0)=0\}$, of Lie algebra $\mathfrak{h}_{q}=\{Y\in \mathfrak{g}_{q}\ \mid\ Y(0)=0\}$, we identify $\widehat{M}^{q}$ to the homogeneous (coset) space $\mathcal{G}_{q}/H_{q}$. If $q$ is regular then $H_{q}=\{0\}$ and thus $\widehat{M}^{q}\simeq \mathcal{G}_{q}$ is a Carnot group endowed with a left-invariant sR structure.

\begin{remark}
Carnot groups are to sub-Riemannian geometry as Euclidean spaces are to Riemannian geometry.
However, there is a major difference, which is of particular importance here.
In Riemannian geometry, all tangent spaces are isometric, but this is not the case in sub-Riemannian geometry: given two points $q_1$ and $q_2$ of $M$, the nilpotentizations $(\widehat{M}^{q_1},\widehat{D}^{q_1},\widehat{g}^{q_1})$ and $(\widehat{M}^{q_2},\widehat{D}^{q_2},\widehat{g}^{q_2})$ of the sR structure respectively at $q_1$ and $q_2$ may not be sR-isometric, even though the growth vectors at $q_1$ and $q_2$ coincide.\footnote{Actually, the flags of two sR structures coincide at any point if and only if the sR structures are locally Lipschitz equivalent, meaning that the corresponding sR distances satisfy $c_1 d_2(q,q')\leq d_1(q,q')\leq c_2 d_2(q,q')$ for some uniform constants $c_1>0$ and $c_2>0$.
}
There are many algebraically non-isomorphic (and thus non-isometric) $n$-dimensional Carnot groups, and even uncountably many for $n\geq 5$ (due to moduli in their classification). 
We refer to \cite{AgrachevMarigo_2005,Marigo_2007} for a complete classification of rigid and semi-rigid Carnot algebras.

Note that, in dimension three, if the growth vector is $(2,3)$ 
then we have a unique model that is the Heisenberg flat case 
in the equivalence class of sR-isometric Carnot groups.
\end{remark}

\subsubsection{Nilpotentized sR Laplacian}
Let $q\in M$ be arbitrary. Associated with the sR structure $(\widehat{M}^{q},\widehat{D}^{q},\widehat{g}^{q})$, we define on $C^\infty(\widehat{M}^{q})$ the differential operator
\begin{equation}\label{def_triangle_nilp}
\widehat{\triangle}^{q} = \sum_{i=1}^m (\widehat{X}^{q}_i)^2 .
\end{equation}

\subsubsection{Nilpotentization of measures}
Let us define the nilpotentization of a smooth measure $\mu$ on $M$. Let $q\in M$ be arbitrary.
Using the bijective correspondence between smooth measures and densities, the measure $\mu$ induces a volume form that we consider at the point $q$. Then, the canonical isomorphism%
\footnote{Indeed, following \cite{ABGR}, considering a basis $(e_1,\ldots,e_n)$ of $T_{q}M$ that is adapted to the flag, that is, such that $e_i\in D_{q}^{w_i(q)}$, the wedge product $e_1\wedge\cdots\wedge e_n$ depends only on $e_i\mod D_{q}^{w_i(q)-1}$. This induces the canonical isomorphism.}
\begin{equation*}
\Lambda^n(T^\star_{q} M) \simeq \Lambda^n\bigg(  \bigoplus_{k=1}^{r(q)} D_{q}^k/D_{q}^{k-1} \bigg)^* 
\end{equation*}
induces a measure $\widehat{\mu}^{q}$ on $\widehat{M}^{q}$. Using a chart $\psi_{q}$ of privileged coordinates at $q$, and using the isometric representation $\widehat{M}^{q}\simeq\R^n$, the measure $\widehat{\mu}^{q}$ on $\R^n$ is given in the chart by
$$
\widehat{\mu}^{q} = \lim_{\varepsilon\rightarrow 0\atop \varepsilon\neq 0} \frac{1}{\vert\varepsilon\vert^{\mathcal{Q}(q)}} \delta_\varepsilon^* \mu 
$$
where the convergence is understood in the vague topology (i.e., the weak star topology of $C_c(M)'$, where $C_c(M)$ is the set of continuous functions on $M$ of compact support).
According to this definition, if $\mu$ and $\nu$ are two smooth measures on $M$, with $\mu=h\nu$, where $h$ is a positive smooth function on $M$, then $\widehat{\mu}^{q}=h(q)\widehat{\nu}^{q}$. Equivalently, this means that
\begin{equation}\label{density_nilp}
h(q) = \frac{d\mu}{d\nu}(q) = \frac{\widehat{\mu}^{q}}{\widehat{\nu}^{q}}.
\end{equation}
In particular, the nilpotentizations at $q$ of all smooth measures are proportional to the Lebesgue measure $m$ on $\widehat{M}^{q}\simeq\R^n$. Note that, if $q$ is regular, then $\widehat{\mu}^{q}$ is a left-invariant measure on the Carnot group $\widehat{M}^{q}$; in this case, $\widehat{M}^{q}$ is a nilpotent Lie group and thus is unimodular, and hence $\widehat{\mu}^{q}$ coincides with the Haar measure, up to scaling.

In passing, note that, applying \eqref{density_nilp} to the measure $\nu=\mathcal{H}_S$ that is the spherical Hausdorff measure and using the fact (proved in \cite{AgrachevBarilariBoscain_CV2012}) that $\widehat{\mathcal{H}_S}^{q}(\widehat{B}^{q})=2^{\mathcal{Q}(q)}$, we obtain that the density at $q$ of $\mu$ with respect to the spherical Hausdorff measure is $h(q) = \frac{d\mu}{d\mathcal{H}_S}(q) = \frac{\widehat{\mu}^{q}(\widehat{B}^{q})}{2^{\mathcal{Q}(q)}}$.

\begin{remark}\label{rem_nilpLap}
Let $q\in M$ be arbitrary, and let $\mu$ be an arbitrary smooth measure on $M$.
Endowing $\widehat{M}^{q}$ with the measure $\widehat{\mu}^{q}$, we claim that
\begin{equation}\label{div_zero}
\mathrm{div}_{\widehat{\mu}^{q}}(\widehat{X}^{q}_i)=0 \qquad \forall i\in\{1,\ldots,m\}.
\end{equation}
As a consequence, 
we have $(\widehat{X}^{q}_i)^*=-\widehat{X}^{q}_i$, where the transpose is considered in $L^2(\widehat{M}^{q},\widehat{\mu}^{q})$.  It follows that
$$
\widehat{\triangle}^{q} = \sum_{i=1}^m (\widehat{X}^{q}_i)^2 = \sum_{i=1}^m - (\widehat{X}^{q}_i)^*\widehat{X}^{q}_i  .
$$
Due to the cancellation of the divergence term, there are no terms of order one (compare with the general formula for a sR Laplacian, given, e.g., in \cite{CHT-SEDP}).

Let us prove \eqref{div_zero}. 
Following \cite[page 25]{Jean_2014}, we write $\widehat{X}^{q}_i(x) = \sum_{j=1}^n a_{ij}(x)\partial_{x_j}$ in a chart of privileged coordinates. By definition, $\partial_{x_j}$ is homogeneous of degree $-w_j(q)$, and since $\widehat{X}^{q}_i$ is homogeneous of degree $-1$, it follows that $a_{ij}$ is a homogenous polynomial of degree $w_j(q)-1$.
Since $x_k$ is of weight greater than or equal to $w_j(q)$ for $k\geq j$, $a_{ij}$ does not depend on variables $x_k$ for $k\geq j$.
It follows that $\partial_{x_j}a_{ij}=0$, and hence $\mathrm{div}_{\widehat{\mu}^{q}}(\widehat{X}^{q}_i)=0$ (recall that $\widehat{\mu}^{q}$ is equal, up to constant scaling, to the Lebesgue measure on $\R^n$). 
\end{remark}

\part{Proof of the limit in Theorem \ref{main_thm}}\label{part_limit}

This part is devoted to proving the following result (first term in the complete asymptotic expansion stated in Theorem \ref{main_thm}).

\begin{customthm}{B}\label{main_thm_weak}
{\it 
Let $q\in M$ be arbitrary (regular or not). 
Let $\psi_{q}:U\rightarrow V$ be a chart of privileged coordinates at $q$ such that $\psi_{q}(q)=0$, where $U$ is an open connected neighborhood of $q$ in $M$ and $V$ is an open neighborhood of $0$ in $\R^n$.
We assume that $X_0$ is a smooth section of $D$ over $M$. 
We also assume that $\sup_{q'\in U} r(q')<+\infty$.
Then, in the chart, we have
\begin{equation}\label{limit_main_thm_weak}
\lim_{\varepsilon\rightarrow 0\atop \varepsilon\neq 0}
\vert\varepsilon\vert^{\mathcal{Q}(q)}\, e_{\triangle,\mu}(\varepsilon^2 t,\delta_\varepsilon(x),\delta_\varepsilon(x')) \\
= \widehat{e}^{q}(t,x,x') 
\end{equation}
in $C^\infty((0,+\infty)\times V\times V)$ topology.

Moreover, if $q$ is regular, then the above convergence is locally uniform with respect to $q$, and the function $\widehat{e}^{q}$ depends smoothly (in $C^\infty$ topology) on $q$ in any open neighborhood of $q$ consisting of regular points. If the manifold $M$ is Whitney stratifiable with strata defined according to the sR flag (i.e., the growth vector $(n_1(q),\ldots, n_{r(q)}(q))$ is constant along each stratum) then the latter property is satisfied along strata.

When $X_0$ is a smooth section of $D^2$ over $M$, the above statement remains true, replacing $\widehat{\triangle}^{q}$ with $\widehat{\triangle}^{q}+\widehat{X}_0^{q}$ (see Remark \ref{rem_X0_length2}).
}
\end{customthm}

\paragraph{Sketch of proof.}
In few words, the proof goes as follows (see Section \ref{sec:proof_main_thm_prelim} for all details).

Assume first that $M=\R^n$. For every $i\in\{0,\ldots,m\}$, the vector field $X_i^\varepsilon = \varepsilon\delta_\varepsilon^*X_i$ converges to $\widehat{X}^q_i$ in $C^\infty$ topology as $\varepsilon\rightarrow 0$. Hence the operator 
$$
\triangle^\varepsilon = \varepsilon^2\delta_\varepsilon^*\triangle(\delta_\varepsilon)_*  
= \sum_{i=1}^m \left( X^\varepsilon_i \right)^2 + \varepsilon X^\varepsilon_0 - \varepsilon^2 \left( \delta_\varepsilon^*\mathbb{V}\right)
$$
converges to $\widehat{\triangle}^{q} = \sum_{i=1}^m \big( \widehat{X}^q_i\big)^2$ in $C^\infty$ topology. By the Trotter-Kato theorem (see \cite{EngelNagel,Pazy}), the corresponding heat kernel
\begin{equation}\label{rel_eeps_eX}
e^\varepsilon (t,x,x') = \vert\varepsilon\vert^{\mathcal{Q}(q)}\, e_{\triangle,\mu}(\varepsilon^2 t,\delta_\varepsilon(x),\delta_\varepsilon(x'))
\end{equation}
converges to $\widehat{e}^{q}$ in a weak topology, but actually convergence is true as well in $C^\infty((0,+\infty)\times\R^n\times\R^n$ topology because, by uniform local subellipticity of $(\triangle^\varepsilon)_{\varepsilon\in[-\varepsilon_0,\varepsilon_0]}$, the family $(e^\varepsilon)_{\varepsilon\in[-\varepsilon_0,\varepsilon_0]}$ is uniformly bounded in the Montel space $C^\infty$, for some $\varepsilon_0>0$ small enough.

On a general manifold, we follow the above argument in a local chart around $q$, extending the vector fields $X_i$ by $0$ outside of a neighborhood of $q$. The relation \eqref{rel_eeps_eX} is then not exactly true, but, thanks to the fact that the small-time asymptotics of hypoelliptic heat kernels is purely local (Kac's principle), the relation \eqref{rel_eeps_eX} remains true with a remainder term $\mathrm{O}(\vert\varepsilon\vert^\infty)$ as $\varepsilon\rightarrow 0$, and we conclude as well.

\medskip

As one can see, in the above argument, we only use \emph{local} results:
\begin{itemize}
\item local subellipticity estimates and local smoothing properties for hypoelliptic heat kernels, uniform with respect to some parameters: these are well known results, but for completeness (and in order to prepare global estimates), we give statements and proofs in Appendix \ref{sec:local_estimates};
\item local nature of the small-time asymptotics of hypoelliptic heat kernels (Kac's principle), uniform with respect to parameters: this is established in Theorem \ref{theo:localheat} in Section \ref{sec_localheat}.
\end{itemize}
In turn, we state in Theorem \ref{thm_CV_general1} (in Section \ref{sec_convergence_heat_kernels}) a general convergence result for hypoelliptic heat kernels depending on parameters: this is a singular perturbation theorem for hypoelliptic operators, generalizing existing singular perturbation results for elliptic operators.

\section{Some general facts for hypoelliptic heat kernels}\label{sec_facts}
\subsection{H\"ormander operators, semigroups and heat kernels}\label{sec_Hormander_kernel}
\subsubsection{Preliminary remarks on H\"ormander operators}\label{sec_rem_Horm}
In this section, we make some remarks on H\"ormander operators, which are useful in view of defining the domains of such operators and then show existence of semigroups, as done in Section \ref{sec_kernel}.
Let $\Omega$ be an open subset of $M$.

\paragraph{Integration by parts with a H\"ormander operator.}
We consider the differential operator $\triangle$ is defined by \eqref{def_DeltaX}. By integration by parts, we compute
\begin{multline}\label{IPP_Hormander}
\langle \triangle f,f\rangle_{L^2(\Omega,\mu)} 
= - \sum_{i=1}^m\Vert X_if\Vert_{L^2(\Omega,\mu)}^2 + \int_{\partial\Omega} f\, d\left( \iota_Y\mu\right) \\
+  \frac{1}{2} \int_\Omega f^2 \left( \sum_{i=1}^m \Big( X_i(\mathrm{div}_\mu(X_i))  +  (\mathrm{div}_\mu(X_i))^2 \Big) - \mathrm{div}_\mu(X_0) - 2\mathbb{V} \right) \, d\mu 
\end{multline}
with
\begin{equation}\label{def_Y}
Y = \sum_{i=1}^m \left( X_if - \frac{1}{2}  f\mathrm{div}_\mu(X_i) \right) X_i + \frac{1}{2}f X_0 .
\end{equation}
Of course, if $\partial\Omega=\emptyset$ then there is no boundary term. When $\partial\Omega\neq\emptyset$, the boundary term $\int_{\partial\Omega} f\, d\left( \iota_Y\mu\right)$ is equal to zero in the two following cases:
\begin{itemize}
\item Dirichlet case: $f=0$ along $\partial\Omega$;
\item Neumann case: $\iota_Y\mu=0$ along $\partial\Omega$.
\end{itemize}
In particular, we have defined here the Neumann boundary condition for the H\"ormander operator $\triangle$: the interior product of $\mu$ and of the vector field $Y$ (defined by \eqref{def_Y}) is zero along $\partial\Omega$. Note that
$$
Y = \nabla_gf + \frac{1}{2} f \left( X_0 - \sum_{i=1}^m \mathrm{div}_\mu(X_i) X_i \right)
$$
where $\nabla_gf$ is the horizontal gradient of $f$.

\paragraph{Symmetry properties of H\"ormander operators.}
Recalling that the differential operator $\triangle$ is defined by \eqref{def_DeltaX} and that $\trianglesR$ is defined by \eqref{def_DeltasR}, we have
$$
\triangle = \trianglesR + X'_0 - \mathbb{V} 
\qquad\textrm{with}\qquad X'_0=X_0 - \sum_{i=1}^m\mathrm{div}_\mu(X_i)X_i .
$$
Integrating by parts, we compute
\begin{multline*}
\langle \triangle u,v\rangle_{L^2(\Omega,\mu)} 
= -\int_\Omega g(\nabla_gu,\nabla_gv)\,d\mu - \int_\Omega u(X'_0v)\, d\mu - \int_\Omega uv ( \mathrm{div}_\mu(X_0')+ \mathbb{V}) \, d\mu \\
+ \int_{\partial\Omega} v\, d\left(\iota_{\nabla_gu} \mu\right) + \int_{\partial\Omega} uv\, d\left(\iota_{X'_0u} \mu\right)
\end{multline*}
and we infer the Green formula for H\"ormander operators:
\begin{multline*}
\langle \triangle u,v\rangle_{L^2(\Omega,\mu)} = \langle u,\triangle v\rangle_{L^2(\Omega,\mu)} 
+ \int_\Omega \big( v(X'_0u)\, d\mu - u(X'_0v)\big)\, d\mu \\
+ \int_{\partial\Omega} v\, d\left(\iota_{\nabla_gu} \mu\right) - \int_{\partial\Omega} u\, d\left(\iota_{\nabla_gv} \mu\right) \qquad \forall u,v\in C^\infty(M).
\end{multline*}
For the operator $\triangle$ to be symmetric on $C^\infty(\Omega)$, there are two necessary conditions:
\begin{itemize}
\item The term $\int_\Omega \big( v(X'_0u)\, d\mu - u(X'_0v)\big)\, d\mu$ must be zero, which is the case if and only if $X'_0=0$, i.e., $X_0 = \sum_{i=1}^m  \mathrm{div}_\mu(X_i) X_i$ on $\Omega$.
\item The boundary term must be zero. This is the case for Dirichlet boundary conditions. For Neumann boundary conditions, using the fact that $X'_0=0$ by the first item and thus that the vector field $Y$ defined by \eqref{def_Y} coincides with the horizontal gradient, we see then that the Neumann boundary condition for the H\"ormander operator coincides with the sR Neumann boundary condition.
\end{itemize}
Therefore, in the Dirichlet as in the Neumann case, the H\"ormander operator $\triangle$ is symmetric on $C^\infty(\Omega)$ if and only if $X_0 = \sum_{i=1}^m  \mathrm{div}_\mu(X_i) X_i$ on $\Omega$.
In this case, $\triangle = \triangle_{sR} - \mathbb{V}$ 
has selfadjoint extensions; moreover, if the manifold $\Omega$ endowed with the induced sR distance is complete then $\triangle$ is essentially selfadjoint on $C_c^\infty(\Omega)$ and thus has a unique selfadjoint extension. 

\subsubsection{Hypoelliptic semigroups and heat kernels}\label{sec_kernel}
We consider the operator $\triangle:D(\triangle)\rightarrow L^2(\Omega,\mu)$ defined 
on a domain $D(\triangle)$ that is assumed to be dense in $L^2(\Omega,\mu)$ and for which $(\triangle,D(\triangle))$ is closed. When $\partial\Omega\neq\emptyset$, the domain $D(\triangle)$ encodes some possible boundary conditions on $\partial\Omega$. 

According to Section \ref{sec_rem_Horm}, given a sufficiently regular function $f$, we speak of the Dirichlet boundary condition when $f=0$ along $\partial\Omega$, and of the Neumann boundary condition when $\iota_Y\mu=0$ (interior product of $\mu$ and $Y$) along $\partial\Omega$, where $Y$ is defined by \eqref{def_Y}.

Let $q\in M$ be arbitrary. Let $\mu$ be an arbitrary smooth measure on $M$.
We set $D(\widehat{\triangle}^{q}) = \{ f\in L^2(\widehat M^{q},\widehat\mu^{q})\ \mid\ \widehat{\triangle}^{q} f\in L^2(\widehat M^{q},\widehat\mu^{q}) \}$.
According to Section \ref{sec_rem_Horm} and Remark \ref{rem_nilpLap}, since $(\widehat M^{q},\hatdsR^{q})$ is complete (indeed, sR balls of small radius are compact, and $\widehat M^{q}$ is invariant under dilations), the operator $\widehat{\triangle}^{q}:D(\widehat{\triangle}^{q})\rightarrow L^2(\widehat M^{q},\widehat\mu^{q})$ is selfadjoint.

\begin{lemma}\label{lem_semigroup}
Under any of the following assumptions:
\begin{enumerate}[label=$\bf (\Alph*)$]
\item\label{caseA} $\triangle:D(\triangle)\rightarrow L^2(\Omega,\mu)$ is selfadjoint (see Section \ref{sec_rem_Horm});
\item\label{caseB} the functions $\mathrm{div}_\mu(X_0)$, $\mathrm{div}_\mu(X_i)$ and $X_i(\mathrm{div}_\mu(X_i))$, $i=1,\ldots,m$, are bounded on $\Omega$,
and we have Dirichlet or Neumann boundary conditions whenever $\partial\Omega\neq\emptyset$;
\end{enumerate}
the operator $(\triangle,D(\triangle))$ 
generates a strongly continuous semigroup $(e^{t\triangle})_{t\geq 0}$ on $L^2(\Omega,\mu)$.

The operator $(\widehat{\triangle}^{q},D(\widehat{\triangle}^{q}))$ is selfadjoint and generates a strongly continuous contraction semigroup $(e^{t\widehat{\triangle}^{q}})_{t\geq 0}$ on $L^2(\widehat M^{q},\widehat\mu^{q})$.
\end{lemma}

\begin{proof}
Under Assumption \ref{caseA}, it follows from Sections \ref{sec_sRLaplacian} and \ref{sec_rem_Horm} that $\triangle = \trianglesR - \mathbb{V}$ 
with $\mathbb{V}$ bounded below and thus there exists $C\geq 0$ such that $\triangle-C\,\mathrm{id}$ is dissipative.

Under Assumption \ref{caseB}, by integration by parts, we have the formula \eqref{IPP_Hormander}.
The integral on $\partial\Omega$ is zero in the case of Dirichlet or Neumann boundary conditions.
Hence there exists $C\geq 0$ (not depending on $f$) such that $\langle(\triangle-C\,\mathrm{id}) f,f\rangle_{L^2(\Omega,\mu)}\leq 0$ for every $f\in D(\triangle)$, and thus $\triangle-C\,\mathrm{id}$ is dissipative in $L^2(\Omega,\mu)$. The same result holds for its adjoint. 

Now, since the operator $\triangle-C\,\mathrm{id}$ in $L^2(\Omega,\mu)$ is closed and dissipative as well as its adjoint, it follows from the Lumer-Phillips theorem (see, e.g., \cite{EngelNagel,Pazy}) that it generates a strongly continuous contraction semigroup $(e^{t(\triangle-C\,\mathrm{id})})_{t\geq 0}$. 
Then, the operator $\triangle$ generates a strongly continuous semigroup $(e^{t\triangle})_{t\geq 0}$, and we have $e^{t\triangle} = e^{C t} e^{t(\triangle-C\,\mathrm{id})}$.

The operator $\widehat{\triangle}^{q}$ on $L^2(\widehat M^{q},\widehat\mu^{q})$ is closed, selfadjoint and dissipative, and thus it generates a strongly continuous contraction semigroup $(e^{t\widehat{\triangle}^{q}})_{t\geq 0}$.
\end{proof}

\paragraph{Hypoelliptic heat kernels.}
Under the assumptions done in Lemma \ref{lem_semigroup}, since $\mathrm{Lie}(D)=TM$ (which implies $\mathrm{Lie}(\widehat D^{q})=T\widehat M^{q}$)\footnote{Actually, the following weaker H\"ormander assumption is enough to ensure hypoellipticity of $\partial_t-\triangle$: $TM$ is spanned by the vector fields $(X_i)_{1\leq i\leq m}$, $([X_i,X_j])_{0\leq i,j\leq m}$, $([X_i,[X_j,X_k]])_{0\leq i,j,k\leq m}$, etc.}, both operators $\partial_t-\triangle$ and $\partial_t-\widehat\triangle^{q}$ are hypoelliptic and therefore the corresponding heat kernels exist and are smooth: 
given any smooth measure $\mu$ on $M$, we consider the heat kernel $e_{\triangle,\mu}$ defined on $(0,+\infty)\times\Omega\times\Omega$, associated with the operator $\triangle$ and with the measure $\mu$, and the heat kernel $\widehat{e}^{q} = e_{\widehat{\triangle}^{q},\widehat{\mu}^{q}}$ defined on $(0,+\infty)\times\widehat{M}^{q}\times\widehat{M}^{q}$, associated with the operator $\widehat{\triangle}^{q}$ and with the measure $\widehat{\mu}^{q}$ (see Appendix \ref{appendix_Schwartz} for reminders on heat kernels).
Smoothness follows from the fact that $e_{\triangle,\mu}$ is solution of $Pe_{\triangle,\mu}=0$, with $P=2\partial_t-(\triangle)_q-(\triangle^*)_{q'}$ that is hypoelliptic.

When $\mathbb{V}$ is bounded on $M$, it follows from the maximum principle for hypoelliptic operators (see \cite{Bony}) that $e_{\triangle,\mu}$ and $\widehat{e}^{q}$ are positive functions. Since $-\widehat\triangle^{q}$ is selfadjoint, $\widehat{e}^{q}$ is also symmetric.
Note also that, using the formulas in Appendix \ref{appendix_Schwartz}, we have the homogeneity property
\begin{equation}\label{homog_nilp}
\widehat{e}^{q}(t,x,x')
= \vert\varepsilon\vert^{\mathcal{Q}(q)}\, \widehat{e}^{q}(\varepsilon^2 t,\delta_\varepsilon(x),\delta_\varepsilon(x'))
\end{equation}
for all $(t,x,x')\in(0,+\infty)\times\R^n\times\R^n$ and for every $\varepsilon\neq 0$ (where we have identified $\widehat{M}^{q}\simeq\R^n$).

In Theorem \ref{main_thm} further, we will establish an asymptotic relationship between the heat kernel $e_{\triangle,\mu}$ and the nilpotentized heat kernel $\widehat{e}^{q}$.

\paragraph{Probabilistic interpretation.}
It can be noted that, when $\Omega=M$ and $\mathbb{V}=0$, the heat kernel $e_{\triangle,\mu}$ is the density of the probability law of the solution to the stochastic differential equation on $M$ in the Stratonovich sense
$$
dx_{t,q} = \sqrt{2}\sum_{i=1}^m X_i(x_{t,q})\circ dw_t^i + X_0(x_{t,q})\, dt 
$$
with $x_{0,q}=q$, where $(w_t^i)_{1\leq i\leq m}$ is a $m$-dimensional Brownian process realized as the coordinate process on $\{u\in C([0,1],\R^m\ \mid\ u(0)=0\}$ under the Wiener measure (see \cite{BenArous_AIF1989}): the solution to $\partial_t u-\triangle u=0$ for $t>0$, $u(0,q)=f(q)$ with $f\in L^2(\Omega,\mu)$, is then given by 
$$
u(t,q) = \int_\Omega e_{\triangle,\mu}(t,q,q')f(q')\, d\mu(q') = \mathbb{E} f(x_{t,q}) .
$$

\subsection{Two general results for parameter-dependent hypoelliptic heat kernels}\label{sec_two_general_results}

This section can be read independently of the rest.

Let $M$ be a smooth connected manifold.
Let $m\in\N^*$ and let $\mathcal{K}$ be a compact set. For every $\tau\in\mathcal{K}$, 
let $\mu^\tau$ be a smooth density on $M$,
let $X_0^\tau,X_1^\tau,\ldots,X_m^\tau$ be smooth vector fields on $M$ and let $\mathbb{V}^\tau$ be a smooth function on $M$, all of them depending continuously on $\tau$ in $C^\infty$ topology. We consider 
the second-order differential operator
\begin{equation*}
\triangle^\tau 
= \sum_{i=1}^m (X_i^\tau)^2 + X_0^\tau - \mathbb{V}^\tau . 
\end{equation*}
Throughout the section, we assume that the Lie algebra $\mathrm{Lie}(X_1^\tau,\ldots,X_m^\tau)$ generated by the vector fields is equal to $T_qM$ at any point $q\in M$, with a degree of nonholonomy that is uniform with respect to $\tau\in\mathcal{K}$ (\emph{uniform strong H\"ormander condition}).

\subsubsection{Local nature of the small-time asymptotics of hypoelliptic heat kernels}\label{sec_localheat}
Let $\Omega_1$ and $\Omega_2$ be two arbitrary open subsets of $M$, assumed to be manifolds with or without boundary. We still denote by $\mu^\tau$ the volume induced on each $\Omega_i$.

For $i=1,2$, we define the operator $\triangle^\tau_i$ on $L^2(\Omega_i,\mu^\tau)$ as follows. Let $D(\triangle^\tau_i)$ be a subset of $\{ u\in L^2(\Omega_i,\mu^\tau) \mid (\triangle^\tau u)_{\vert \Omega_i} \in L^2(\Omega_i,\mu^\tau) \}$, standing for the domain of $\triangle^\tau_i$ and encoding possible boundary conditions on $\partial \Omega_i$ if $\Omega_i$ has a nontrivial boundary. We now consider the operator $\triangle^\tau_i:D(\triangle^\tau_i)\rightarrow L^2(\Omega_i,\mu^\tau)$ defined by $\triangle^\tau_i u=(\triangle^\tau u)_{\vert \Omega_i}$ for every $u\in D(\triangle^\tau_i)$. 

In other words, we consider here the operator $\triangle^\tau$ on different subsets $\Omega_i$, with some boundary conditions. For instance, one can take $\Omega_1=M$ and $\Omega_2$ an open subset of $M$ with Dirichlet conditions on $\partial \Omega_2$.

We assume that $\triangle^\tau_i$ generates a strongly continuous semigroup $(e^{t\triangle^\tau_i})_{t\geq 0}$ on $L^2(\Omega_i,\mu^\tau)$, satisfying the uniform estimate
$$
\Vert e^{t\triangle^\tau_i} \Vert_{L(L^2(\Omega_i,\mu^\tau))} \leq \Cst\, e^{C t}\qquad \forall t\geq 0\qquad \forall \tau\in\mathcal{K} 
$$
for some $C>0$.
By an obvious (parameter-dependent) generalization of Lemma \ref{lem_semigroup}, we note that this is the case if $\mathbb{V}^\tau$ is uniformly bounded below on $\Omega_1\cup \Omega_2$ with respect to $\tau$ and if, for $i=1,2$:
\begin{enumerate}[label=$\bf (\Alph*_\tau)$]
\item\label{caseAtau} either $\triangle^\tau_i:D(\triangle^\tau_i)\rightarrow L^2(\Omega_i,\mu^\tau)$ is selfadjoint for every $\tau\in\mathcal{K}$,
\item\label{caseBtau} or the functions $\mathrm{div}_\mu^\tau(X_0^\tau)$, $\mathrm{div}_\mu^\tau(X_j^\tau)$ and $X_j^\tau(\mathrm{div}_\mu^\tau(X_j^\tau))$, $j=1,\ldots,m$, are bounded on $\Omega$, uniformly with respect to $\tau$, 
and we have Dirichlet or Neumann boundary conditions (see Lemma \ref{lem_semigroup}).
\end{enumerate}

Let $e^\tau_i(t,q,q')=e_{\triangle^\tau_i,\mu^\tau}(t,q,q')$ be its heat kernel, defined on $(0,+\infty) \times\Omega_i\times\Omega_i$.
Note that $e^\tau_i$ is symmetric whenever $\triangle^\tau_i$ is selfadjoint.

The following fact was noticed in \cite{Je-Sa-86}: extending the heat kernels by $0$ for $t<0$, by hypoellipticity of the operator $\partial_t-\triangle^\tau_i$ (under the Lie algebra generating assumption), it follows that $e^\tau_i(t,q,q')$ vanishes at infinite order as $t\rightarrow 0$ for fixed $q$ and $q'$ such that $q\neq q'$.
This observation inspired to us the result below.

Hereafter, given a function $e$ depending on three variables $(t,q,q')$, the notation $\partial_1$ (resp., $\partial_2$, $\partial_3$) denotes the partial derivative with respect to $t$ (resp., to $q$, to $q'$).

\begin{theorem}\label{theo:localheat}
For all $(k,\alpha,\beta)\in \N\times\N^d\times\N^d$, we have
$$
(\partial_1^k \partial_2^\alpha \partial_3^\beta e^\tau_1) (t,q,q') = (\partial_1^k \partial_2^\alpha \partial_3^\beta e^\tau_2) (t,q,q') + \mathrm{O}(t^\infty) 
$$
as $t\rightarrow 0$, $t>0$, uniformly with respect to $\tau\in\mathcal{K}$ and to $q$ and $q'$ varying in any compact subset of $\Omega_1\cap \Omega_2$.
This means that, for all $(k,\alpha,\beta)\in \N\times\N^d\times\N^d$, for every $t_1>0$, for every compact subset $K\subset\Omega_1\cap\Omega_2$, for every $N\in\N^*$, we have 
\begin{multline*}
\left\vert (\partial_1^k \partial_2^\alpha \partial_3^\beta (e^\tau_1 - e^\tau_2)) (t,q,q') \right\vert \leq \Cst(k,\alpha,\beta,t_1,K,N) t^N  \\
\qquad \forall t\in(0,t_1]\qquad\forall (q,q')\in K\times K\qquad\forall \tau\in\mathcal{K}.
\end{multline*}
\end{theorem}

This result reflects Kac's principle of ``not feeling the boundary", showing that the small-time asymptotic behavior of the heat kernel is purely local. Moreover, we establish here a uniform parameter-dependent version, which is possible thanks to the uniform subelliptic estimates obtained in Appendix \ref{sec:uniform_subelliptic_estimates}.

\begin{proof}
Let $\tau\in\mathcal{K}$ be arbitrary. Let $\Omega$ be an open subset of $\Omega_1\cap\Omega_2$.
We set $w^\tau(t,q,q')=e^\tau_1(t,q,q')-e^\tau_2(t,q,q')$, for all $t\in\R$ and $(q,q')\in\Omega\times\Omega$. The function $w^\tau$ is smooth on $(0,+\infty)\times\Omega\times\Omega$, and actually (extending by $0$ for $t<0$) we are going to prove that it is smooth on $\R\times\Omega\times\Omega$, with uniform estimates with respect to $\tau\in\mathcal{K}$.

On $\Omega\times\Omega$, we consider the differential operator $(\triangle^\tau)_q = \triangle^\tau\otimes\mathrm{id}$, meaning that given any smooth function $g$ on $\Omega\times\Omega$, the function $(\triangle^\tau)_q g$ designates the partial derivative, using the differential operator $\triangle^\tau$, of the function $g$, with respect to $q$. Accordingly, we consider the operator $(\triangle^\tau)_{q'}^* = \mathrm{id}\otimes (\triangle^\tau)^*$.

Noticing that the heat kernels have been extended by $0$ for $t<0$, both kernels $e^\tau_1$ and $e^\tau_2$ are solutions of the same differential equation $(\partial_t - (\triangle^\tau)_q)e = \delta_{(0,q')}(t,q)$ in the sense of distributions, for any fixed $q'\in\Omega$, where the distribution pairing is considered with respect to the measure $dt\times d\mu^\tau(q)$ on $\R\times\Omega$. Hence $(\partial_t-(\triangle^\tau)_q)w^\tau=0$ on $\R\times\Omega\times\Omega$.
Using that $e_{\triangle^\tau,\mu^\tau}(t,q,q') = e_{\triangle^\tau_*,\mu^\tau}(t,q',q)$, both $e^\tau_1$ and $e^\tau_2$ are also solutions of $(\partial_t - (\triangle^\tau)_{q'}^*)e = \delta_{(0,q)}(t,q')$ in the sense of distributions, for any fixed $q\in\Omega$. Hence $(\partial_t-(\triangle^\tau)_{q'}^*)w^\tau=0$ on $\R\times\Omega\times\Omega$.
Setting $P_\tau = (\triangle^\tau)_q + (\triangle^\tau)_{q'}^* - 2  \partial_t$, we infer that $P_\tau w^\tau=0$ on $\R\times\Omega\times\Omega$. 
At this step, for any $\tau$ fixed, we infer by hypoellipticity of $P_\tau$ that $w^\tau$ is smooth, and since $w^\tau$ vanishes for $t<0$, it follows that $w^\tau$ is flat at $t=0$. This gives the result, for $\tau$ fixed.

But in order to ensure uniform estimates with respect to $\tau$, we have to elaborate further arguments.
In order to use the uniform local subelliptic estimates \eqref{thm_subelliptic_uniform_cor} established in Section \ref{sec:uniform_local_subelliptic_estimates}, as an initialization, we need to prove that $w^\tau$ is bounded, uniformly with respect to $\tau$, for some weak enough Sobolev norm.
To this aim, let us first establish a rough norm estimate, valid for both heat kernels, and uniform with respect to $\tau$.

\begin{lemma}\label{lem_rough_uniform}
For every $t_1>0$, for every open subset $V\subset\Omega$ of compact closure, there exists $p\in\N^*$ such that
$$
\Vert e^\tau_i(\cdot,\cdot,\cdot) \Vert_{L^\infty_t(0,t_1)\times H^{-p}_q(V)\times L^2_{q'}(V)} \leq \Cst(t_1,V)
\qquad\forall \tau\in\mathcal{K}\qquad \forall i\in\{1,2\}.
$$
\end{lemma}

In the norm above, and in the proof hereafter, the Sobolev spaces are considered with respect to the density $\mu^\tau$.

\begin{proof}[Proof of Lemma \ref{lem_rough_uniform}.]
Let $t_1>0$ be arbitrary. 
As in the proofs of Lemma \ref{lem_semigroup} (Section \ref{sec_kernel}) and of Corollary \ref{cor_heat_uniform} (Appendix \ref{sec:localregheat}), we first note that $\Vert e^{t\triangle^\tau_i}\Vert_{L(L^2(\Omega_i))}\leq \Cst(t_1)$, for every $t\in[0,t_1]$.
Besides, we set $\Lambda_a = a (\mathrm{id}-\triangle_R)^{1/2}$, where $\triangle_R$ is any second-order elliptic operator on $\Omega_i$ (for instance, a Riemannian Laplacian if $M$ is Riemannian) and $a$ is a smooth positive function on $\Omega_i$ chosen such that there exists $p\in\N^*$ large enough so that  $\Lambda_a^{-p}$ is Hilbert-Schmidt (as an operator on $L^2(\Omega_i)$), i.e., $\left\Vert \Lambda_a^{-p} \right\Vert_{HS}<+\infty$
(see also the proof of Corollary \ref{cor_heat_uniform} in Appendix \ref{sec:localregheat} for the existence of such an integer $p$).
It follows that the operator $\Lambda_a^{-p} e^{t\triangle^\tau_i}$ on $L^2(\Omega_i)$ is Hilbert-Schmidt, and its Hilbert-Schmidt norm is bounded uniformly with respect to $\tau\in\mathcal{K}$ and to $t\in[0,t_1]$. Since $\Vert \Lambda_a^{-p} e^{t\triangle^\tau_i}\Vert_{HS} = \Vert (\Lambda_a)_q^{-p} e^\tau_i(t,\cdot,\cdot)\Vert_{L^2_q(\Omega_i)\times L^2_{q'}(\Omega_i)}$, the conclusion follows.
\end{proof}

Now, let $\zeta$ and $\zeta'$ be arbitrary smooth functions compactly supported in $(-t_1,t_1)\times\Omega\times\Omega$, with $\zeta'=1$ on the support of $\zeta$.
From Lemma \ref{lem_rough_uniform}, there exists $s<0$ such that $\Vert \zeta' w^\tau\Vert_{H^s((-t_1,t_1)\times\Omega\times\Omega)}\leq \Cst(t_1,\zeta')$.
Applying Theorem \ref{thm_subelliptic_uniform} to the family of operators $P_\tau$ (in particular, applying to $w^\tau$ the uniform estimates \eqref{thm_subelliptic_uniform_cor} that follow from this theorem), we infer that, for every $k\in\N$, the norm $\Vert \zeta w^\tau\Vert_{H^{s+k\sigma}((-t_1,t_1)\times\Omega\times\Omega)}$ is uniformly bounded with respect to $\tau\in\mathcal{K}$.
Using Sobolev embeddings, the theorem follows.
\end{proof}

\begin{remark}
Note that a quite similar result has been established in \cite{Hsu1995}, without parameter dependence and under completeness assumptions.
%
%
\end{remark}


\subsubsection{A general convergence result for hypoelliptic heat kernels}\label{sec_convergence_heat_kernels}
We keep the notations and assumptions done at the beginning of Section \ref{sec_facts}.
We assume that $\triangle^\tau$ generates a strongly continuous semigroup $(e^{t\triangle^\tau})_{t\geq 0}$ on $L^2(\Omega,\mu^\tau)$, satisfying uniform estimate
$$
\Vert e^{t\triangle^\tau} \Vert_{L(L^2(\Omega,\mu^\tau))} \leq \Cst\, e^{C t}\qquad \forall t\geq 0\qquad \forall \tau\in\mathcal{K} 
$$
for some $C>0$.
Like in Section \ref{sec_localheat}, we note that this is the case if $\mathbb{V}^\tau$ is uniformly bounded below on $M$ and if either \ref{caseAtau} or \ref{caseBtau} is satisfied.
We denote by $e^\tau=e_{\triangle^\tau,\mu^\tau}$ the associated heat kernel, defined on $(0,+\infty)\times\Omega\times\Omega$.

\begin{theorem}\label{thm_CV_general1}
The heat kernel $e^\tau$ is smooth on $(0,+\infty)\times\Omega\times\Omega$, for every $\tau\in\mathcal{K}$, and depends continuously on $\tau\in\mathcal{K}$ in $C^\infty((0,+\infty)\times\Omega\times\Omega)$ topology.
\end{theorem}

\begin{proof}
Let $\tau_0\in\mathcal{K}$ be arbitrary.
The differential operator $\triangle^{\tau_0}$ is the limit of $\triangle^\tau$ in $C^\infty$ topology as $\tau\rightarrow\tau_0$, meaning that $\triangle^\tau f\rightarrow \triangle^{\tau_0} f$ as $\tau\rightarrow\tau_0$ uniformly on any compact subset of $M$, for every smooth function $f$ on $M$. 

By the Trotter-Kato theorem (see, e.g., \cite[Chapter III]{EngelNagel} or \cite[Chapter 3]{Pazy}), $e^{t\triangle^\tau} f \rightarrow e^{t \triangle^{\tau_0}} f$ in $L^2(\Omega,\mu^\tau)$ as $\tau\rightarrow\tau_0$, for every $t\geq 0$ and every smooth function $f$ on $\Omega$ with compact support, and the convergence is uniform with respect to $t$ on $[0,t_1]$, for every $t_1>0$.
Taking the Schwartz kernels (see Appendix \ref{appendix_Schwartz}), it follows that, given any $0<t_0<t_1$ and any compact subset $K$ of $\Omega$, $e^\tau$ converges to $e^{\tau_0}$ in $C^{-\infty}([t_0,t_1]\times K\times K)$ as $\tau\rightarrow\tau_0$ for the weak-star topology. Here, $C^{-\infty}([t_0,t_1]\times K\times K)$ is the topological dual of the Fr\'echet Montel space $C^\infty([t_0,t_1]\times K\times K)$. 

By Corollary \ref{cor_heat_uniform} in Appendix \ref{sec:localregheat}, applied with $L^\tau = \triangle^\tau$, the family $(e^\tau)_{\tau\in\mathcal{K}}$ is uniformly bounded in $C^{\infty}([t_0,t_1]\times K\times K)$. Therefore, thanks to the Heine-Borel property, we conclude that $e^\tau$ converges to $e^{\tau_0}$ in the Fr\'echet Montel space $C^\infty((0,+\infty)\times\Omega\times\Omega)$ as $\tau\rightarrow\tau_0$.
\end{proof}

\begin{remark}
We note that $\triangle^\tau$ is a singular perturbation of $\triangle^0$. One can find in \cite{Lions_perturbations} (see also \cite{Huet}) a number of results on singular perturbations of elliptic operators, i.e., when $\triangle^0$ is an elliptic operator. Here, our results can be seen as some singular perturbations of hypoelliptic operators.

For example, Theorem \ref{thm_CV_general1} can be applied to the situation where $\triangle^\tau = \triangle^0 + \tau \triangle_R$, with $\triangle^0$ being hypoelliptic and $\triangle_R$ being a Riemannian Laplacian if $M$ is Riemannian: this is an elliptic perturbation of a hypoelliptic operator.
We thus recover results established in \cite[Theorem 7.2]{Rumin} (see also \cite{Ge}) in the particular case where $\triangle^0$ is a contact sR Laplacian.

Note also that, when the considered operators are selfadjoint and of compact resolvent, using the max-min principle, our results imply convergence of the spectrum of $\triangle^\tau$ (eigenvalues and eigenfunctions) to that of $\triangle^0$ (as in \cite{Fukaya,Ge,Rumin}). We do not give details. This convergence is of course not uniform in general because the leading term in the short-time asymptotics of heat kernels may differ: for instance when $\triangle^0$ is a 3D contact sub-Riemannian Laplacian then the short-time asymptotics is like $1/t^2$, whereas for $\tau\neq 0$, assuming that $\triangle^\tau = \triangle^0 + \tau \triangle_R$ as above, the short-time asymptotics is like $1/t^{3/2}$ (asymptotics in the Riemannian case).
%
\end{remark}


\section{Proof of the limit in Theorem \ref{main_thm}}\label{sec:proof_main_thm_prelim}
In this section, we prove Theorem \ref{main_thm_weak}, which is Theorem \ref{main_thm} at the order zero.

\subsection{Preliminaries}
Throughout, we assume that $X_0$ is a smooth section of $D$ over $M$. Each time this is required, we will indicate the modifications that must be done when $X_0$ is a smooth section of $D^2$ over $M$.

Let $\varepsilon_0>0$ be small enough such that $\delta_\varepsilon(V)\subset V$ for every $\varepsilon\in[-\varepsilon_0,\varepsilon_0]$. 
We first extend the vector fields $(\psi_{q})_*X_i$ (which are defined in the neighborhood $V$ given by the chart) to $\R^n$.
Let $W_1$ and $W_2$ be open subsets of $\R^n$ of compact closure such that $\overline W_1\subset W_2\subset\overline W_2\subset V$ and such that $\delta_\varepsilon(W_1)\subset W_1$ and $\delta_\varepsilon(W_2)\subset W_2$ for every $\varepsilon\in[-\varepsilon_0,\varepsilon_0]$.
Let $\chi$ be a smooth function of compact support on $\R^n$, such that $0\leq\chi\leq 1$, $\chi(x)=1$ if $x\in \overline W_1$ and $\chi(x)=0$ if $x\in\R^n\setminus W_2$. 

Hereafter, we will use the measure $\widehat{\mu}^{q}$ on $\R^n$, which coincides, up to a constant scaling, with the Lebesgue measure on $\R^n$.

\paragraph{Definition of $\widetilde\triangle$ (local version of $\triangle$ in the chart).}
We define
$$
Y_i = \chi (\psi_{q})_* X_i  , \qquad i=0,\ldots,m, 
$$
so that $Y_i = (\psi_{q})_* X_i$ on $W_1$ and $Y_i = 0$ on $\R^n\setminus W_2$. Similarly, we define the function $\mathtt{v}$ on $\R^n$ by $\mathtt{v}=\chi (\psi_{q})_* \mathbb{V}$,
and the measure $\nu$ on $\R^n$ by $\langle\nu,f\rangle=\langle\psi_{q})_*\mu,\chi f\rangle$ for every $f\in C^0(\R^n)$.
Setting $Y=(Y_0,Y_1,\ldots,Y_m)$, we consider on $C^\infty(\R^n)$ the differential operator
\begin{equation*}
\widetilde\triangle = \sum_{i=1}^m Y_i^2 + Y_0 - \mathtt{v} . 
\end{equation*}
The operator $\widetilde\triangle$ (resp., the measure $\nu$) is the extension to $\R^n$ (by $0$) of the local version of $\triangle$ (resp., of $\mu$) in the chart. As we are going to see, in the proof of Theorem \ref{main_thm_weak}, the way we extend does not have any impact on the local asymptotics of the heat kernel, thanks to the localization result stated in Theorem \ref{theo:localheat} (Section \ref{sec_localheat}). 

Since the vector fields $Y_i$ are of compact support, setting $D(\widetilde\triangle) = \{ f\in L^2(\R^n,\nu) \ \mid\ \widetilde\triangle f \in L^2(\R^n,\nu) \}$, it follows from Lemma \ref{lem_semigroup} that the operator $(\widetilde\triangle,D(\widetilde\triangle))$ generates a strongly continuous semigroup $(e^{t\widetilde\triangle})_{t\geq 0}$ on $L^2(\R^n,\nu)$.
By hypoellipticity (see Corollary \ref{cor_heat_uniform}), the Schwartz kernel of $e^{t\widetilde\triangle}$, restricted to $(0,+\infty)\times W_1\times W_1\rightarrow(0,+\infty)$, has a continuous density with respect to $\nu$, which is the smooth function
\begin{equation*}
\widetilde e = e_{\widetilde\triangle,\nu}:  (0,+\infty)\times W_1\times W_1\rightarrow(0,+\infty) .
\end{equation*}

\subsection{Definition of the vector fields $Y_i^\varepsilon$}
For every $\varepsilon\in[-\varepsilon_0,\varepsilon_0]\setminus\{0\}$, we set
\begin{equation}\label{defYieps}
\nu^\varepsilon = \frac{1}{\vert\varepsilon\vert^{\mathcal{Q}(q)}}\delta_\varepsilon^*\nu,\qquad
Y^\varepsilon_i = \varepsilon\delta_\varepsilon^*Y_i,\qquad i=0,\ldots,m .
\end{equation}
Using that $\delta_\varepsilon=\delta_{\frac{\varepsilon}{\lambda}}\delta_\lambda=\delta_\lambda\delta_{\frac{\varepsilon}{\lambda}}$, we observe that 
\begin{equation}\label{homog_deltaY}
Y_i^\varepsilon = \lambda\delta_\lambda^* Y_i^{\varepsilon/\lambda},\qquad i=0,\ldots,m, \qquad \forall\lambda>0\qquad\forall\varepsilon\in[-\lambda\varepsilon_0,\lambda\varepsilon_0]\setminus\{0\}.
\end{equation}

When $X_0$ is a smooth section of $D^2$ over $M$, we modify the definition of $Y^\varepsilon_0$ by setting $Y^\varepsilon_0 = \varepsilon^2\delta_\varepsilon^*Y_0$, and we have the homogeneity property $Y_0^\varepsilon = \lambda^2\delta_\lambda^* Y_0^{\varepsilon/\lambda}$.

Note that $Y_i^\varepsilon$ is nontrivial on $\delta_\varepsilon^{-1}(W_1)=\delta_{1/\varepsilon}(W_1)$ which is a neighborhood of $0$ increasing to $\R^n$ as $\varepsilon\rightarrow 0$. 
For every $i\in\{0,\ldots,m\}$, $Y_i^\varepsilon$ converges to $\widehat{X}_i^{q}$ in $C^\infty(\R^n,\R^n)$ as $\varepsilon\rightarrow 0$ (see Section \ref{sec:nilp} and Remark \ref{rem_X0_length2}).
Actually, using \eqref{expansion_Xeps}, we have the expansion
\begin{equation}\label{expansion_Yieps}
Y_i^\varepsilon = \varepsilon \delta_\varepsilon^* Y_i = \widehat{X}_i^{q} + \varepsilon Y_i^{(0)} + \varepsilon^2 Y_i^{(1)} + \cdots + \varepsilon^N Y_i^{(N-1)} + \mathrm{o}\big(\vert\varepsilon\vert^N\big)
\end{equation}
in $C^\infty$ topology, where $Y_i^{(k)}$ is polynomial and homogeneous of degree $k$ (with respect to dilations), and $\widehat{X}_i^{q} = Y_i^{(-1)}$, i.e., setting $Y^0_i=\widehat{X}_i^{q}$ for $\varepsilon=0$, $Y^\varepsilon_i$ depends smoothly on $\varepsilon$ in $C^\infty$ topology.

Since $Y_i^\varepsilon$ converges to $\widehat{X}_i^{q}$, as well as all its derivatives, on any compact, since the $m$-tuple $(\widehat{X}_1^{q},\ldots,\widehat{X}_m^{q})$ satisfies the H\"ormander condition, using that $[Y_i^\varepsilon,Y_j^\varepsilon] = \varepsilon^2 \delta_\varepsilon^* [Y_i,Y_j]$, it is clear that the $m$-tuple $(Y_1^\varepsilon,\ldots,Y_m^\varepsilon)$ satisfies the uniform strong H\"ormander condition (as defined in Section \ref{sec:localregheat}) on $W_1$, for $\varepsilon\in[-\varepsilon_0,\varepsilon_0]$, provided that $\varepsilon_0$ be small enough.

Moreover, we have $\nu^\varepsilon\rightarrow \widehat{\mu}^{q}$ for the vague topology as $\varepsilon\rightarrow 0$. Actually, the density of $\nu^\varepsilon$ with respect to the Lebesgue measure of $\R^n$ converges in $C^\infty$ topology to the density of $\widehat{\mu}^{q}$ with respect to the Lebesgue measure of $\R^n$ (which is constant).

\subsection{Definition of the operator $\triangle^\varepsilon$}
\paragraph{Differential operator $\triangle^\varepsilon$.}
For every $\varepsilon\in[-\varepsilon_0,\varepsilon_0]\setminus\{0\}$ we define on $C^\infty(\R^n)$ the differential operator
\begin{equation}\label{def_triangleepsilon}
\triangle^\varepsilon 
= \varepsilon^2 \delta_\varepsilon^* \widetilde\triangle (\delta_\varepsilon)_* .
\end{equation}
Using \eqref{homog_deltaY}, we have the homogeneity property
\begin{equation}\label{homog_deltaeps}
\triangle^\varepsilon = \lambda^2 \delta_\lambda^* \triangle^{\frac{\varepsilon}{\lambda}} (\delta_\lambda)_* \qquad \forall\lambda>0\qquad\forall\varepsilon\in[-\lambda\varepsilon_0,\lambda\varepsilon_0]\setminus\{0\}.
\end{equation}
Using 
Appendix \ref{appendix_Schwartz}, we have
\begin{equation}\label{formula_Deltaeps}
\triangle^\varepsilon =
\sum_{i=1}^m \left( \varepsilon\delta_\varepsilon^*Y_i \right)^2 + \varepsilon^2\delta_\varepsilon^* Y_0 - \varepsilon^2 \left(\delta_\varepsilon^*\mathtt{v} \right) 
= \sum_{i=1}^m \left( Y^\varepsilon_i \right)^2 + \varepsilon Y^\varepsilon_0 - \varepsilon^2 \left( \delta_\varepsilon^*\mathtt{v}\right) . 
\end{equation}

When $X_0$ is a smooth section of $D^2$ over $M$, the definition of $\triangle^\varepsilon$ is modified as follows:
$$
\triangle^\varepsilon 
= \sum_{i=1}^m \left( Y^\varepsilon_i \right)^2 + Y^\varepsilon_0 - \varepsilon^2 \left( \delta_\varepsilon^*\mathtt{v}\right) . 
$$

\paragraph{Convergence of $\triangle^\varepsilon$ to $\widehat{\triangle}^{q}$.}
Since $\varepsilon Y^\varepsilon_0 \rightarrow 0$ in $C^\infty(\R^n,\R^n)$, the differential operator $\widehat{\triangle}^{q}$ defined by \eqref{def_triangle_nilp} is the limit of $\triangle^\varepsilon$ in $C^\infty$ topology as $\varepsilon\rightarrow 0$, meaning that $\triangle^\varepsilon f\rightarrow \widehat{\triangle}^{q} f$ in $C^\infty(\R^n)$
as $\varepsilon\rightarrow 0$, for every $f\in C^\infty(\R^n)$.
Defining $\triangle^0=\widehat{\triangle}^{q}$ for $\varepsilon=0$, $\triangle^\varepsilon$ depends smoothly on $\varepsilon$ in $C^\infty$ topology.

We could give an asymptotic expansion of $\triangle^\varepsilon$ in $C^\infty$ topology, as we will do further in Section \ref{sec_def_deltaepsgam} for an appropriate modification of $\triangle^\varepsilon$, but we do not give it because it will not be useful. Indeed, we will see further that the $C^\infty$ topology is not strong enough to establish the complete asymptotic expansion stated in Theorem \ref{main_thm}. Anyway, the limit in $C^\infty$ topology suffices to establish Theorem \ref{main_thm_weak}. 

\medskip

When $X_0$ is a smooth section of $D^2$ over $M$, we have
$Y_0^\varepsilon = \varepsilon^2\delta_\varepsilon^* Y_0 = \widehat{X}_0^{q} + \varepsilon Y_0^{(-1)} +  \mathrm{o}(\vert\varepsilon\vert)$
where $\widehat{X}_0^{q}$ is homogeneous of order $-2$ (see Remark \ref{rem_X0_length2} for the details). We obtain in this case that $\triangle^\varepsilon\rightarrow \widehat{\triangle}^{q} +\widehat{X}_0^{q}$ in $C^\infty$ topology as $\varepsilon\rightarrow 0$. 

\paragraph{Semigroup generated by $\triangle^\varepsilon$.}
Setting $D(\triangle^\varepsilon) = \{ f\in L^2(\R^n,\nu^\varepsilon) \ \mid\ \triangle^\varepsilon f \in L^2(\R^n,\nu^\varepsilon) \}$, using \eqref{def_triangleepsilon}, the operator $(\triangle^\varepsilon,D(\triangle^\varepsilon))$ generates a strongly continuous semigroup $(e^{t\triangle^\varepsilon})_{t\geq 0}$ on $L^2(\R^n,\nu^\varepsilon)$, satisfying
$$
\delta_\varepsilon^*e^{\varepsilon^2 t\widetilde\triangle}(\delta_\varepsilon)_* = e^{t\triangle^\varepsilon}\qquad \forall t\geq 0\qquad \forall \varepsilon\in[-\varepsilon_0,\varepsilon_0]\setminus\{0\} .
$$
Since $\Vert e^{t\widetilde\triangle}\Vert_{L(L^2(\R^n,\nu))} \leq \Cst\, e^{C t}$ for every $t\geq 0$, for some $C>0$, it follows that the semigroup $(e^{t\triangle^\varepsilon})_{t\geq 0}$ satisfies the uniform estimate $\Vert e^{t\triangle^\varepsilon}\Vert_{L(L^2(\R^n,\nu^\varepsilon))} \leq \Cst\, e^{C\varepsilon^2 t}$.

\paragraph{Heat kernel $e^\varepsilon$ of $\triangle^\varepsilon$.}
By hypoellipticity (see Corollary \ref{cor_heat_uniform}), the Schwartz kernel of $e^{t\triangle^\varepsilon}$, restricted to $(0,+\infty)\times W_1\times W_1$, has a continuous density with respect to $\nu^\varepsilon$, which is the smooth function
\begin{equation*}
e^\varepsilon = e_{\triangle^\varepsilon,\nu^\varepsilon}: (0,+\infty)\times W_1\times W_1\rightarrow(0,+\infty) .
\end{equation*}
Using \eqref{def_triangleepsilon} and the formulas \eqref{formulas_kernel} of Appendix \ref{appendix_Schwartz}, we have
\begin{multline}\label{relation_eY_eepsilon}
e^\varepsilon(t,x,x') = \vert \varepsilon\vert^{\mathcal{Q}(q)} \widetilde e(\varepsilon^2t,\delta_\varepsilon(x),\delta_\varepsilon(x'))  \\
\forall \varepsilon\in[-\varepsilon_0,\varepsilon_0]\setminus\{0\} \qquad \forall (t,x,x')\in(0,+\infty)\times W_1\times W_1 .
\end{multline}

\paragraph{Convergence of $e^\varepsilon$ to $\widehat{e}^{q}$.}
Recall that the nilpotentized heat kernel is the smooth function $\widehat{e}^{q} = e_{\widehat{\triangle}^{q},\widehat{\mu}^{q}}:  (0,+\infty)\times\R^n\times\R^n\rightarrow(0,+\infty)$, defined as the continuous density with respect to $\widehat{\mu}^{q}$ of the Schwartz kernel of $e^{t\widehat{\triangle}^{q}}$.

Applying the general convergence result stated in Theorem \ref{thm_CV_general1} (in Section \ref{sec_convergence_heat_kernels}) with $\mathcal{K} = [-\varepsilon_0,\varepsilon_0]$, $\tau=\varepsilon$, $\Omega=W_1$, $\mu^\tau=\nu^\varepsilon$ and $L^\tau = \triangle^\varepsilon$, we obtain that 
\begin{equation}\label{CVeepsehat}
e^\varepsilon \underset{\varepsilon\rightarrow 0}{\longrightarrow}  \widehat{e}^{q} \qquad\textrm{in}\quad C^\infty((0,+\infty)\times W_1\times W_1) .
\end{equation}
When $X_0$ is a smooth section of $D^2$ over $M$, the result remains true provided that $\widehat{\triangle}^{q}$ be replaced with $\widehat{\triangle}^{q} +\widehat{X}_0^{q}$.

\begin{remark}
The above argument yields a convergence that is much stronger than the convergence on semigroups provided by the Trotter-Kato theorem (which is only pointwise).
We do not know if this could have been established by general results on analytic semigroups. Indeed, although the strongly continuous contraction semigroup $(e^{t\widehat{\triangle}^{q}})_{t\geq 0}$ is analytic of angle $\pi/2$ (this follows, e.g., from \cite[Chapter II, Corollary 4.7]{EngelNagel}, because the operator $\widehat{\triangle}^{q}$ on $L^2(\R^n)$ is nonpositive selfadjoint and thus has a real nonpositive spectrum), given any $\varepsilon>0$, we do not know if the strongly continuous semigroup $(e^{t\triangle^\varepsilon})_{t\geq 0}$ on $L^2(\R^n)$ is analytic with an angle that would be uniform with respect to $\varepsilon$ in general (unless, of course, we are in the case where $\triangle^\varepsilon$ is selfadjoint). Actually, there are hints (see \cite{EckmannHairer}) showing that the operator $\triangle^\varepsilon$ may fail to be uniformly sectorial.\footnote{We thank Martin Hairer for a discussion on this subject.}
Note also that $\triangle^\varepsilon-\widehat{\triangle}^{q}$ \emph{is not} $\widehat{\triangle}^{q}$-bounded in general and that $\widehat{\triangle}^{q}$ is not an elliptic operator.
This is why, instead of using classical integral representations of analytic semigroups, we used the fact (proved in Section \ref{sec:localregheat}) that $e^{t\triangle^\varepsilon}$ is uniformly locally smoothing. 
\end{remark}

\subsection{End of the proof of Theorem \ref{main_thm_weak}}\label{sec_proof_weak_main_thm}
We already know that $e^\varepsilon$ is related to $\widetilde e$ by the formula \eqref{relation_eY_eepsilon}, which gives as well
\begin{equation}\label{relation_eY_eepsilon_der}
(\partial_1^k \partial_2^\alpha \partial_3^\beta \, e_{\triangle^\varepsilon,\nu^\varepsilon}) (t,x,x') = \vert\varepsilon\vert^{\mathcal{Q}(q)+2k} \varepsilon^{\sum_{i=1}^n (\alpha_i+\beta_i)w_i(q)}  (\partial_1^k \partial_2^\alpha \partial_3^\beta \, e_{\widetilde\triangle,\nu}) (\varepsilon^2 t,\delta_\varepsilon(x),\delta_\varepsilon(x')) 
\end{equation}
for every $\varepsilon\in[-\varepsilon_0,\varepsilon_0]\setminus\{0\}$, for all $(t,x,x')\in(0,+\infty)\times W_1\times W_1$ and for all $(k,\alpha,\beta)\in\N\times\N^d\times\N^d$, where we have set $\alpha=(\alpha_1,\ldots,\alpha_n)$ and $\beta=(\beta_1,\ldots,\beta_n)$.
Let us now relate $\widetilde e$ with $e_{\triangle,\mu}$. This is done thanks to the localization result stated in Theorem \ref{theo:localheat} (Section \ref{sec_localheat}).
Recalling that $\widetilde\triangle$ coincides with $\triangle$ in the chart, Theorem \ref{theo:localheat} gives, in the chart,
\begin{equation}\label{rel_XY}
e_{\triangle,\mu} (\varepsilon^2 t,\delta_\varepsilon(x),\delta_\varepsilon(x')) = 
e_{\widetilde\triangle,\nu} (\varepsilon^2 t,\delta_\varepsilon(x),\delta_\varepsilon(x')) + \mathrm{O}(\vert\varepsilon\vert^\infty) 
\end{equation}
in $C^\infty$ topology.
The limit \eqref{limit_main_thm_weak} in $C^\infty$ topology then follows from \eqref{CVeepsehat}, \eqref{relation_eY_eepsilon_der} and \eqref{rel_XY}.

\paragraph{Case $q$ regular.}
Let us assume that $q$ is regular, meaning that there exists an open subset $U$ of $M$ on which the flag is regular. Let us make vary $q$ in $U$. 
We first remark that it is possible to choose, at the beginning of the proof, a chart $\psi_{q}$ depending smoothly on $q$: for instance, one may use the map \eqref{example_privileged} that is obtained with a frame of vector fields $Z_i^{q}$ adapted to the flag and depending smoothly on $q$.
With such a choice, $Y_i^\varepsilon$ depends smoothly on $q$, and the convergence of $Y_i^\varepsilon$ to $\widehat{X}_i^{q}$ is uniform as well with respect to $q$. Similar properties hold for all convergences under consideration in the proof. Since all our results on subelliptic estimates (Appendix \ref{sec:uniform_subelliptic_estimates}) and localization of the heat kernels (hypoelliptic Kac's principle, Theorem \ref{theo:localheat}) are valid uniformly with respect to $q$ in this regular neighborhood, we can keep track of the regularity with respect to $q$ in the entire proof above, and smoothness with respect to $q$ of all the coefficients of the expansion follows.

\part{Proof of the complete asymptotic expansion in Theorem \ref{main_thm}}\label{part_complete}

Theorem \ref{main_thm_weak} is a weaker version of Theorem \ref{main_thm}, in which we have obtained the limit, i.e., the first term of the expansion with respect to $\varepsilon$. 
The full version of Theorem \ref{main_thm} states an asymptotic expansion at any order with respect to $\varepsilon$.

\section{Idea of the proof}\label{sec_ideaoftheproof}
Surprisingly, the proof of Theorem \ref{main_thm}, done in Section \ref{sec:proof_main_thm}, is much more difficult than the (quite easy) one of Theorem \ref{main_thm_weak} done in Section \ref{sec:proof_main_thm_prelim}.
Deriving the complete expansion indeed requires significant additional work. In particular, as we explain hereafter, it requires to use global smoothing estimates (established in Appendix \ref{sec:global_sR} and in Appendix \ref{sec:global_estimates}) and to consider an adequate modification $\triangle^{\varepsilon,\gamma}$ of the operator $\triangle^\varepsilon$, which complicates significantly the analysis.

Hereafter, we explain our proof approach and we point out the main difficulties, in order to motivate some of the developments that will follow.

\subsection{Duhamel formula}\label{sec_duh}
Consider the operator $\triangle^\varepsilon$ defined by \eqref{def_triangleepsilon} in Section \ref{sec:proof_main_thm_prelim}.
As in \cite{Ba-13}, the starting point is the Duhamel formula
\begin{equation*}
e^{t\triangle^\varepsilon} = e^{t \widehat{\triangle}^{q}} + \int_0^t e^{(t-s)\triangle^\varepsilon}(\triangle^\varepsilon-\widehat{\triangle}^{q})e^{s \widehat{\triangle}^{q}}\, ds 
\end{equation*}
for $t>0$.
Setting 
\begin{equation*}
\Sigma_i(t) = \left\{ s^{i+1} = (s_1,\ldots,s_{i+1})\in(0,+\infty)^{i+1}\quad \Big\vert\quad \sum_{k=1}^{i+1} s_k = t \right\} \qquad \forall i\in\N^* ,
\end{equation*}
given any $N\in\N^*$ we obtain by iteration
\begin{multline}\label{rough_duh1}
e^{t\triangle^{\varepsilon}} 
= e^{t \widehat{\triangle}^{q}} + \sum_{i=1}^N \int_{\Sigma_i(t)} e^{s_1 \widehat{\triangle}^{q}}  (\triangle^{\varepsilon} - \widehat{\triangle}^{q}) e^{s_2 \widehat{\triangle}^{q}} \cdots  (\triangle^{\varepsilon} - \widehat{\triangle}^{q}) e^{s_{i+1} \widehat{\triangle}^{q}} \, ds^{i+1} \\
+ \int_{\Sigma_{N+1}(t)} e^{s_1 \triangle^\varepsilon}  (\triangle^{\varepsilon} - \widehat{\triangle}^{q}) e^{s_2 \widehat{\triangle}^{q}} \cdots  (\triangle^{\varepsilon} - \widehat{\triangle}^{q}) e^{s_{N+2} \widehat{\triangle}^{q}} \, ds^{N+2} .
\end{multline}
Besides, using an expansion in homogeneous terms, we have an asymptotic expansion at any order
\begin{equation}\label{expansionaaa}
\triangle^\varepsilon = \widehat{\triangle}^{q} + \varepsilon \mathcal{A}_1 + \cdots + \varepsilon^N\mathcal{A}_N + \varepsilon^{N+1}\mathcal{R}_{N+1}^\varepsilon 
\end{equation}
where $\mathcal{A}_i$, $i\in\N^*$, and $\mathcal{R}_{N+1}^\varepsilon$ are second-order differential operators. 
Moreover all derivations $\mathcal{A}_i$, for $i=1,\ldots,N$, have polynomial coefficients with a degree that is bounded by some power of $N$.
Of course, we must be careful with the topology taken for the convergences and for the asymptotic expansions, and by the way, this is one of the main problems, because the $C^\infty$ topology, which was considered previously, will not be sufficient. Let us explain why.

Using \eqref{rough_duh1} and \eqref{expansionaaa}, for the moment in a formal way, we obtain
\begin{equation}\label{rough_duhamel2}
e^{t\triangle^\varepsilon} = e^{t \widehat{\triangle}^{q}} + \varepsilon\mathcal{C}_1(t) + \cdots + \varepsilon^N \mathcal{C}_N(t) + \varepsilon^{N+1} \mathcal{P}_{N+1}^\varepsilon(t)
\end{equation}
where each operator $\mathcal{C}_j(t)$ is a finite sum of terms $\mathcal{I}_i(t)$ for some $i\in\{1,\ldots,N\}$, where $\mathcal{I}_i(t)$ is defined by
\begin{equation}\label{def_mathcal_Ii}
\mathcal{I}_i(t) = \int_{\Sigma_i(t)} e^{s_1 \widehat{\triangle}^{q}} \mathcal{A}_{j_1} \, e^{s_2 \widehat{\triangle}^{q}} \cdots  \mathcal{A}_{j_i} \, e^{s_{i+1} \widehat{\triangle}^{q}} \, ds^{i+1} 
\end{equation}
with $j_1,\ldots,j_N\in\{1,\ldots,N\}$, and where the remainder term $\mathcal{P}_{N+1}^\varepsilon(t)$ is a finite sum of terms $\varepsilon^k\mathcal{I}_i(t)$, $\varepsilon^k\mathcal{J}_i^\varepsilon(t)$ and $\varepsilon^k\mathcal{K}_i^\varepsilon(t)$ with $k,i\in\N$, $k\leq (N+1)^2$, $1\leq i\leq N$, and $\mathcal{J}_i^\varepsilon(t)$ and $\mathcal{K}_i^\varepsilon(t)$ are defined by
\begin{equation}\label{def_mathcal_Jieps}
\mathcal{J}_i^\varepsilon(t) = \int_{\Sigma_i(t)} e^{s_1 \widehat{\triangle}^{q}}  \mathcal{B}_{1}^\varepsilon \, e^{s_2 \widehat{\triangle}^{q}} \cdots  \mathcal{B}_i^\varepsilon \, e^{s_{i+1} \widehat{\triangle}^{q}} \, ds_{i+1}
\end{equation}
\begin{equation}\label{def_mathcal_Kieps}
\mathcal{K}_i^\varepsilon(t) = \int_{\Sigma_i(t)} e^{s_1 \triangle^\varepsilon}  \mathcal{B}_{1}^\varepsilon \, e^{s_2 \widehat{\triangle}^{q}} \cdots  \mathcal{B}_i^\varepsilon \, e^{s_{i+1} \widehat{\triangle}^{q}} \, ds^{i+1} 
\end{equation}
where each $\mathcal{B}_{j}^\varepsilon$ is a second-order derivation, either equal to some $\mathcal{A}_i$, $i\in\{1,\ldots,N\}$, or to $\mathcal{R}_{N+1}^\varepsilon$.

All operators above are defined as convolutions, i.e., iterated compositions involving the operators $e^{s_i \widehat{\triangle}^{q}}$ and $e^{s_1\triangle^\varepsilon}$, and derivations $\mathcal{A}_i$ and $\mathcal{R}_{N+1}^\varepsilon$ in-between. For instance, we have 
$$
\mathcal{C}_1(t) = \int_0^t e^{(t-s)\widehat{\triangle}^{q}}  \mathcal{A}_1 e^{s \widehat{\triangle}^{q}}\, ds,
\qquad
\mathcal{C}_2(t) = \int_0^t e^{(t-s)\widehat{\triangle}^{q}} \left( \mathcal{A}_2 e^{s\widehat{\triangle}^{q}} + \mathcal{A}_1 \mathcal{C}_1(s)  \right) ds .
$$
Only the terms $\mathcal{K}_i^\varepsilon(t)$ (involved in the remainder $\mathcal{P}_{N+1}^\varepsilon(t)$) contain $e^{s_1 \triangle^\varepsilon}$ as a first term in the convolution. Note that $\widehat{\triangle}^{q}$ is selfadjoint but $\triangle^\varepsilon$ is not selfadjoint in general.

The basic idea is then to take Schwartz kernels in \eqref{rough_duhamel2}, in order to obtain the expansion of the heat kernel $e^\varepsilon$ with respect to $\varepsilon$. 

\medskip

Although apparently simple, at least in a formal way, establishing rigorously the expansion at any order appears to be difficult and technical. The main difficulty is to give a sense to the expansion \eqref{rough_duhamel2} with respect to some appropriate topology. The asymptotic expansion \eqref{rough_duhamel2} will be written in the sense of smoothing operators (see Proposition \ref{prop_expansion_semigroup_epsgam} in Section \ref{sec:asymptotic_expansion_epsilongamma}), i.e., operators that map continuously any $H^j_{loc}(\R^n)$ to any $H^k_{loc}(\R^n)$, with a norm that is uniformly bounded with respect to $\varepsilon$. More precisely, let $0<t_0<t_1$ be fixed. We would like to prove that
\begin{multline}\label{smoothingproperty_Ii}
\Vert\chi_1\mathcal{I}_i(t)\chi_2\Vert_{L(H^j(\R^n),H^k(\R^n))} \leq \Cst(N,\chi_1,\chi_2,j,k,t_0,t_1) \\
\forall t\in[t_0,t_1] \qquad \forall \chi_1,\chi_2\in C^\infty_c(\R^n) \qquad\forall j,k\in\Z  \qquad \forall i\in\{1,\ldots,N\}
\end{multline}
and that there exists $\varepsilon_0>0$ such that
\begin{multline}\label{smoothingproperty_JKNeps}
\Vert\chi_1\mathcal{J}_{N+1}^\varepsilon(t)\chi_2\Vert_{L(H^j(\R^n),H^k(\R^n))}  + 
\Vert\chi_1\mathcal{K}_{N+1}^\varepsilon(t)\chi_2\Vert_{L(H^j(\R^n),H^k(\R^n))} 
\leq \Cst(N,\chi_1,\chi_2,j,k,t_0,t_1) \\
\forall t\in[t_0,t_1] \qquad \forall \chi_1,\chi_2\in C^\infty_c(\R^n) \qquad\forall j,k\in\Z  \qquad\forall\varepsilon\in[-\varepsilon_0,\varepsilon_0] .
\end{multline}

To prove these smoothing properties, we will prove that the chain of compositions appearing in the integrals \eqref{def_mathcal_Ii}, \eqref{def_mathcal_Jieps} and \eqref{def_mathcal_Kieps}, involving the operators $e^{s_i \widehat{\triangle}^{q}}$ and $e^{s_1\triangle^\varepsilon}$ and the derivations $\mathcal{A}_i$ and $\mathcal{R}_{N+1}^\varepsilon$, is performed in the scale of Sobolev spaces with polynomial weight $H^\alpha_\beta(\R^n)$ (whose definition is recalled in Appendix \ref{sec:global_estimates}), and that at least one of the operators $e^{s_i \widehat{\triangle}^{q}}$ and $e^{s_1\triangle^\varepsilon}$ enjoys a strong smoothing property, able to map continuously any $H^\alpha_\beta(\R^n)$ to $H^{\alpha'}_{\beta'}(\R^n)$ for any $\alpha'\in\R$ and for some appropriate $\beta'\in\R$ (and this, uniformly with respect to $\varepsilon$).

Here, in contrast to Section \ref{sec:proof_main_thm_prelim}, \emph{local} subelliptic estimates are not enough and \emph{global} subelliptic estimates are required, in order to establish such global smoothing properties. Since the derivations appearing in the chain of compositions can be arbitrary (they have no specific relationship with the operators $\widehat{\triangle}^{q}$ and $\triangle^\varepsilon$), the use of Sobolev spaces with polynomial weight $H^\alpha_\beta(\R^n)$ appears to be relevant.

Hereafter, we list the properties that we will have to establish in order to prove \eqref{smoothingproperty_Ii} and \eqref{smoothingproperty_JKNeps}.
Note that \eqref{smoothingproperty_Ii} involves operators $e^{s_i \widehat{\triangle}^{q}}$ and derivations $\mathcal{A}_i$, not depending on $\varepsilon$, while \eqref{smoothingproperty_JKNeps} involves also the operator $e^{s_1\triangle^\varepsilon}$ and the derivation $\mathcal{R}_{N+1}^\varepsilon$ for which we will have to establish properties that are uniform with respect to $\varepsilon$. As we will see in Section \ref{sec_idea_eps}, this will raise a significant additional difficulty.

\subsection{Requirements to prove (\ref{smoothingproperty_Ii})}\label{sec_idea_nilp}
To prove \eqref{smoothingproperty_Ii}, we observe that, inside the integral \eqref{def_mathcal_Ii}, at least one of the real numbers $s_p$ is such that $s_p\geq \frac{t_0}{N}$. We want the corresponding operator $e^{s_p \widehat{\triangle}^{q}}$ to be globally smoothing. To this aim, we will need to establish the following property:
\begin{enumerate}[label=$\bf (P_1)$] 
\item\label{P_globsmoothing} \emph{Global smoothing property in Sobolev spaces with polynomial weight}: there exists $k_0\in\N$ such that 
\begin{multline*}
\Vert e^{\tau \widehat{\triangle}^{q}}\Vert_{L\big(H^\alpha_\beta(\R^n),H^{\alpha'}_{\beta/k_0-k_0(\vert\alpha\vert+\vert\alpha'\vert)}(\R^n)\big)} \leq \Cst(\alpha,\alpha',\beta,\tau_0) \\
\forall \alpha,\alpha',\beta\in\R \qquad \forall \tau_0\in(0,1) \qquad \forall\tau\in[\tau_0,1] .
\end{multline*}
\end{enumerate}
In other words, in positive times ($\tau_0\leq\tau\leq1$) the operator $e^{\tau \widehat{\triangle}^{q}}$ gains differential regularity, with a controlled loss of polynomial weight regularity.

Besides, all other operators in the composition in the integral \eqref{def_mathcal_Ii} are either $e^{s_i \widehat{\triangle}^{q}}$, with $0\leq s_i\leq 1$ or derivations $\mathcal{A}_i$. The second-order derivations $\mathcal{A}_i$ for $i=1,\ldots,N$, because they have polynomial coefficients with maximal degree, say, $N^\ell$ for some $\ell\in\N^*$, and thus map continuously any $H^\alpha_\beta(\R^n)$ to $H^{\alpha-2}_{\beta-N^\ell}(\R^n)$. Concerning $e^{s_i \widehat{\triangle}^{q}}$, we require the following property:
\begin{enumerate}[label=$\bf (P_2)$] 
\item\label{P_stab} \emph{Continuity with controlled loss in the Sobolev spaces with polynomial weight, including time zero}: there exist $\beta_0>0$ and $k_1\in\N$ such that
$$
\Vert e^{\tau \widehat{\triangle}^{q}}\Vert_{L\big(H^\alpha_\beta(\R^n),H^{\alpha-k_1\vert\alpha\vert}_{\beta/k_1-k_1\vert\alpha\vert-k_1}(\R^n)\big)} \leq \Cst(\alpha,\beta) \qquad
\forall \alpha\in\R\qquad\forall\beta\geq\beta_0 \qquad \forall\tau\in[0,1]  .
$$
\end{enumerate}
Note that, in contrast to \ref{P_globsmoothing} which is a smoothing property in positive time, \ref{P_stab} is stated on the time interval $[0,1]$, including time zero.
The property \ref{P_stab} is inferred from the following two properties that we will establish:
\begin{enumerate}[label=$\bf (P_{\theenumi})$]
\setcounter{enumi}{2}
\item\label{P_semigroup} Let $\widehat{\mathcal{D}}^{q,\mathrm{sR}}_{k,\beta}$ be the completion of $C_c^\infty(\R^n)$ for the norm $\Vert \langle x\rangle_{\mathrm{sR}}^\beta(\mathrm{id} - \widehat{\triangle}^{q})^k u\Vert_{L^2(\R^n)}$, which is the domain of $(\mathrm{id} - \widehat{\triangle}^{q})^k$ polynomially weighted with the power $\beta$ of the sR Japanese bracket (see a precise definition in Appendix \ref{sec:global_estimates_sR_weight}). We have
$$
\big\Vert e^{\tau\widehat{\triangle}^{q}} \big\Vert_{L\big(\widehat{\mathcal{D}}^{q,\mathrm{sR}}_{k,\beta}\big)} \leq \Cst(k,\beta) \qquad \forall \tau\in[0,1] \qquad \forall k\in\Z\qquad \forall \beta\geq 1 .
$$

\item\label{P_contembed} \emph{Continuous embeddings}: there exist $\sigma>0$ and $N_0\in\N$ such that
$$
H^{2k}_{\beta+2kN_0}(\R^n) \hookrightarrow \widehat{\mathcal{D}}^{q,\mathrm{sR}}_{k,\beta} \hookrightarrow H^{k\sigma}_{\beta/r(q)-kN_0}(\R^n)
\qquad\forall k\in\N\qquad\forall\beta\in\R
$$
and by duality, $H^{-k\sigma}_{\beta/r(q)+kN_0}(\R^n) \hookrightarrow \widehat{\mathcal{D}}^{q,\mathrm{sR}}_{-k,\beta} \hookrightarrow H^{-2k}_{\beta-2kN_0}(\R^n)$ for all $k\in\N$ and $\beta\in\R$.
\end{enumerate}
The smoothing property of $\mathcal{I}_i(t)$ follows from \ref{P_globsmoothing}, \ref{P_semigroup} and \ref{P_contembed} (the full detail of the argument will be given in Proposition \ref{prop_expansion_semigroup_epsgam}, but one can already note the important fact that $\beta$ must be taken large enough).

The property \ref{P_contembed} follows from \emph{global} subelliptic estimates, that we establish in Appendix \ref{sec:global_estimates} for general H\"ormander operators whose coefficients (as well as their derivatives) have a growth at infinity that is at most polynomial (of course, there, the polynomial property is crucial). Actually, in Appendix \ref{sec:global_estimates} we will establish \ref{P_contembed} for the domains $\widehat{\mathcal{D}}^{q}_{k,\beta}$ of $(\mathrm{id} - \widehat{\triangle}^{q})^k$ polynomially weighted with the power $\beta$ of the usual (not sR) Japanese bracket (see their definition in Appendix \ref{sec:uniform_global_subelliptic_estimates}). But, since we have the inequality $\Cst \Vert \cdot\Vert_{\widehat{\mathcal{D}}^q_{j,\alpha/r(q)}} \leq \Vert \cdot\Vert_{\widehat{\mathcal{D}}_{j,\alpha}^{q,\mathrm{sR}}} \leq \Cst \Vert \cdot\Vert_{\widehat{\mathcal{D}}^q_{j,\alpha}}$ for all $j\in\Z$ and $\alpha>0$ (see \eqref{equivDjalpha} in Section \ref{sec:global_estimates_sR_weight}), \ref{P_contembed} follows.

The property \ref{P_globsmoothing} is inferred from \ref{P_contembed} and from the following property:
\begin{enumerate}[label=$\bf (P'_1)$] 
\item\label{P'_globsmoothing} \emph{Global smoothing property in the iterated domains with polynomial weight}: there exists $k_0\in\N$ such that 
$$
\hspace{-5mm}\Vert e^{\tau \widehat{\triangle}^{q}}\Vert_{L\big(\widehat{\mathcal{D}}^{q,\mathrm{sR}}_{j,\beta}, \widehat{\mathcal{D}}^{q,\mathrm{sR}}_{k,\beta}\big)} \leq \Cst(j,k,\beta,\tau_0) \qquad
\forall j,k\in\Z\qquad \forall\beta\geq 1 \qquad \forall \tau_0\in(0,1) \qquad \forall \tau\in[\tau_0,1] .
$$
\end{enumerate}
While \ref{P'_globsmoothing} is a smoothing property (valid for positive times), the property \ref{P_semigroup}, which must hold also at $\tau=0$ but does not provide any gain of regularity, is equivalent to the fact that $(e^{t\widehat{\triangle}^{q}})_{t\geq 0}$ is a semigroup on the weighted Hilbert space $L^2_\beta(\R^n)$, for every $\beta\geq 1$. These facts are not obvious. We prove \ref{P'_globsmoothing} and \ref{P_semigroup} in Appendix \ref{sec:global_sR} (more precisely, see Proposition \ref{prop_globreg_poidsfixe} in Appendix \ref{sec:global_estimates_sR_weight}) by using upper exponential estimates of the heat kernel of the nilpotent sR Laplacian $\widehat{\triangle}^{q}$. 

At this step, we have realized that, to prove that the operators $\chi_1\mathcal{I}_i(t)\chi_2$ are smoothing, we have to establish some properties for the operator $\widehat{\triangle}^q$ that are of a global nature. To prove them, we will use the instrumental facts that the operator $\widehat{\triangle}^q$ is selfadjoint and polynomial. This is the objective of Appendix \ref{sec:global_sR}. Selfadjointness allows us to use the spectral theorem. 

\begin{remark}
It seems unavoidable in general to use Sobolev spaces with polynomial weight $H^\alpha_\beta(\R^n)$, in the above chain of arguments. Indeed, the second-order derivations $\mathcal{A}_i$ can be arbitrary, in the sense that they ``cannot be factorized" by $\widehat{\triangle}^{q}$ and thus they do not map an iterated domain of $\widehat{\triangle}^{q}$ to another in general.
\end{remark}

\subsection{Requirements to prove (\ref{smoothingproperty_JKNeps})}\label{sec_idea_eps}
Let us now search what properties are required to prove \eqref{smoothingproperty_JKNeps}.
Compared with the operator $\mathcal{J}_i(t)$ defined by \eqref{def_mathcal_Ii}:
\begin{itemize}
\item in the definitions \eqref{def_mathcal_Jieps} and \eqref{def_mathcal_Kieps} of $\mathcal{J}_i^\varepsilon(t)$ and $\mathcal{K}_i^\varepsilon(t)$, the derivations in-between can be either $\mathcal{A}_j$ or $\mathcal{R}_{N+1}^\varepsilon$;
\item in the definition \eqref{def_mathcal_Kieps} of $\mathcal{K}_i^\varepsilon(t)$, we have a final composition by the operator $e^{s_0\triangle^\varepsilon}$ in the integral.
\end{itemize}
Note that, in contrast to $\widehat{\triangle}^q$, the operator $\triangle^\varepsilon$ is not selfadjoint in general.

In order to perform a reasoning as above, it would be desirable that there exist $\ell\in\N$ and $\varepsilon_0>0$ such that:
\begin{enumerate}[label=(\roman*)]
\item\label{desirable_i} $\Vert \mathcal{R}_{N+1}^\varepsilon\Vert_{L(H^\alpha_\beta(\R^n),H^{\alpha-2}_{\beta-N^\ell}(\R^n))} \leq \Cst(\alpha,\beta) \qquad \forall\alpha,\beta\in\R \qquad \forall\varepsilon\in[-\varepsilon_0,\varepsilon_0]$;
\item\label{desirable_ii} the operator $\chi_1 e^{\tau\triangle^\varepsilon}$ satisfy the properties \ref{P_globsmoothing} (global smoothing in the Sobolev spaces with polynomial weight) and \ref{P_stab} (continuity with controlled loss, including time zero), uniformly with respect to $\varepsilon\in[-\varepsilon_0,\varepsilon_0]$.
\end{enumerate}
However, except in the following particular case, we are going to see that these properties are not satisfied in general, which raises a serious difficulty that we will show how to overcome.

\paragraph{A particular case.}
It is interesting to note that \ref{desirable_i} and \ref {desirable_ii} are satisfied under the following additional assumptions:
\begin{itemize}
\item $M=\R^n$;
\item the vector fields $X_0, X_1,\ldots,X_m$ are polynomial;
\item the operator $\triangle$ defined by \eqref{def_DeltaX} is selfadjoint (see Section \ref{sec_rem_Horm}).
\end{itemize}
Indeed, under these assumptions the coefficients of $\triangle^\varepsilon$ are polynomial, and thus the homogeneous expansion \eqref{expansionaaa} is exact for $N$ large enough, i.e., $\mathcal{R}_{N+1}^\varepsilon=0$, hence \ref{desirable_i} is satisfied. As for \ref{desirable_ii}, since $\triangle^\varepsilon$ is also selfadjoint, all results established in Appendix \ref{sec:global_sR} can straightforwardly be extended to such a one-parameter family of operators. In particular, Proposition \ref{prop_globreg_poidsfixe} and continuous embeddings give \ref{desirable_ii}.

Hence, under the above additional assumptions, the results established in Appendix \ref{sec:global_sR} are sufficient to conclude that $\chi_1\mathcal{J}_{N+1}^\varepsilon(t)\chi_2$ is smoothing and then complete the proof of Theorem \ref{main_thm}.

\paragraph{Difficulties in the general case.}
But in the general case where either the vector fields are not polynomial at infinity, or $\triangle$ is not selfadjoint, the properties \ref{desirable_i} and \ref{desirable_ii} may fail and we have to proceed differently.

First of all, when $\triangle$ is not selfadjoint, we cannot use the results of Appendix \ref{sec:global_sR}, but this difficulty is bypassed in Appendix \ref{sec:global_estimates}, where we establish some global smoothing properties for general hypoelliptic H\"ormander operators (not necessarily selfadjoint) depending on a parameter. 

However, this can only be done under the crucial assumption that the growth at infinity of the differential operator is at most polynomial, with a degree that is uniform with respect to the parameter.
This requirement of being at most polynomial at infinity is the main constraint. The results of Section \ref{sec:global_estimates} cannot be applied to $e^{t\triangle^\varepsilon}$, because the growth at infinity of the vector fields $X_i^\varepsilon$ is \emph{not} at most polynomial uniformly with respect to $\varepsilon$, in general (see Example \ref{example_lemmaYi} in Section \ref{sec_defYiepsgam}).

This serious flaw is due to the fact that $Y_i^\varepsilon$ (defined by \eqref{defYieps}) converges to $\widehat{X}_i^{q}$ only in $C^\infty$ topology as $\varepsilon\rightarrow 0$, i.e., $Y_i^\varepsilon$ converges uniformly, as well as all its derivatives, to $\widehat{X}_i^{q}$ only on every compact subset of $\R^n$. But this convergence is \emph{not global} in general.

\paragraph{Adding a ``damping" parameter $\gamma$.}
To overcome this defect of convergence, we introduce in Section \ref{sec_construction_deltaepsiloneta} a slightly different operator $\triangle^{\varepsilon,\gamma}$, depending on an additional (fixed) parameter $\gamma\in(0,1)$, which satisfies $\triangle^{\varepsilon,\gamma} \rightarrow \widehat{\triangle}^{q}$ as $\varepsilon\rightarrow 0$ in a much stronger sense, and whose growth at infinity is uniformly at most polynomial.

To do so, we replace each vector field $Y_i^\varepsilon$ with an adequate modification $Y_i^{\varepsilon,\gamma}$: we define the vector field $Y_i^{\varepsilon,\gamma}$ such that, roughly speaking, $Y_i^{\varepsilon,\gamma} = \widehat{X}^{q}_i$ outside of the sR ball $\hatBsR^{q}(0,1/\varepsilon^\gamma)$ and $Y_i^{\varepsilon,\gamma} = Y_i^\varepsilon$ inside the ball. 
In the key Lemma \ref{lem_epsgamma}, we establish that, while $Y_i^\varepsilon$ converges uniformly (as well as all its derivatives) to $\widehat{X}_i^{q}$ only on every compact, for appropriate (small enough) values of $\gamma$, $Y_i^{\varepsilon,\gamma}$ converges to $\widehat{X}_i^{q}$ globally, on the whole $\R^n$, as $\varepsilon\rightarrow 0$, and the convergence is even valid in any space $L^p$, $p\in[1,+\infty]$, as well as all its derivatives. This much stronger convergence property, obtained thanks to the adequate modification using the parameter $\gamma$, is instrumental in our proof. By the way, we think that it could be useful for other purposes.

The parameter $\gamma$ can be viewed, in some sense, as a ``damping" parameter: $Y_i^{\varepsilon,\gamma}$ is an adequate modification of $Y_i^\varepsilon$, which is sufficiently damped (but not too much) to be of polynomial growth at infinity, uniformly with respect to $\varepsilon$.

The construction of $\triangle^{\varepsilon,\gamma}$ is done in Section \ref{sec_construction_deltaepsiloneta}.
We will prove that the family of vector fields $Y_i^{\varepsilon,\gamma}$ satisfies a uniform polynomial strong H\"ormander condition if $\gamma>0$ is chosen small enough (see Lemma \ref{lem_unif_poly_Hormander}), thus allowing us to use the uniform global subelliptic estimates established in Section \ref{sec:global_estimates}.

Note that, to ensure the validity of the global subelliptic estimates, it could seem sufficient to consider the above truncation on a ball $\hatBsR^{q}(0,1)$ of fixed radius, rather than on a ball $\hatBsR^{q}(0,1/\varepsilon^\gamma)$. But then, the end of the proof of Theorem \ref{main_thm} would fail; more precisely, the application of the final localization argument (hypoelliptic Kac's principle) would fail (see Section \ref{end_of_the_proof_main_thm}). So, we underline that it is important to take $\gamma>0$ small and thus keep the equality $Y_i^{\varepsilon,\gamma} = Y_i^\varepsilon$ valid on a neighborhood of $0$ increasing to $\R^n$ as $\varepsilon\rightarrow 0$.

\medskip

With this modification in mind, we go back to Section \ref{sec_duh} and we apply the Duhamel formula, replacing $\triangle^\varepsilon$ with $\triangle^{\varepsilon,\gamma}$. This is what we will do in Section \ref{sec:asymptotic_expansion_epsilongamma}, and we will obtain the expansion \eqref{rough_duhamel2} with $\triangle^{\varepsilon,\gamma}$ instead of $\triangle^\varepsilon$. 

We will then want to prove \eqref{smoothingproperty_JKNeps} with $\triangle^{\varepsilon,\gamma}$ replacing $\triangle^\varepsilon$. Recalling \ref{desirable_i} and \ref{desirable_ii} as desirable properties, what we will be able to establish is:
\begin{enumerate}[label=$\bf (P_{\theenumi})$]
\setcounter{enumi}{4}
\item\label{Repsgam} We have $\triangle^{\varepsilon,\gamma} = \widehat{\triangle}^{q} + \varepsilon \mathcal{A}_1^{\varepsilon,\gamma} + \varepsilon^2 \mathcal{A}_2^{\varepsilon,\gamma} + \cdots + \varepsilon^N \mathcal{A}_N^{\varepsilon,\gamma} + \varepsilon^{N(1-\gamma)+1-\gamma r(q)} \mathcal{R}_{N+1}^{\varepsilon,\gamma}$ and there exist $\beta_0>0$ and $\varepsilon_0>0$ such that
$$
\Vert \mathcal{R}_{N+1}^{\varepsilon,\gamma}\Vert_{L\big(H^\alpha_\beta(\R^n),H^{\alpha-2}_{\beta-\beta_0}(\R^n)\big)} \leq \Cst(\alpha,\beta) \qquad \forall\alpha,\beta\in\R \qquad \forall\varepsilon\in[-\varepsilon_0,\varepsilon_0] .
$$

\item\label{globsmootheps} \emph{Global-to-local smoothing property}: there exist $k_0\in\N$ and $\alpha_0>0$ such that 
$$
\Vert \chi_1 e^{\tau \triangle^{\varepsilon,\gamma}}\Vert_{L\big(H^{-\alpha}_{k_0\alpha}(\R^n),H^{\beta}(\R^n)\big)} \leq \Cst(\alpha,\beta,\tau_0) \qquad \forall \alpha,\beta\geq\alpha_0 \qquad \forall \tau_0\in(0,1) \qquad \forall\tau\in[\tau_0,1] .
$$
\end{enumerate}

We note that, in contrast to the asymptotic expansion \eqref{expansionaaa} of $\triangle^\varepsilon$, the introduction of the damping parameter $\gamma$ implies a loss in the power of $\varepsilon$ for the remainder term, in the asymptotic expansion of $\triangle^{\varepsilon,\gamma}$ given in \ref{Repsgam}. 
The property \ref{globsmootheps} is proved in Appendix \ref{sec:globalregheat}, thanks to the fact that a uniform polynomial strong H\"ormander condition is satisfied as soon as $0<\gamma<\frac{1}{r(q)(r(q)+1)}$ (see Lemma \ref{lem_unif_poly_Hormander}).

We show in Proposition \ref{prop_expansion_semigroup_epsgam} that the properties \ref{Repsgam} and \ref{globsmootheps} are sufficient to conclude.

\subsection{End of the proof}
Finally, we will take Schwartz kernels in Section \ref{taking_Skernels}, in order to obtain an expansion at any order of the heat kernel $e^{\varepsilon,\gamma}$ associated with $\triangle^{\varepsilon,\gamma}$, in function of the heat kernel $\widehat{e}^{q}$ of $e^{t\widehat{\triangle}^{q}}$. 

To conclude, it will remain to relate the heat kernels $e^{\varepsilon,\gamma}$ and $e_{\triangle,\mu}$. As in Section \ref{sec:proof_main_thm_prelim}, the local nature of the small-time asymptotics of heat kernels (hypoelliptic Kac's principle), established in Theorem \ref{theo:localheat} in Section \ref{sec_facts}, will be instrumental there and will be applied several times (as in the proof of Theorem \ref{main_thm_weak}). First, by localization, small-time asymptotics of $e_{\triangle,\mu}$ and of the heat kernel of a representation of $\triangle$ in a chart are the same. Second, the heat kernel $e^\varepsilon$ is directly related to the kernel of the local representation by homogeneity. Unfortunately, the heat kernel $e^{\varepsilon,\gamma}$ does not satisfy this homogeneity property, and an additional difficulty arises here, which we will solve thanks to the adequate construction of the operator $\triangle^{\varepsilon,\gamma}$. There, the fact that the localization argument can be applied is strongly due to the fact that $0<\gamma<1$, more precisely, that $\varepsilon/\varepsilon^\gamma\rightarrow 0$ as $\varepsilon\rightarrow 0$.

\section{Proof of Theorem \ref{main_thm}}\label{sec:proof_main_thm}

Throughout, we assume that $X_0$ is a smooth section of $D$ over $M$. In view of Remark \ref{rem_X0_length2}, each time this is required, we will indicate the modifications that must be done when $X_0$ is a smooth section of $D^2$ over $M$.

We consider the framework and notations introduced in Section \ref{sec:proof_main_thm_prelim}.

\subsection{Construction of $\triangle^{\varepsilon,\gamma}$, with a damping parameter $\gamma$}\label{sec_construction_deltaepsiloneta}
Let $\gamma\in(0,1)$ be fixed, to be chosen later. We assume that $\varepsilon_0>0$ is small enough so that $\delta_{\varepsilon^\gamma}(V)\subset V$ for every $\varepsilon\in[-\varepsilon_0,\varepsilon_0]$.

\subsubsection{Definition and properties of the vector fields $Y_i^{\varepsilon,\gamma}$}\label{sec_defYiepsgam}
Considering the subsets $W_1,W_2$ and the function $\chi\in C^\infty(\R^n)$ introduced in Section \ref{sec:proof_main_thm_prelim}, 
for every $\varepsilon\in[-\varepsilon_0,\varepsilon_0]\setminus\{0\}$, we define
$$
\nu^{\varepsilon,\gamma} = (\delta_{\varepsilon^\gamma}^*\chi) \nu^\varepsilon + (1-(\delta_{\varepsilon^\gamma}^*\chi))\widehat{\mu}^{q}
= \widehat{\mu}^{q} + (\delta_{\varepsilon^\gamma}^*\chi) ( \nu^\varepsilon - \widehat{\mu}^{q})
$$
and
$$
Y_i^{\varepsilon,\gamma} = (\delta_{\varepsilon^\gamma}^*\chi) Y_i^\varepsilon + (1-\delta_{\varepsilon^\gamma}^*\chi)\widehat{X}_i^{q} = \widehat{X}_i^{q} + (\delta_{\varepsilon^\gamma}^*\chi) (Y_i^\varepsilon-\widehat{X}_i^{q}) , \qquad i=0,\ldots,m . 
$$
By construction, we have $Y_i^{\varepsilon,\gamma} = Y_i^\varepsilon$ on $\delta_{1/\varepsilon^\gamma}(W_1)$ and $Y_i^{\varepsilon,\gamma} = \widehat{X}_i^{q}$ on $\R^n\setminus\delta_{1/\varepsilon^\gamma}(W_2)$.
Note that $\delta_{1/\varepsilon^\gamma}(W_1)\subset \delta_{1/\varepsilon}(W_1)$ is a neighborhood of $0$ increasing to $\R^n$ as $\varepsilon\rightarrow 0$. The vector field $Y_i^{\varepsilon,\gamma}$ is thus an adequate restriction of $Y_i^\varepsilon$ on a neighborhood increasing to $\R^n$ as $\vert\varepsilon\vert$ decreases, and $Y_i^{\varepsilon,\gamma}$coincides with $\widehat{X}_i^{q}$ ``at infinity", more precisely, outside of the increasing neighborhood $\delta_{1/\varepsilon^\gamma}(W_2)$. The asymptotics in $\varepsilon^\gamma$, with a $\gamma\in(0,1)$ to be chosen small enough later, will be instrumental several times in the proof. In particular, if we would replace $(\delta_{\varepsilon^\gamma}^*\chi)$ by $(\delta_\lambda^*\chi)$ in the definition of $Y_i^{\varepsilon,\gamma}$, for some $\lambda>0$ fixed, then, at the very end of the proof, the localization theorem (Theorem \ref{theo:localheat}) could not be applied.

Setting $Y_i^{0,\gamma}= \widehat{X}_i^{q}$ for $\varepsilon=0$, $Y_i^{\varepsilon,\gamma}$ depends smoothly on $\varepsilon$ at $\varepsilon=0$ in $C^\infty$ topology (this is because $Y_i^\varepsilon$ depends smoothly on $\varepsilon$ and, on any fixed compact subset $K\subset\R^n$, one has $\delta_{\varepsilon^\gamma}^*\chi=1$ on $K$ as soon as $\varepsilon$ is small enough).

We note that, for $i=0,\ldots,m$, we have the homogeneity property
\begin{equation}\label{homogYiepsgamma}
\varepsilon^\beta \delta_{\varepsilon^\beta}^* Y_i^{\varepsilon,\gamma} = Y_i^{\varepsilon^{1+\beta},\gamma+\beta} \qquad \forall\beta\in(-\gamma,1-\gamma) \qquad\forall \varepsilon\in \big[ -\varepsilon_0^{1/(1+\beta)},\varepsilon_0^{1/(1+\beta)} \big] \setminus\{0\} .
\end{equation}

\begin{remark}\label{rem_comp_Yiepsgamma_Yieps}
Given any $i\in\{0,\ldots,m\}$, since $Y_i^{\varepsilon,\gamma}$ coincides with $Y_i^\varepsilon$ on the growing (as $\varepsilon$ decreases) neighborhood on $\delta_{1/\varepsilon^\gamma}(W_1)$, we obviously have $Y_i^{\varepsilon,\gamma} = Y_i^\varepsilon + \mathrm{O}(\vert\varepsilon\vert^\infty)$ as $\varepsilon\rightarrow 0$ in $C^\infty$ topology.
\end{remark}



The parameter $\gamma>0$ introduced above is used to sufficiently \emph{damp} the growth of the vector field $Y_i^{\varepsilon,\gamma}$ at infinity, with respect to that of the undamped vector field $Y_i^\varepsilon$.

In the next lemma, we prove that, while $Y_i^\varepsilon$ converges uniformly (as well as all its derivatives) to $\widehat{X}_i^{q}$ only on every compact, $Y_i^{\varepsilon,\gamma}$ converges to $\widehat{X}_i^{q}$ \emph{globally}, on the whole $\R^n$, as $\varepsilon\rightarrow 0$, and the convergence is even valid in any space $L^p$, $p\in[1,+\infty]$. This much stronger convergence property is due to the adequate modification using the parameter $\gamma$. 
It implies in particular that the growth at infinity of the coefficients of $Y_i^{\varepsilon,\gamma}$, as well as all their derivatives, is at most polynomial, uniformly with respect to $\varepsilon$.
This will be crucial in what follows, as motivated and explained in Section \ref{sec_idea_eps}.

\begin{lemma}\label{lem_epsgamma}
Recall that $r(q)$ is the degree of nonholonomy at $q$ (we have $r(q)=w_n(q)$, the largest weight at $q$, see Section \ref{sec:sRflag}) and that $\mathcal{Q}(q)=\sum_{i=1}^n w_i(q)$. If $\gamma < \frac{1}{r(q)}$ then:
\begin{enumerate}[label=(\roman*)]
\item\label{lem_epsgamma_i} For every $i\in\{0,\ldots,m\}$, considering the vector field $Y_i^{\varepsilon,\gamma} - \widehat{X}_i^{q}$ as a derivation, we have
\begin{multline*}
\big\Vert (Y_i^{\varepsilon,\gamma} - \widehat{X}_i^{q})f \big\Vert_{W^{k,p}(\R^n)} \leq \Cst(k) \vert\varepsilon\vert^{1-\gamma r(q)} \Vert f\Vert_{W^{k+1,p}(\R^n)}\\
\qquad \forall k\in\N\qquad \forall p\in[1,+\infty]\qquad \forall \varepsilon\in[-\varepsilon_0,\varepsilon_0] \qquad \forall f\in C_c^\infty(\R^n).
\end{multline*}
In particular:
\begin{itemize}
\item Taking $p=+\infty$: $Y_i^{\varepsilon,\gamma}$ converges uniformly on $\R^n$ to $\widehat{X}_i^{q}$ as $\varepsilon\rightarrow 0$ (meaning that all coefficients of the vector field converge uniformly on $\R^n$), as well as all its derivatives. 
\item Taking $p=2$: $Y_i^{\varepsilon,\gamma}f$ converges to $\widehat{X}_i^{q}f$ in $L^2(\R^n)$ as $\varepsilon\rightarrow 0$, as well as all its derivatives. 
\end{itemize}

\item\label{lem_epsgamma_ii} The density of $\nu^{\varepsilon,\gamma}$ with respect to $\widehat{\mu}^{q}$ (which is a constant times the Lebesgue measure of $\R^n$) converges uniformly on the whole $\R^n$ to $1$, and all its derivatives converge uniformly on $\R^n$ to $0$.
Therefore, for every $s\in \R$ we have $H^s(\R^n,\nu^{\varepsilon,\gamma}) = H^s(\R^n,\widehat{\mu}^{q}) = H^s(\R^n)$ with respective norms that are equivalent, uniformly with respect to $\varepsilon$.

\item\label{lem_epsgamma_iii} As in Appendix \ref{sec:global_estimates}, we denote by $(Y_i^{\varepsilon,\gamma})_{i\geq 1}$ the family of vector fields consisting of the vector fields $Y_1^{\varepsilon,\gamma}, Y_2^{\varepsilon,\gamma}, \ldots, Y_m^{\varepsilon,\gamma}$ completed with all their successive Lie brackets.
Then \ref{lem_epsgamma_i} is satisfied as well for every $i\geq 1$.

In particular, if $\varepsilon_0$ is small enough then $\mathrm{Lie}(Y_1^{\varepsilon,\gamma},\ldots,Y_m^{\varepsilon,\gamma})=\R^n$ for every $\varepsilon\in[-\varepsilon_0,\varepsilon_0]$, with a uniform degree of nonholonomy $r$. 

\item\label{lem_epsgamma_iv} Generalizing \ref{lem_epsgamma_i} (which we recover for $N=0$), we have the following asymptotic expansion: given any $N\in\N$, for every $i\in\{0,\ldots,m\}$, for every $\varepsilon\in[-\varepsilon_0,\varepsilon_0]$,
\begin{equation}\label{asymptotic_expansion_Yepsgamma}
Y_i^{\varepsilon,\gamma} = \widehat{X}_i^q + \varepsilon (\delta_{\varepsilon^\gamma}^*\chi) Y_i^{(0)} + \cdots  + \varepsilon^N (\delta_{\varepsilon^\gamma}^*\chi) Y_i^{(N-1)} + \varepsilon^{N(1-\gamma)+1-\gamma r(q)} R_{i,N}^{\varepsilon,\gamma}
\end{equation}
where $R_{i,N}^{\varepsilon,\gamma}$ is a smooth vector field on $\R^n$, depending smoothly on $\varepsilon$, which satisfies
\begin{multline*}
\Vert R_{i,N}^{\varepsilon,\gamma} f \Vert_{W^{k,p}(\R^n)} \leq \Cst(k,N) \Vert f\Vert_{W^{k+1,p}(\R^n)} \\
\forall k\in\N\qquad\forall p\in[1,+\infty]\qquad \forall \varepsilon\in[-\varepsilon_0,\varepsilon_0]\qquad \forall f\in C_c^\infty(\R^n).
\end{multline*}
\end{enumerate}
\end{lemma}

Note that the asymptotic expansion \eqref{asymptotic_expansion_Yepsgamma} is not at order $N$ but is at the order of the floor (integer part) of $N(1-\gamma)+1-\gamma r(q)$.
Note also that the asymptotic expansion \eqref{asymptotic_expansion_Yepsgamma} of $Y_i^{\varepsilon,\gamma}$ is \emph{global}, in contrast to the asymptotic expansion \eqref{expansion_Yieps} of $Y_i^\varepsilon$ which is in $C^\infty$ topology.

We will mainly use this lemma with $p=2$ (note that $W^{k,2}(\R^n)=H^k(\R^n)$) or with $p=+\infty$.

\begin{proof}
Using \eqref{homog_deltaY}, we have $Y_i^{\varepsilon} = \varepsilon^\gamma \delta_{\varepsilon^\gamma}^* Y_i^{\varepsilon^{1-\gamma}}$, and besides, since $\widehat{X}_i^{q}$ is homogeneous of degree $-1$, we have $\widehat{X}_i^{q} = \varepsilon^\gamma \delta_{\varepsilon^\gamma}^*\widehat{X}_i^{q}$, therefore,
\begin{equation}\label{Yiepsgammaeps1mgam}
Y_i^{\varepsilon,\gamma} = \widehat{X}_i^{q} + \varepsilon^\gamma \delta_{\varepsilon^\gamma}^* \Big( \chi \big( Y_i^{\varepsilon^{1-\gamma}} - \widehat{X}_i^{q} \big) \Big),\qquad i=0,\ldots,m,\qquad \forall \varepsilon\in[-\varepsilon_0,\varepsilon_0].
\end{equation}
Now, using \eqref{expansion_XZeps}, we have $Y_i^{\eta} = \widehat{X}_i^{q} + \eta Z_i^{\eta}$ for every $\eta\in\R$ with $Z_i^{\eta}$ depending smoothly on $\eta$ in $C^\infty$ topology, and we have $\delta_\lambda^* Z_i^{\eta} = Z_i^{\lambda\eta}$ for every $\lambda>0$. Hence
$$
Y_i^{\varepsilon,\gamma} = \widehat{X}_i^{q} + \varepsilon \delta_{\varepsilon^\gamma}^* \big( \chi  Z_i^{\varepsilon^{1-\gamma}} \big),\qquad i=0,\ldots,m,\qquad \forall \varepsilon\in[-\varepsilon_0,\varepsilon_0].
$$
Recalling that $Y_i^\eta$ has an asymptotic expansion in $\eta$ around $0$, at any order, in $C^\infty$ topology, $Y_i^\eta = \widehat{X}_i^{q}+\eta Y_i^{(0)} + \mathrm{o}(\eta)$ (see \eqref{expansion_Yieps}), we have that $Z_i^\eta$ converges uniformly to $Y_i^{(0)}$ as well as all its derivatives, on any compact, as $\eta\rightarrow 0$. Therefore, multiplying by $\chi$ that is of compact support, $\chi  Z_i^{\varepsilon^{1-\gamma}}$ converges to $\chi Y_i^{(0)}$ uniformly on $\R^n$, as well as all its derivatives, as $\varepsilon\rightarrow 0$. In particular, 
for every $\alpha\in\N^n$, we have
\begin{equation}\label{bound_chiZ}
\left\Vert \partial_x^\alpha \big( \chi  Z_i^{\varepsilon^{1-\gamma}} \big) \right\Vert_{L^p(\R^n)} \leq \Cst(\alpha) \qquad \forall p\in[1,+\infty]\qquad \forall \varepsilon\in[-\varepsilon_0,\varepsilon_0].
\end{equation}
We note that $\chi(\delta_\eta(x))\, d\delta_{1/\eta}(\delta_\eta(x))$ is uniformly bounded on $\R^n$ by $\Cst/\eta^{r(q)}$, as well as all its derivatives. Now, since
$$\delta_{\varepsilon^\gamma}^* \big( \chi  Z_i^{\varepsilon^{1-\gamma}} \big) (x) = d\delta_{1/\varepsilon^\gamma}(\delta_{\varepsilon^\gamma}(x)) . \left( \chi(\delta_{\varepsilon^\gamma}(x)) \, Z_i^{\varepsilon^{1-\gamma}}( \delta_{\varepsilon^\gamma}(x)) \right) ,
$$
using \eqref{bound_chiZ}, it follows that the pullback $\delta_{\varepsilon^\gamma}^* \big( \chi  Z_i^{\varepsilon^{1-\gamma}} \big)$ is uniformly bounded on $\R^n$ by $\Cst/\vert\varepsilon\vert^{\gamma r(q)}$, as well as all its derivatives (indeed, differentiation with respect to $x$ can only multiply terms by $\varepsilon^\gamma$), and then, given any $f\in C_c^\infty(\R^n)$, $(Y_i^{\varepsilon,\gamma}-\widehat{X}_i^{q})f$ is uniformly bounded on $\R^n$ by $\Cst\,\vert\varepsilon\vert^{1-\gamma r(q)}$, as well as all its derivatives, in any $L^p$ space and independently of $\varepsilon$. 
The proof is similar for the measure $\nu^{\varepsilon,\gamma}$.
The statements \ref{lem_epsgamma_i} and \ref{lem_epsgamma_ii} follow immediately. 

In order to prove \ref{lem_epsgamma_iii}, let us for instance prove that $[Y_i^{\varepsilon,\gamma},Y_j^{\varepsilon,\gamma}]$ converges to $[\widehat{X}_i^{q},\widehat{X}_j^{q}]$. Noting that, by \ref{lem_epsgamma_i}, we have $\Vert Y_i^{\varepsilon,\gamma}\Vert_{L(W^{1,p}(\R^n),L^p(\R^n))}\leq\Cst$ for every $p\in[1,+\infty]$, given any $f\in C_c^\infty(\R^n)$, we have 
\begin{equation*}
\begin{split}
\Vert Y_i^{\varepsilon,\gamma}Y_j^{\varepsilon,\gamma}f - \widehat{X}_i^{q}\widehat{X}_j^{q}f\Vert_{L^p(\R^n)}
&\leq \Vert Y_i^{\varepsilon,\gamma} ( Y_j^{\varepsilon,\gamma} - \widehat{X}_j^{q}) f\Vert_{L^p(\R^n)} + \Vert ( Y_i^{\varepsilon,\gamma} - \widehat{X}_i^{q})\widehat{X}_j^{q}f\Vert_{L^p(\R^n)} \\
&\leq \Cst \Vert Y_j^{\varepsilon,\gamma} - \widehat{X}_j^{q}) f\Vert_{W^{1,p}(\R^n)} + \Cst \vert\varepsilon\vert^{1-\gamma r(q)} \Vert\widehat{X}_j^{q}f\Vert_{W^{1,p}(\R^n)} \\
&\leq \Cst \vert\varepsilon\vert^{1-\gamma r(q)} \Vert f\Vert_{W^{2,p}(\R^n)}
\end{split}
\end{equation*}
and the result follows.

Let us finally prove \ref{lem_epsgamma_iv}. We generalize the above argument, starting from the expansion at the order $N$ (see \eqref{expansion_Xeps} and \eqref{expansion_Yieps})
$$
Y_i^\eta = \widehat{X}_i^{q} + \eta Y_i^{(0)} + \cdots + \eta^N Y_i^{(N-1)} + \eta^{N+1} Z_{i,N}^\eta,\qquad i=0,\ldots,m,
$$
for every $\eta\in\R$, where $Z_{i,N}^\eta$ is a smooth vector field on $\R^n$ depending smoothing on $\eta$ in $C^\infty$ topology. Hence, using \eqref{Yiepsgammaeps1mgam} and the fact that $Y_i^{(k)}$ is homogeneous of order $k$ (with respect to dilations), we obtain
$$
Y_i^{\varepsilon,\gamma} = \widehat{X}_i^{q} + \varepsilon (\delta_{\varepsilon^\gamma}^* \chi) Y_i^{(0)} + \cdots + \varepsilon^N (\delta_{\varepsilon^\gamma}^* \chi) Y_i^{(N-1)} + \varepsilon^{N(1-\gamma)+1} \delta_{\varepsilon^\gamma}^* \big( \chi Z_{i,N}^{\varepsilon^{1-\gamma}} \big)
$$
and then, reasoning as above, all coefficients of the vector field $\delta_{\varepsilon^\gamma}^* \big( \chi Z_{i,N}^{\varepsilon^{1-\gamma}} \big)$ are smooth functions of compact support that are uniformly bounded on $\R^n$ by $\Cst/\vert\varepsilon\vert^{\gamma r(q)}$, as well as all their derivatives. 
Setting $R_{i,N}^{\varepsilon,\gamma} = \varepsilon^{\gamma r(q)} \delta_{\varepsilon^\gamma}^* \big( \chi Z_{i,N}^{\varepsilon^{1-\gamma}} \big)$, the statement follows.
\end{proof}

Thanks to the damping parameter $\gamma$, chosen such that $1-\gamma r(q)>0$, the modified vector fields $Y^{\varepsilon,\gamma}_i$ converge uniformly on $\R^n$ to $\widehat{X}^{q}_i$, as well as all their derivatives (with rate of convergence $\vert\varepsilon\vert^{1-\gamma r(q)}$), while the vector fields $Y^\varepsilon_i$ \emph{do not} converge uniformly on $\R^n$ to $\widehat{X}^{q}_i$ in general, although their convergence is true on any compact. Convergence is also established in Sobolev spaces.

To shed light on the above proof, let us take an example, which moreover shows that the estimates derived in Lemma \ref{lem_epsgamma} (convergence rate $\vert\varepsilon\vert^{1-\gamma r(q)}$) are sharp.

\begin{example}\label{example_lemmaYi}
In $\R^2$, consider the two vector fields
$$
X_1(x_1,x_2) = \frac{\partial}{\partial x_1},\qquad X_2(x_1,x_2) = f(x_1,x_2)\frac{\partial}{\partial x_2}
$$
where $f\in C^\infty(\R^2)$ is such that $f(0,0)=\frac{\partial f}{\partial x_1}(0,0)=\cdots=\frac{\partial^{k-1}f}{\partial x_1^{k-1}}(0,0)=0$ and $\frac{\partial^k f}{\partial x_1^k}(0,0)=1$ for some $k\in\N^*$, i.e., $f(x_1,x_2)=x_1^k(1+\mathrm{O}(x_1,x_2))+\mathrm{O}(x_2)$ around $q=(0,0)$.
The coordinates $(x_1,x_2)$ are privileged around $(0,0)$ and we have $w_1(0,0)=1$ and $r(0,0)=w_2(0,0)=k+1$. 

For instance, for $k=1$ one can take $f(x_1,x_2)=x_1$ or $f(x_1,x_2)=x_1+x_2^2$ (this case is known in the literature as the \emph{Grushin case} and the corresponding Grushin sR structure is singular at $(0,0)$); for $k=2$ one can take $f(x_1,x_2)=x_1^2$ or $f(x_1,x_2)=x_1^2+x_2^2$ or $f(x_1,x_2)=x_1^2-x_2$ (the latter is called \emph{singular Grushin case}).

For every $\varepsilon\in\R\setminus\{0\}$, we have
$$
X_1^\varepsilon(x_1,x_2) = \frac{\partial}{\partial x_1},\qquad X_2^\varepsilon(x_1,x_2) = \frac{f(\varepsilon x_1,\varepsilon^{k+1} x_2)}{\varepsilon^k}\frac{\partial}{\partial x_2}
$$
and thus
$$
\widehat{X}^{(0,0)}_1(x_1,x_2) = \frac{\partial}{\partial x_1},\qquad \widehat{X}^{(0,0)}_2(x_1,x_2) = x_1^k\frac{\partial}{\partial x_2} .
$$
The family of functions $(f^\varepsilon)_{\varepsilon\in\R\setminus\{0\}}$, defined by $f^\varepsilon(x_1,x_2)=\frac{f(\varepsilon x_1,\varepsilon^{k+1} x_2)}{\varepsilon^k}$, converges to the function $\widehat f$ defined by $\widehat f(x_1,x_2)=x_1^k$, uniformly on any compact subset of $\R^2$, but does not converge uniformly on $\R^2$ to $\hat f$ in general (take for instance $f(x_1,x_2)=x_1+x_1^2$, or $f(x_1)=e^{x_1}-1$).

Now, let us fix an arbitrary $\gamma\in(0,\frac{1}{k+1})$ (recall that $r(0,0)=k+1$) and an arbitrary function $\chi\in C_c^\infty(\R)$ such that $\chi(x_1,x_2)=1$ on the sR ball $\BsR(0,1)$ and $\chi(x_1,x_2)=0$ on $\R^2\setminus\BsR(0,2)$. Writing, in short, $\chi\simeq \mathbf{1}_{\BsR(0,1)}$, we have $\delta_{\varepsilon^\gamma}^*\chi\simeq \mathbf{1}_{\BsR(0,1/\varepsilon^\gamma)}$, and we compute
$$
X_1^{\varepsilon,\gamma}(x_1,x_2) = \frac{\partial}{\partial x_1},\qquad X_2^{\varepsilon,\gamma}(x_1,x_2) = x_1\frac{\partial}{\partial x_2} + \chi(\varepsilon^\gamma x_1,\varepsilon^{(k+1)\gamma} x_2) \left( \frac{f(\varepsilon x_1,\varepsilon^{k+1} x_2)}{\varepsilon^k} -x_1 \right)\frac{\partial}{\partial x_2} .
$$
Thanks to the introduction of the parameter $\gamma$ (small enough), now, the family of functions $(f^{\varepsilon,\gamma})_{\varepsilon\in\R\setminus\{0\}}$, defined by $f^{\varepsilon,\gamma}(x_1,x_2)=\frac{f(\varepsilon x_1,\varepsilon^{k+1} x_2)}{\varepsilon^k}$, converges uniformly on $\R^2$ to the function $\widehat f$ defined by $\widehat f(x_1,x_2)=x_1$. Indeed, thanks to the truncation function, we have to prove this convergence for $\vert x_1\vert\leq 1/\vert\varepsilon\vert^\gamma$ and $\vert x_2\vert\leq 1/\vert\varepsilon\vert^{(k+1)\gamma}$, i.e., $\vert \varepsilon x_1\vert\leq \vert\varepsilon\vert^{1-\gamma}$ and $\vert \varepsilon^{k+1} x_2\vert\leq \vert\varepsilon\vert^{(k+1)(1-\gamma)}$. But then, for such values of $(x_1,x_2)$, we have 
$$
f(\varepsilon x_1,\varepsilon^{k+1} x_2)=\varepsilon^k x_1^k(1+\mathrm{O}(\vert\varepsilon\vert^{1-\gamma}))+\mathrm{O}(\vert\varepsilon\vert^{(k+1)(1-\gamma)})
$$
and thus 
$$
f^{\varepsilon,\gamma}(x_1,x_2)= x_1^k(1+\mathrm{O}(\vert\varepsilon\vert^{1-\gamma}))+\mathrm{O}(\vert\varepsilon\vert^{1-(k+1)\gamma}) = x_1^k+\mathrm{O}(\vert\varepsilon\vert^{1-(k+1)\gamma})
$$
where the remainder term $\mathrm{O}(\vert\varepsilon\vert^{1-(k+1)\gamma})$ is uniform with respect to $(x_1,x_2)\in\R^2$. Hence $\Vert f^{\varepsilon,\gamma}-\hat f\Vert_{L^\infty(\R^2)}\leq \Cst\,\vert\varepsilon\vert^{1-(k+1)\gamma}$, as stated in the general lemma above.
\end{example}

Thanks to the introduction of the damping parameter $\gamma$, we have seen that $Y_i^{\varepsilon,\gamma}$ converges to $\widehat{X}_i^{q}$ in a much stronger way than $Y_i^\varepsilon$. We next prove that these modified vector fields $Y_i^{\varepsilon,\gamma}$ satisfy a global strong H\"ormander condition on $\R^n$.

\subsubsection{Uniform polynomial strong H\"ormander condition}
As in Appendix \ref{sec:global_estimates}, we denote by $(Y_i^{\varepsilon,\gamma})_{i\geq 1}$ the family of vector fields consisting of the vector fields $Y_1^{\varepsilon,\gamma}, Y_2^{\varepsilon,\gamma}, \ldots, Y_m^{\varepsilon,\gamma}$ completed with all their successive Lie brackets.
In the following lemma, we prove that the family $(Y_i^{\varepsilon,\gamma})_{i\geq 1}$ satisfies the uniform polynomial strong H\"ormander condition \eqref{UPH_nondegen_1} (see Appendix \ref{sec:globalregheat}), whenever $\gamma>0$ is small enough.

\begin{lemma}\label{lem_unif_poly_Hormander}
There exist $N_1\in\N$ such that, if $\gamma<\frac{1}{r(q)(r(q)+1)}$ then
\begin{equation}\label{ccl_lem_unif_poly_Hormander}
\Vert y\Vert_2^2 \leq \Cst\, \langle x\rangle^{2r(q)} \sum_{i=1}^{N_1} \left\langle Y_i^{\varepsilon,\gamma}(x),y\right\rangle^2
\qquad \forall x,y\in\R^n\qquad\forall \varepsilon\in[-\varepsilon_0,\varepsilon_0] .
\end{equation}
\end{lemma}

This lemma will allow us to use the global subelliptic estimates established in Appendix \ref{sec:global_estimates}, required in order to obtain the asymptotic expansion of the semi-group with respect to $\varepsilon$ (see further).

\begin{proof}
Recall that $Y_i^{0,\gamma}= \widehat{X}_i^{q}$ for $\varepsilon=0$. Let us first prove that the family of vector fields $(\widehat{X}_i^{q})_{i\geq 1}$ (consisting of the vector fields $\widehat{X}_1^{q}, \ldots, \widehat{X}_m^{q}$, completed with their iterated Lie brackets), satisfies the polynomial strong H\"ormander condition, i.e., that \eqref{ccl_lem_unif_poly_Hormander} is satisfied for $\varepsilon=0$. 
Actually, one could apply Remark \ref{rem_polynomial_Hormander} (i.e., one could use \cite[Lemma 2]{Ho-58}), which would give the inequality \eqref{ccl_lem_unif_poly_Hormander} with the term $\langle x\rangle^{2N_0}$ at the right-hand side, for some $N_0\in\N$. This would be enough for our needs.
But actually, hereafter, instead of using the non-obvious result \cite[Lemma 2]{Ho-58}, we make a direct proof by using the homogeneity property of the vector fields $\widehat{X}_i^{q}$. In turn, we obtain the optimal integer $N_0=r(q)$ (it is easy to see on examples that it is optimal).

Let $N_1\in\N$ be an integer, large enough so that the finite family $(\widehat{X}_i^{q})_{0\leq i\leq N_1}$ is a frame of $\R^n$ at every point: such an integer exists because the $m$-tuple of polynomial vector fields $(\widehat{X}_1^{q},\ldots,\widehat{X}_m^{q})$ satisfies the H\"ormander condition at every point of $\R^n$, with a uniform degree of nonholonomy. 
We have to prove that 
\begin{equation}\label{coerciv_N1}
\Vert y\Vert_2^2\leq\Cst\langle x\rangle^{2r(q)}\sum_{i=1}^{N_1} \big\langle \widehat{X}^{q}_i(x),y\big\rangle^2
\qquad \forall x,y\in\R^n.
\end{equation}
Let us introduce the sR pseudo-norm, defined by
$$
\Vert x\Vert_{\mathrm{sR}} = \sum_{i=1}^n \vert x_i\vert^{1/w_i(q)}
$$
for every $x=(x_1,\ldots,x_n)\in\R^n$ in privileged coordinates. Note that $\Vert\delta_\varepsilon(x)\Vert_{\mathrm{sR}}=\vert\varepsilon\vert\Vert x\Vert_{\mathrm{sR}}$ for every $x\in\R^n$ and every $\varepsilon\in\R$.
By the H\"ormander condition and by compactness, we first note that
\begin{equation}\label{coerciv_0}
\Vert y\Vert_2^2\leq\Cst\sum_{i=1}^{N_1} \big\langle \widehat{X}^{q}_i(x),y\big\rangle^2
\qquad  \forall y\in\R^n\qquad\forall x\in\R^n\ \mid\ \Vert x\Vert_{\mathrm{sR}}\leq 1
\end{equation}
and thus \eqref{coerciv_N1} is satisfied as well for $x$ such that $\Vert x\Vert_{\mathrm{sR}}\leq 1$.
Let us now establish \eqref{coerciv_N1} for every $x\in\R^n$. Let $x\in\R^n\setminus\{0\}$ be arbitrary. Setting $\varepsilon=\frac{1}{\Vert x\Vert_{\mathrm{sR}}}$, we have $\Vert\delta_\varepsilon(x)\Vert_{\mathrm{sR}}=1$ and hence, using \eqref{coerciv_0}, we get
$$
\Vert y\Vert_2^2\leq\Cst\sum_{i=1}^{N_1} \big\langle \widehat{X}^{q}_i(\delta_\varepsilon(x)),y\big\rangle^2
\leq \frac{\Cst}{\varepsilon^{2r(q)}} \sum_{i=1}^{N_1} \big\langle \widehat{X}^{q}_i(x),y\big\rangle^2
= \Cst\Vert x\Vert_{\mathrm{sR}}^{2r(q)} \sum_{i=1}^{N_1} \big\langle \widehat{X}^{q}_i(x),y\big\rangle^2
$$
for all $x,y\in\R^n$,
where we have used that $\varepsilon\widehat{X}^{q}_i(\delta_\varepsilon(x))=\widehat{X}^{q}_i(x)$ for $i\in\{1,\ldots,m\}$ and thus $\varepsilon^{k_i}\widehat{X}^{q}_i(\delta_\varepsilon(x))=\widehat{X}^{q}_i(x)$ for $i\in\{1,\ldots,N_1\}$ with $k_i\leq w_n(q)=r(q)$.
To obtain \eqref{coerciv_N1}, we use the fact that $\Vert x\Vert_{\mathrm{sR}}\leq\Cst\langle x\rangle$ for every $x\in\R^n$. 

Now, let us establish \eqref{ccl_lem_unif_poly_Hormander} for every $\varepsilon\in[-\varepsilon_0,\varepsilon_0]$, whenever $\gamma$ is small enough. 
Since $\chi$ has a compact support, there exists $R>0$ such that $\supp(\chi)\subset\hatBsR^{q}(0,R)$. 
Here, given any $r>0$, the set $\hatBsR^{q}(0,r) = \{x\in\R^n\ \mid\ \hatdsR^{q}(0,x)<r\}$ is the sR ball of center $0$ and of radius $r$, for the sR distance $\hatdsR^{q}$ on $\widehat M^{q}\simeq\R^n$.
By definition, given any $f\in C^\infty(\R^n)$ we have
$$
\supp\big( (Y^{\varepsilon,\gamma}_i -\widehat X_i^{q})f \big) \subset \supp( \delta_{\varepsilon^\gamma}^*\chi ) \subset \hatBsR^{q}(0,R/\vert\varepsilon^\gamma\vert) ,
$$
i.e., $Y^{\varepsilon,\gamma}_i$ coincides with $\widehat X_i^{q}$ outside of the sR ball $\hatBsR^{q}(0,R/\vert\varepsilon^\gamma\vert)$. Since \eqref{ccl_lem_unif_poly_Hormander} is satisfied for the family $(\widehat X_i^{q})_{1\leq i\leq N_1}$, we only have to prove that
\eqref{ccl_lem_unif_poly_Hormander} is satisfied for the family $(Y^{\varepsilon,\gamma}_i)_{1\leq i\leq N_1}$ for $x\in\hatBsR^{q}(0,R/\vert\varepsilon^\gamma\vert)$, i.e., defining the $n$-by-$N_1$ matrix
$$
P^{\varepsilon,\gamma}(x) = \begin{pmatrix}
Y_1^{\varepsilon,\gamma}(x) & Y_2^{\varepsilon,\gamma}(x) & \cdots & Y_{N_1}^{\varepsilon,\gamma}(x)
\end{pmatrix}\qquad\forall x\in\R^n,
$$
we have to prove that, if $\gamma$ is small enough then
\begin{equation}\label{coerciv_epsN0}
y^\top P^{\varepsilon,\gamma}(x) P^{\varepsilon,\gamma}(x)^\top y \geq \Cst \langle x\rangle^{-2r(q)} \Vert y\Vert_2^2\qquad \forall x,y\in\R^n.
\end{equation}
It follows from Lemma \ref{lem_epsgamma} that $P^{\varepsilon,\gamma}(x) P^{\varepsilon,\gamma}(x)^\top = \widehat{P}^{q}(x) \widehat{P}^{q}(x)^\top y + \mathrm{O}(\vert\varepsilon\vert^{1-\gamma r(q)})$, when $0<\gamma<\frac{1}{r(q)}$. Therefore, using \eqref{coerciv_N1}, to prove \eqref{coerciv_epsN0} it suffices to observe that
$$
\langle x\rangle^{2r(q)}\vert\varepsilon\vert^{1-\gamma r(q)} \leq \vert\varepsilon\vert^{1-\gamma r(q)(r(q)+1)}
\qquad\forall x\in\hatBsR^{q}(0,R/\vert\varepsilon^\gamma\vert) ,
$$
which is so because, for $x\in\hatBsR^{q}(0,R/\vert\varepsilon^\gamma\vert)$, we have $\vert x\vert\leq \Cst \vert\varepsilon\vert^{-\gamma r(q)}$ (actually, in privileged coordinates, we have $\vert x_i\vert\leq \Cst\vert\varepsilon\vert^{-\gamma w_i(q)}$). The conclusion follows, by taking $\gamma<\frac{1}{r(q)(r(q)+1)}$.
\end{proof}

\subsubsection{Definition of the operator $\triangle^{\varepsilon,\gamma}$}\label{sec_def_deltaepsgam}
\paragraph{Differential operator $\triangle^{\varepsilon,\gamma}$.}
For every $\gamma\in(0,1)$, for every $\varepsilon\in[-\varepsilon_0,\varepsilon_0]$, we define on $C^\infty(\R^n)$ the differential operator
\begin{equation}\label{def_triangleepsgamma}
\triangle^{\varepsilon,\gamma} = \sum_{i=1}^m \left( Y^{\varepsilon,\gamma}_i \right)^2 + \varepsilon (\delta_{\varepsilon^\gamma}^*\chi)  Y^\varepsilon_0 - \varepsilon^2 (\delta_{\varepsilon^\gamma}^*\chi)\left( \delta_\varepsilon^*\mathtt{v} \right) 
\end{equation}
(note that $\triangle^{0,\gamma} = \widehat{\triangle}^{q}$).
By construction, we have $\triangle^{\varepsilon,\gamma} = \triangle^\varepsilon$ on $\delta_{1/\varepsilon^\gamma}(W_1)$ and $\triangle^{\varepsilon,\gamma} = \widehat{\triangle}^{q}$ on $\R^n\setminus\delta_{1/\varepsilon^\gamma}(W_2)$.
We have the homogeneity property (obvious to check, using \eqref{homogYiepsgamma}):
\begin{equation}\label{homog_deltaepsgamma}
\varepsilon^{2\beta} \delta_{\varepsilon^\beta}^* \triangle^{\varepsilon,\gamma} (\delta_{\varepsilon^\beta})_* = \triangle^{\varepsilon^{1+\beta},\gamma+\beta}
\qquad \forall\beta\in(\gamma-1,\gamma) \qquad\forall \varepsilon\in\big[-\varepsilon_0^{1/(1+\beta)},\varepsilon_0^{1/(1+\beta)}\big]\setminus\{0\} .
\end{equation}

When $X_0$ is a smooth section of $D^2$ over $M$, we modify the definition of $\triangle^{\varepsilon,\gamma}$ as follows:
$$
\triangle^{\varepsilon,\gamma} = \sum_{i=1}^m \left( Y^{\varepsilon,\gamma}_i \right)^2 + Y^{\varepsilon,\gamma}_0 - \varepsilon^2 (\delta_{\varepsilon^\gamma}^*\chi) \left( \delta_\varepsilon^*\mathtt{v} \right) . 
$$

\begin{remark}\label{rem_comp_triangleepsgamma_triangleeps}
As in Remark \ref{rem_comp_Yiepsgamma_Yieps}, since $\triangle^{\varepsilon,\gamma}$ coincides with $\triangle^\varepsilon$ on the growing neighborhood $\delta_{1/\varepsilon^\gamma}(W_1)$, we have in particular $\triangle^{\varepsilon,\gamma} = \triangle^\varepsilon + \mathrm{O}(\vert\varepsilon\vert^\infty)$ as $\varepsilon\rightarrow 0$ in $C^\infty$ topology.
\end{remark}


The operator $\triangle^{\varepsilon,\gamma}$, defined by \eqref{def_triangleepsgamma}, depends smoothly on $\varepsilon\in[-\varepsilon_0,\varepsilon_0]$ in $C^\infty$ topology. In particular, $\triangle^{\varepsilon,\gamma}$ converges to $\widehat{\triangle}^{q}$ in $C^\infty$ topology as $\varepsilon\rightarrow 0$. But, as a consequence of Lemma \ref{lem_epsgamma}, we have the following  stronger result.

\begin{lemma}\label{lem_comp_triangleepsgamma_triangleeps2}
If $0<\gamma<1/r(q)$ then
\begin{multline*}
\Vert ( \triangle^{\varepsilon,\gamma} - \widehat{\triangle}^{q} ) f\Vert_{L^p(\R^n)} \leq \Cst \vert\varepsilon\vert^{1-\gamma r(q)}\Vert f\Vert_{W^{2,p}(\R^n)} \\
\forall p\in[1,+\infty] \qquad\forall\varepsilon\in[-\varepsilon_0,\varepsilon_0] \qquad \forall f\in C_c^\infty(\R^n). 
\end{multline*}
\end{lemma}

\paragraph{Asymptotic expansion of $\triangle^{\varepsilon,\gamma}$.}
Let us first give an asymptotic expansion in $C^\infty$ topology.
Using \eqref{expansion_Yieps}, \eqref{formula_Deltaeps} and Remark \ref{rem_comp_triangleepsgamma_triangleeps}, we get that $\triangle^{\varepsilon,\gamma}$ has an asymptotic expansion in $C^\infty$ topology at any order $N$ with respect to $\varepsilon$,
\begin{equation}\label{expansion_triangleepsilon1}
\triangle^{\varepsilon,\gamma} = \widehat{\triangle}^{q} + \varepsilon \mathcal{A}_1 + \varepsilon^2 \mathcal{A}_2 + \cdots + \varepsilon^N \mathcal{A}_N + \mathrm{o}\big(\vert\varepsilon\vert^N\big) 
\end{equation}
where $\mathcal{A}_i$ is a second-order differential operator for every $i\in\N^*$, with
\begin{equation}\label{def_Ai}
\begin{split}
\mathcal{A}_1 &= \sum_{i=1}^m \big( \widehat{X}^{q}_i Y^{(0)}_i + Y^{(0)}_i \widehat{X}^{q}_i \big) + \widehat{X}^{q}_0  \\
\mathcal{A}_2 &= \sum_{i=1}^m \big( \widehat{X}^{q}_i Y^{(1)}_i + Y^{(1)}_i \widehat{X}^{q}_i + \big( Y^{(0)}_i \big)^2 \big) + Y^{(0)}_0 - \mathtt{v}(0)  
\end{split}
\end{equation}
etc. Moreover, $\mathcal{A}_i$ has polynomial coefficients of degree less than $(r(q)+i)^2$ (this bound is not optimal).
Indeed, given any integer $k\geq -1$, the vector field $Y_i^{(k)}$ is a polynomial that is homogeneous of degree $k$ with respect to dilations. In privileged coordinates, its coefficient along $\frac{\partial}{\partial x_n}$ (where $x_n$ has the largest weight, $w_n(q)=r(q)$) must therefore be a polynomial of degree $\leq w_n(q)+k=r(q)+k$. 

\medskip

When $X_0$ is a smooth section of $D^2$ over $M$, since $Y_0^{\varepsilon,\gamma} = \widehat{X}_0^{q} + (\delta _{\varepsilon^\gamma}^*\chi) (Y_0^\varepsilon-\widehat{X}_0^{q}) = \widehat{X}_i^{q} + (\delta _{\varepsilon^\gamma}^*\chi) ( \varepsilon Y_0^{(-1)} + \varepsilon^2 Y_0^{(0)} + \cdots  + \varepsilon^N Y_0^{(N-2)} + \mathrm{o}\big(\vert\varepsilon\vert^N\big) )$ with $\widehat{X}_0^{q}$ homogeneous of order $-2$ (see Remark \ref{rem_X0_length2}), the expansion \eqref{expansion_triangleepsilon1} of $\triangle^{\varepsilon,\gamma}$ remains true as well by replacing $\widehat{\triangle}^{q}$ with $\widehat{\triangle}^{q} +\widehat{X}_0^{q}$ and $\mathcal{A}_1$ with $\mathcal{A}_1 = \sum_{i=1}^m \big( \widehat{X}^{q}_i  Y^{(0)}_i + Y^{(0)}_i \widehat{X}^{q}_i \big) + Y^{(-1)}_0$, etc.

\medskip

The asymptotic expansion \eqref{expansion_triangleepsilon1} is in $C^\infty$ topology. Thanks to \eqref{asymptotic_expansion_Yepsgamma} in Lemma \ref{lem_epsgamma}, we now derive an asymptotic expansion of $\triangle^{\varepsilon,\gamma}$, at any order, valid on the whole $\R^n$ (not only on every compact): in the above definition of $\mathcal{A}_i$, it suffices to replace $Y_i^{(k)}$ by $(\delta_{\varepsilon^\gamma}^*\chi)Y_i^{(k)}$, $\widehat{X}_0^q$ by $(\delta_{\varepsilon^\gamma}^*\chi)\widehat{X}_0^q$, etc, and we obtain
\begin{equation}\label{expansion_triangleepsilon_stronger}
\triangle^{\varepsilon,\gamma} = \widehat{\triangle}^{q} + \varepsilon \mathcal{A}_1^{\varepsilon,\gamma} + \varepsilon^2 \mathcal{A}_2^{\varepsilon,\gamma} + \cdots + \varepsilon^N \mathcal{A}_N^{\varepsilon,\gamma} + \varepsilon^{N(1-\gamma)+1-\gamma r(q)} \mathcal{R}_{N+1}^{\varepsilon,\gamma}
\end{equation}
with
\begin{equation}\label{def_Aiepsgam}
\begin{split}
\mathcal{A}_1^{\varepsilon,\gamma} &= \sum_{i=1}^m \big( \widehat{X}^{q}_i  (\delta_{\varepsilon^\gamma}^*\chi) Y^{(0)}_i + (\delta_{\varepsilon^\gamma}^*\chi) Y^{(0)}_i \widehat{X}^{q}_i \big) + (\delta_{\varepsilon^\gamma}^*\chi) \widehat{X}^{q}_0  \\
\mathcal{A}_2^{\varepsilon,\gamma} &= \sum_{i=1}^m \big( \widehat{X}^{q}_i (\delta_{\varepsilon^\gamma}^*\chi) Y^{(1)}_i + (\delta_{\varepsilon^\gamma}^*\chi) Y^{(1)}_i \widehat{X}^{q}_i + \big( (\delta_{\varepsilon^\gamma}^*\chi) Y^{(0)}_i \big)^2 \big) + (\delta_{\varepsilon^\gamma}^*\chi) Y^{(0)}_0 - (\delta_{\varepsilon^\gamma}^*\chi) \mathtt{v}(0) 
\end{split}
\end{equation}
etc. Because of the damping parameter $\gamma$, the growth of the coefficients of the second-order differential operator $\mathcal{A}_i^{\varepsilon,\gamma}$ is at most polynomial of degree $r(q)-1$ (which is the maximal degree of the coefficients of $\widehat{X}^{q}_i$), for all $i\in\N^*$ and $\varepsilon\in[-\varepsilon_0,\varepsilon_0]$. Note that $\mathcal{A}_i^{\varepsilon,\gamma} = \mathcal{A}_i + \mathrm{O}(\vert\varepsilon\vert^\infty)$
as $\varepsilon\rightarrow 0$ in $C^\infty$ topology.
We infer from \ref{lem_epsgamma_iv} in Lemma \ref{lem_epsgamma} that, in \eqref{expansion_triangleepsilon_stronger}, the term $\mathcal{R}_{N+1}^{\varepsilon,\gamma}$ is a finite sum of products of $R_{j,N}^{\varepsilon,\gamma}$ (whose coefficients are uniformly bounded) with either some $\widehat{X}^{q}_i$ (whose coefficients are polynomial of degree $\leq r(q)-1$) or some $(\delta_{\varepsilon^\gamma}^*\chi) Y^{(0)}_i$ (whose coefficients are uniformly bounded). Therefore $\mathcal{R}_{N+1}^{\varepsilon,\gamma}$ a second-order differential operator, depending smoothly on $\varepsilon$, with coefficients whose growth is at most polynomial of degree $r(q)-1$, and
\begin{equation*}
\Vert \mathcal{R}_{N+1}^{\varepsilon,\gamma}\Vert_{L\big(H^\alpha_\beta(\R^n),H^{\alpha-2}_{\beta-r(q)+1}\big)} \leq \Cst(\alpha,\beta) \qquad \forall\alpha,\beta\in\R \qquad \forall\varepsilon\in[-\varepsilon_0,\varepsilon_0] 
\end{equation*}
which is exactly the property \ref{Repsgam} that we have identified in Section \ref{sec_idea_eps}.
This means that the asymptotic expansion \eqref{expansion_triangleepsilon_stronger} is in the sense of operators mapping $H^\alpha_\beta(\R^n)$ to $H^{\alpha-2}_{\beta-N^\ell}(\R^n)$, uniformly with respect to $\varepsilon$.

Note that we also have $\Vert \mathcal{R}_{N+1}^{\varepsilon,\gamma} f \Vert_{W^{k,p}(\R^n)} \leq \Cst(k,N) \Vert f\Vert_{W^{k+2,p}(\R^n)}$ for all $k\in\N$, $p\in[1,+\infty]$, $\varepsilon\in[-\varepsilon_0,\varepsilon_0]$ and $f\in C_c^\infty(\R^n)$. This means that the asymptotic expansion \eqref{expansion_triangleepsilon_stronger} is also in the sense of operators mapping $W^{k+2,p}(\R^n)$ to $W^{k,p}(\R^n)$, uniformly with respect to $\varepsilon$.

\paragraph{Semigroup generated by $\triangle^{\varepsilon,\gamma}$.} 
In the sequel, we assume that $0<\gamma<1/r(q)$. By Lemma \ref{lem_epsgamma}, we have $H^s(\R^n,\nu^{\varepsilon,\gamma})=H^s(\R^n)$, so that, hereafter, we use the Lebesgue measure in all Sobolev spaces that we consider.

Reasoning as in the proof of Lemma \ref{lem_semigroup} and 
using Lemma \ref{lem_comp_triangleepsgamma_triangleeps2} with $p=2$,
we first observe that there exists $C\geq 0$, not depending on $\varepsilon\in[-\varepsilon_0,\varepsilon_0]$, such that the operator 
$\triangle^{\varepsilon,\gamma}-C\,\mathrm{id}$ 
is closed and dissipative in $L^2(\R^n)$, as well as its adjoint. 
Setting $D(\triangle^{\varepsilon,\gamma}) = \{ f\in L^2(\R^n) \ \mid\ \triangle^{\varepsilon,\gamma} f \in L^2(\R^n) \}$, it follows that the operator $(\triangle^{\varepsilon,\gamma},D(\triangle^{\varepsilon,\gamma}))$ generates a strongly continuous (quasicontraction) semigroup $(e^{t\triangle^{\varepsilon,\gamma}})_{t\geq 0}$ on $L^2(\R^n)$, satisfying 
$\Vert e^{t\triangle^{\varepsilon,\gamma}}\Vert_{L(L^2(\R^n))} \leq e^{Ct}$ 
for every $t\geq 0$ and every $\varepsilon\in[-\varepsilon_0,\varepsilon_0]\setminus\{0\}$. 

Following Section \ref{sec:global_estimates}, for every $k\in\Z$, we define the Hilbert space $\mathcal{D}^{\varepsilon,\gamma}_k$ as the completion of $C_c^\infty(\R^n)$ for the norm 
$$
\Vert f\Vert_{\mathcal{D}^{\varepsilon,\gamma}_k} = \Vert (2C\,\mathrm{id}-\triangle^{\varepsilon,\gamma}-(\triangle^{\varepsilon,\gamma})^*)^k f\Vert_{L^2(\R^n)} .
$$
For $k=1$, we have $\mathcal{D}^{\varepsilon,\gamma}_1 = D(\triangle^{\varepsilon,\gamma})$.

By applying the semigroup estimate to $(2C\,\mathrm{id}-\triangle^{\varepsilon,\gamma}-(\triangle^{\varepsilon,\gamma})^*)^k f$ and by a classical argument of restriction or extension of semigroups on the Sobolev towers (see \cite{EngelNagel}), we obtain the uniform estimates
\begin{equation}\label{ineg_contraction_epsgamma}
\Vert  e^{t\triangle^{\varepsilon,\gamma}}\Vert_{L(\mathcal{D}^{\varepsilon,\gamma}_k)}\leq \Cst\, e^{Ct}\qquad \forall t\geq 0\qquad \forall k\in\Z \qquad\forall\varepsilon\in[-\varepsilon_0,\varepsilon_0].
\end{equation}
Recall that, for $\varepsilon=0$, we have $\triangle^{0,\gamma}=\widehat{\triangle}^{q}$. For every $k\in\Z$, the Hilbert space $\widehat{\mathcal{D}}^{q}_k$ is the completion of $C_c^\infty(\R^n)$ for the (equivalent) norm $\Vert f\Vert_{\widehat{\mathcal{D}}^{q}_k} = \Vert (\mathrm{id}-\widehat{\triangle}^{q})^kf\Vert_{L^2(\R^n)}$.
Actually, for $\varepsilon=0$ we have the better estimate $\Vert e^{t \widehat{\triangle}^{q}}\Vert_{L(\widehat{\mathcal{D}}^{q}_k)}\leq 1$ for every $k\in\Z$ (contraction semigroup).

\paragraph{Heat kernel $e^{\varepsilon,\gamma}$ of $\triangle^{\varepsilon,\gamma}$.}
By hypoellipticity (see Corollary \ref{cor_heat_uniform}), the Schwartz kernel of $e^{t\triangle^{\varepsilon,\gamma}}$ has a continuous density with respect to $\nu^{\varepsilon,\gamma}$, which is the smooth function
\begin{equation*}
e^{\varepsilon,\gamma} = e_{\triangle^{\varepsilon,\gamma},\nu^{\varepsilon,\gamma}}: (0,+\infty)\times \R^n\times\R^n\rightarrow(0,+\infty) .
\end{equation*}
Using \eqref{homog_deltaepsgamma}, we have the homogeneity property (which will be useful at the end of the proof):
\begin{multline}\label{homog_eepsgamma}
e^{\varepsilon,\gamma}(t,x,x') = \vert\varepsilon\vert^{\beta\mathcal{Q}(q)} e^{\varepsilon^{1-\beta},\gamma-\beta} ( \varepsilon^{2\beta} t, \delta_{\varepsilon^\beta}(x), \delta_{\varepsilon^\beta}(x') ) \\
\forall\beta\in(\gamma,1-\gamma)\qquad \forall \varepsilon\in \big[ -\varepsilon_0^{1/(1-\beta)},\varepsilon_0^{1/(1-\beta)} \big] \setminus\{0\}\qquad \forall (t,x,x')\in (0,+\infty)\times\R^n\times\R^n .
\end{multline}

\paragraph{Convergence of $e^{\varepsilon,\gamma}$ to $\widehat{e}^{q}$.}
Applying Theorem \ref{thm_CV_general1}  (in Section \ref{sec_convergence_heat_kernels}), we obtain that $e^{\varepsilon,\gamma}$ converges to $\widehat{e}^{q}$ in $C^\infty((0,+\infty)\times\R^n\times\R^n)$ topology as $\varepsilon\rightarrow 0$ (when $X_0$ is a smooth section of $D^2$ over $M$, one has to replace $\widehat{\triangle}^{q}$ with $\widehat{\triangle}^{q} +\widehat{X}_0^{q}$), but, as announced earlier, we are now going to derive an asymptotic expansion in $\varepsilon$ at any order. This is possible thanks to the introduction of the parameter $\gamma$.

\subsection{Asymptotic expansion in $\varepsilon$ of the semigroup}\label{sec:asymptotic_expansion_epsilongamma}
Recall that the operators $\mathcal{C}_i(t)$ have been defined in Section \ref{sec_duh}, as finite sums of  operators $\mathcal{I}_i(t)$, defined by the integral \eqref{def_mathcal_Ii}, which is a convolution in which the compositions involve the operators $e^{s\widehat{\triangle}^{q}}$ and the derivations $\mathcal{A}_j$. The operators $\mathcal{C}_i(t)$ appear in the asymptotic expansion \eqref{rough_duhamel2} of $e^{t\triangle^\varepsilon}$ in $C^\infty$ topology. In the proposition hereafter, we prove that they appear as well in the asymptotic expansion of $e^{t\triangle^{\varepsilon,\gamma}}$ in the sense of uniformly smoothing operators. This key result is the byproduct of the introduction of the damping parameter $\gamma$.

\begin{proposition}\label{prop_expansion_semigroup_epsgam}
We assume that
$$
\gamma<\frac{1}{r(q)(r(q)+1)} .
$$
Given any $0<t_0<t_1<+\infty$, given any $N\in\N^*$, we have
\begin{equation}\label{expansion_exp_triangle_eps}
\begin{split}
e^{t\triangle^{\varepsilon,\gamma}} = e^{t \widehat{\triangle}^{q}} + \varepsilon \mathcal{C}_1(t) + \cdots + \varepsilon^N \mathcal{C}_N(t) + \varepsilon^{N+1} \mathcal{Q}_{N+1}^{\varepsilon,\gamma}(t)
\end{split}
\end{equation}
as $\varepsilon\rightarrow 0$, for every $t\in[t_0,t_1]$, where all operators in \eqref{expansion_exp_triangle_eps} are locally smoothing for $t\in[t_0,t_1]$, i.e.,
\begin{multline*}
\sum_{i=1}^N\Vert\chi_1\mathcal{C}_i(t)\chi_2\Vert_{L(H^j(\R^n),H^k(\R^n))} + \Vert\chi_1\mathcal{Q}_{N+1}^{\varepsilon,\gamma}(t)\chi_2\Vert_{L(H^j(\R^n),H^k(\R^n))}
\leq \Cst(\chi_1,\chi_2,j,k,t_0,t_1,N) \\
\forall t\in[t_0,t_1] \qquad \forall \chi_1,\chi_2\in C^\infty_c(\R^n) \qquad\forall j,k\in\Z  .
\end{multline*}
\end{proposition}

\begin{proof}
Following the approach described in Section \ref{sec_ideaoftheproof}, we start by applying the Duhamel formula, with the operator $\triangle^{\varepsilon,\gamma}$,
$$
e^{t\triangle^{\varepsilon,\gamma}} 
=\, e^{t \widehat{\triangle}^{q}} + \int_0^t e^{(t-s)\triangle^{\varepsilon,\gamma}} (\triangle^{\varepsilon,\gamma} - \widehat{\triangle}^{q}) e^{s \widehat{\triangle}^{q}}\, ds  ,
$$
which we iterate $N$ times, thus obtaining the formula \eqref{rough_duh1} with $\triangle^\varepsilon$  replaced by $\triangle^{\varepsilon,\gamma}$. Now, using the asymptotic expansion \eqref{expansion_triangleepsilon_stronger} of $\triangle^{\varepsilon,\gamma}$, applied at order $N_1$ such that $N_1(1-\gamma)+1-\gamma r(q)\geq N+1$, we obtain
\begin{equation}\label{expansion_exp_triangle_epsgam}
e^{t\triangle^{\varepsilon,\gamma}} = e^{t \widehat{\triangle}^{q}} + \varepsilon \mathcal{C}_i^{\varepsilon,\gamma}(t) + \cdots + \varepsilon^N \mathcal{C}_N^{\varepsilon,\gamma}(t) + \varepsilon^{N+1} \mathcal{P}_{N+1}^{\varepsilon,\gamma}(t)
\end{equation}
where each operator $\mathcal{C}_j^{\varepsilon,\gamma}(t)$ is a finite sum of terms $\mathcal{I}_i^{\varepsilon,\gamma}(t)$ for some $i\in\{1,\ldots,N\}$, where $\mathcal{I}_i^{\varepsilon,\gamma}(t)$ is defined by
\begin{equation}\label{def_mathcal_Ii_epsgam}
\mathcal{I}_i^{\varepsilon,\gamma}(t) = \int_{\Sigma_i(t)} e^{s_1 \widehat{\triangle}^{q}} \mathcal{A}_{j_1}^{\varepsilon,\gamma} \, e^{s_2 \widehat{\triangle}^{q}} \cdots  \mathcal{A}_{j_i}^{\varepsilon,\gamma} \, e^{s_{i+1} \widehat{\triangle}^{q}} \, ds^{i+1} 
\end{equation}
with $j_1,\ldots,j_N\in\{1,\ldots,N\}$, and where the remainder term $\mathcal{P}_{N+1}^{\varepsilon,\gamma}(t)$ is a finite sum of terms $\varepsilon^k\mathcal{I}_i^{\varepsilon,\gamma}(t)$, $\varepsilon^\alpha\mathcal{J}_i^{\varepsilon,\gamma}(t)$ and $\varepsilon^\alpha\mathcal{K}_i^{\varepsilon,\gamma}(t)$ with $k,i\in\N$, $k\leq (N+1)^2$, $1\leq i\leq N$, $\alpha\in[0,(N+1)^2]$, and $\mathcal{J}_i^{\varepsilon,\gamma}(t)$ and $\mathcal{K}_i^{\varepsilon,\gamma}(t)$ are defined by
\begin{equation}\label{def_mathcal_Jiepsgam}
\mathcal{J}_i^{\varepsilon,\gamma}(t) = \int_{\Sigma_i(t)} e^{s_1 \widehat{\triangle}^{q}}  \mathcal{B}_{1}^{\varepsilon,\gamma} \, e^{s_2 \widehat{\triangle}^{q}} \cdots  \mathcal{B}_i^{\varepsilon,\gamma} \, e^{s_{i+1} \widehat{\triangle}^{q}} \, ds^{i+1} 
\end{equation}
\begin{equation}\label{def_mathcal_Kiepsgam}
\mathcal{K}_i^{\varepsilon,\gamma}(t) = \int_{\Sigma_i(t)} e^{s_1 \triangle^{\varepsilon,\gamma}}  \mathcal{B}_{1}^{\varepsilon,\gamma} \, e^{s_2 \widehat{\triangle}^{q}} \cdots  \mathcal{B}_i^{\varepsilon,\gamma} \, e^{s_{i+1} \widehat{\triangle}^{q}} \, ds^{i+1}
\end{equation}
where each $\mathcal{B}_{j}^{\varepsilon,\gamma}$ is a second-order derivation, either equal to some $\mathcal{A}_i^{\varepsilon,\gamma}$, $i\in\{1,\ldots,N\}$, or to $\mathcal{R}_{N_1+1}^{\varepsilon,\gamma}$, whose coefficients growth at infinity is at most polynomial of degree $r(q)$.

Note that, compared with the operator $\mathcal{I}_i(t)$ defined by \eqref{def_mathcal_Ii} in Section \ref{sec_duh}, in the definition \eqref{def_mathcal_Ii_epsgam} of $\mathcal{I}_i^{\varepsilon,\gamma}(t)$ the derivations $\mathcal{A}_j$ are replaced with $\mathcal{A}_j^{\varepsilon,\gamma}$. 

The above expansion is formal. We give it a rigorous meaning in the following proposition, by using, as explained in Section \ref{sec_ideaoftheproof}, the various global smoothing properties established in Appendix \ref{sec:global_estimates} and in Appendix \ref{sec:global_sR}. 

The arguments that we develop hereafter will show that all operators $\mathcal{C}_i^{\varepsilon,\gamma}(t)$, $\mathcal{C}_i(t)$, for $i\in\N^*$ and the operators $\mathcal{P}_{N+1}^{\varepsilon,\gamma}(t)$ and $\mathcal{Q}_{N+1}^{\varepsilon,\gamma}(t)$ are locally smoothing for $t\in[t_0,t_1]$, and that, actually, $\mathcal{C}_i^{\varepsilon,\gamma}(t) = \mathcal{C}_i(t) + \mathrm{O}(\vert\varepsilon\vert^\infty)$ in $C^\infty$ topology, and then \eqref{expansion_exp_triangle_eps} will follow from \eqref{expansion_exp_triangle_epsgam}.

We start by noting that:
\begin{itemize}
\item by construction, the coefficients of the operator $\triangle^{\varepsilon,\gamma}$, as well as all their derivatives, have a growth at infinity that is at most polynomial, uniformly with respect to $\varepsilon\in[-\varepsilon_0,\varepsilon_0]$;
\item the vector fields defining the one-parameter family of H\"ormander operators $(\triangle^{\varepsilon,\gamma})_{-\varepsilon_0\leq\varepsilon\leq\varepsilon_0}$ satisfy the uniform polynomial strong H\"ormander condition established in Lemma \ref{lem_unif_poly_Hormander}.
\end{itemize}
Therefore, the global subelliptic estimates established in Appendix \ref{sec:global_estimates} can be applied to $\triangle^{\varepsilon,\gamma}$ (in other words, we have the properties \ref{Repsgam} and \ref{globsmootheps} motivated in Section \ref{sec_idea_eps}).

We are also going to apply to $\widehat{\triangle}^{q}$ the global smoothing properties established in Appendix \ref{sec:global_sR} (in other words, we have the properties \ref{P_globsmoothing} and \ref{P_stab} motivated in Section \ref{sec_idea_nilp}).


Let $\chi_1,\chi_2\in C_c^\infty(\R^n)$, let $0<t_0<t_1$ be arbitrarily fixed and let $j,k\in\Z$ be arbitrary.

\begin{lemma}\label{lem_Ci_locsmoothing}
All operators $\mathcal{C}_i(t)$ and $\mathcal{C}_i^{\varepsilon,\gamma}(t)$, $i\in\N^*$, are locally smoothing for $t\in[t_0,t_1]$. 
\end{lemma}

\begin{proof}
It suffices to prove that, for every $i\in\N^*$, the operator $\chi_1\mathcal{I}_i(t)\chi_2$, defined by \eqref{def_mathcal_Ii}, maps continuously $H^j(\R^n)$ to $H^{k}(\R^n)$, with a norm that is uniformly bounded with respect to $t\in[t_0,t_1]$ (in other words, we want to prove \eqref{smoothingproperty_Ii}).
The proof for the operator $\mathcal{I}_i^{\varepsilon,\gamma}(t)$, defined by \eqref{def_mathcal_Ii_epsgam}, will be exactly similar, by replacing the derivations $\mathcal{A}_i$ with $\mathcal{A}_i^{\varepsilon,\gamma}$ which are, as well, derivations whose coefficients growth is at most polynomial, uniformly with respect to $\varepsilon$.

Although the argument has been sketched in Section \ref{sec_idea_nilp}, we provide hereafter the detailed proof. Inside the integral \eqref{def_mathcal_Ii} defining $\mathcal{I}_i(t)$, at least one of the real numbers $s_p$ is such that $s_p\geq \frac{t_0}{N}$. We write:
$$
\chi_1\mathcal{I}_i(t)\chi_2 = \int_{\Sigma_i(t)} \chi_1 e^{s_1 \widehat{\triangle}^{q}} \mathcal{A}_{j_1} \, e^{s_2 \widehat{\triangle}^{q}} \cdots \mathcal{A}_{j_{p-1}} \, e^{s_p \widehat{\triangle}^{q}}  \mathcal{A}_{j_p} e^{s_{p+1} \widehat{\triangle}^{q}} \cdots \mathcal{A}_{j_i} \, e^{s_{i+1} \widehat{\triangle}^{q}} \chi_2 \, ds^{i+1}  .
$$
Using \eqref{equivDjalpha} and the global continuous embeddings \eqref{inclusions_global} and \eqref{inclusions_global_neg}, the localization operator $\chi_2$ maps continuously $H^j(\R^n)$ to $\widehat{\mathcal{D}}^{q}_{-m,\beta}$ for some $m\in\N$, for every $\beta\in\R$. We note here that $\beta$ can be chosen arbitrarily large.
Then, by Proposition \ref{prop_globreg_poidsfixe} (and more precisely, by \eqref{15:34}), the operator $e^{s_{i+1} \widehat{\triangle}^{q}}$ maps continuously $\widehat{\mathcal{D}}^{q}_{-m,\beta}$ to itself for every $s_i\in[0,t_1]$, provided that $\beta\geq 1$, with a norm constant not depending on $s_i\in[0,t_1]$.
By \eqref{inclusions_global_neg}, $\widehat{\mathcal{D}}^{q}_{-m,\beta}$ is continously embedded into $H^{-2m}_{\beta-2mN}(\R^n)$. Applying the derivation $\mathcal{A}_{j_i}$ then maps to $H^{-2m-2}_{\beta-2mN-(r(q)+j_i)^2}(\R^n)$, because $\mathcal{A}_{j_i}$ is a second-order derivation with polynomial coefficients of degree less than $(r(q)+j_i)^2$. This step is repeated $i-p$ times, so that 
$\mathcal{A}_{j_p} e^{s_{p+1} \widehat{\triangle}^{q}} \cdots \mathcal{A}_{j_i} \, e^{s_{i+1} \widehat{\triangle}^{q}} \chi_2$ maps continuously $H^j(\R^n)$ to $H^\alpha_\beta(\R^n)$, for some $\alpha\leq 0$, for every $\beta\geq 1$. We stress that, in the latter repeated argument, as well as in the following, it is important to ensure that the weight be greater than $1$, in order to be able to apply \eqref{15:34}. 
One now applies the operator $e^{s_p \widehat{\triangle}^{q}}$, which maps continuously $H^{\alpha}_{\beta}(\R^n)$ to $H^{\alpha'}_{\beta-k_0(\vert\alpha\vert+\vert\alpha'\vert)}(\R^n)$, for every $\alpha'\in\R$, by applying the global smoothing property \ref{P_globsmoothing} stated in Section \ref{sec_idea_nilp} (itself following from Proposition \ref{prop_globreg_poidsfixe} in Appendix \ref{sec:global_estimates_sR_weight} and from the embeddings \eqref{inclusions_global} and \eqref{inclusions_global_neg}). Then we continue by applying again compositions $\mathcal{A}_j e^{s\widehat{\triangle}^q}$ as above, and the final composition by the localization operator $\chi_1$, which maps continuously to $H^k(\R^n)$ if $\alpha'$ and $\beta$ have been chosen large enough.
\end{proof}

\begin{lemma}
The operator $\mathcal{P}_{N+1}^{\varepsilon,\gamma}(t)$ is locally smoothing for $t\in[t_0,t_1]$.
\end{lemma}

\begin{proof}
It suffices to prove that $\chi_1\mathcal{I}_i^{\varepsilon,\gamma}(t)\chi_2$ and $\chi_1\mathcal{K}_i^{\varepsilon,\gamma}(t)\chi_2$ (see \eqref{def_mathcal_Jiepsgam} and \eqref{def_mathcal_Kiepsgam}) map continuously $H^j(\R^n)$ to $H^k(\R^n)$, with a norm that is uniformly bounded with respect to $t\in[t_0,t_1]$ and to $\varepsilon\in[-\varepsilon_0,\varepsilon_0]$. 
As pointed out at the beginning of Section \ref{sec_idea_eps}, the argument is similar to the one used in the proof of Lemma \ref{lem_Ci_locsmoothing} with the following differences:
\begin{itemize}
\item The second-order derivations $\mathcal{B}_{j}^{\varepsilon,\gamma}$ can be either equal to some $\mathcal{A}_i^{\varepsilon,\gamma}$, $i\in\{1,\ldots,N\}$, or to $\mathcal{R}_{N_1+1}^{\varepsilon,\gamma}$. In the latter case, this does not affect our previous reasoning because its coefficients have a growth at infinity that is at most polynomial of degree $r(q)$.
\item In the definition \eqref{def_mathcal_Kiepsgam} of $\mathcal{K}_i^{\varepsilon,\gamma}(t)$, the last term in the composition is $e^{s_1 \triangle^{\varepsilon,\gamma}}$.
\end{itemize}
Hence, the only main difference is in the last item: we have to explain how to deal with the composition by $e^{s_1 \triangle^{\varepsilon,\gamma}}$. There are two cases.

If $s_1\geq \frac{t_0}{N}$, then the term $\chi_1 e^{s_1 \triangle^{\varepsilon,\gamma}}$ is expected to act as a smoothing operator. By the above reasoning, the operator $\mathcal{B}_{1}^{\varepsilon,\gamma} \, e^{s_2 \widehat{\triangle}^{q}} \cdots  \mathcal{B}_i^{\varepsilon,\gamma} \, e^{s_{i+1} \widehat{\triangle}^{q}} \chi_2$ maps continuously $H^j(\R^n)$ to $H^{-\alpha}_{k_0\alpha}(\R^n)$ for some $\alpha>0$ and some $k_0\in\N$. Now, we apply to $\chi_1 e^{s_1 \triangle^{\varepsilon,\gamma}}$ the property \ref{globsmootheps} stated in Section \ref{sec_idea_eps} (itself following from Corollary \ref{cor_heat_uniform_global} in Appendix \ref{sec:globalregheat}), which gives the result.

Otherwise, another term $e^{s_p\widehat{\triangle}^q}$ acts as a smoothing operator, and thus, by the above reasoning, the operator $\mathcal{B}_{1}^{\varepsilon,\gamma} \, e^{s_2 \widehat{\triangle}^{q}} \cdots  \mathcal{B}_i^{\varepsilon,\gamma} \, e^{s_{i+1} \widehat{\triangle}^{q}} \chi_2$ maps continuously $H^j(\R^n)$ to $H^m_\beta(\R^n)$ for any $m\in\N$ and $\beta\geq 1$, and thus, using the global continuous embeddings \eqref{inclusions_global}, it maps $H^j(\R^n)$ to $\mathcal{D}^{\varepsilon,\gamma}_m$, for any $m\in\N$ arbitrarily large, uniformly with respect to $\varepsilon\in[-\varepsilon_0,\varepsilon_0]$. One then applies $e^{s_1\triangle^{\varepsilon,\gamma}}$, which, by \eqref{ineg_contraction_epsgamma}, maps continously $\mathcal{D}^{\varepsilon,\gamma}_m$ to $\mathcal{D}^{\varepsilon,\gamma}_m$ (uniformly with respect to $\varepsilon$), and the final application of the localization operator $\chi_1$ maps to $H^k(\R^n)$, by applying the uniform local subellipticity estimates (Theorem \ref{thm_subelliptic_uniform}), provided that $m$ be large enough.
\end{proof}

To conclude, it remains to establish the following lemma.

\begin{lemma}
We have
\begin{equation*}
\mathcal{C}_i^{\varepsilon,\gamma}(t) = \mathcal{C}_i(t) + \mathrm{O}(\vert\varepsilon\vert^\infty)
\end{equation*}
as $\varepsilon\rightarrow 0$ in $C^\infty$ topology, for every $i\in\{1,\ldots,N\}$, for every $t\in[t_0,t_1]$.
\end{lemma}

\begin{proof}
It suffices to prove that $\chi_1 ( \mathcal{I}_i^{\varepsilon,\gamma}(t)  - \mathcal{I}_i(t) )\chi_2 = \mathrm{O}(\vert\varepsilon\vert^\infty)$, i.e., that
\begin{equation}\label{int_diffA}
\int_{\Sigma_i(t)} \chi_1 e^{s_1 \widehat{\triangle}^{q}} ( \mathcal{A}_{j_1}^{\varepsilon,\gamma} - \mathcal{A}_{j_1}) \, e^{s_2 \widehat{\triangle}^{q}} \cdots ( \mathcal{A}_{j_i}^{\varepsilon,\gamma} - \mathcal{A}_{j_i}) \, e^{s_{i+1} \widehat{\triangle}^{q}} \chi_2 \, ds^{i+1} = \mathrm{O}(\vert\varepsilon\vert^\infty)
\end{equation}
as $\varepsilon\rightarrow 0$, for every $i\in\N^*$.
Recall that the second-order differential operators $\mathcal{A}_i$ and $\mathcal{A}_i^{\varepsilon,\gamma}$ are respectively defined by \eqref{def_Ai} and \eqref{def_Aiepsgam}. It follows from their definition that there exists $R>0$ such that
$$
\supp\big( ( \mathcal{A}_i^{\varepsilon,\gamma} - \mathcal{A}_i ) f \big) \subset \R^n \setminus\widehat{B}^q(0,R/\varepsilon^\gamma) \qquad \forall \varepsilon\in[-\varepsilon_0,\varepsilon_0]\qquad\forall f\in C^\infty(\R^n) .
$$
Besides, for $i=1,2$, let $R_i>0$ be such that $\supp(\chi_i)\subset \widehat{B}^q(0,R_i)$. Here, $\widehat{B}^q(0,R)$ denotes the sR ball in $\R^n$ of center $0$ and radius $R$, for the (nilpotent) sR structure associated with $\widehat{\triangle}^{q}$.
We are going to use the exponential estimates \eqref{upper_estim_kernel} given in Appendix \ref{sec:global_estimates_sR}.

The Sobolev regularity arguments are to those in the proof of Lemma \ref{lem_Ci_locsmoothing}, but we slightly modify those arguments with the following additional considerations.

Inside the integral \eqref{int_diffA}, at least one of the real numbers $s_p$ is such that $s_p\geq \frac{t_0}{N}$. We write $e^{s_p \widehat{\triangle}^{q}}=e^{\frac{s_p}{2} \widehat{\triangle}^{q}} e^{\frac{s_p}{2} \widehat{\triangle}^{q}}$. Let $\chi_3^\varepsilon$ be a function of compact support such that $\chi_3(x)=1$ on $\widehat{B}^q(0,R/\varepsilon^{\gamma/2})$ and $\chi_3(x)=0$ on $\R^n\setminus\widehat{B}^q(0,2R/\varepsilon^{\gamma/2})$. In other words, one has $\chi_3\simeq\mathbf{1}_{\widehat{B}^q(0,R/\varepsilon^{\gamma/2})}$. We write the operator inside the integral \eqref{int_diffA} as
$$
\chi_1\mathcal{D}_1^{\varepsilon,\gamma} e^{\frac{s_p}{2} \widehat{\triangle}^{q}} e^{\frac{s_p}{2} \widehat{\triangle}^{q}} \mathcal{D}_2^{\varepsilon,\gamma} \chi_2
=
\chi_1\mathcal{D}_1^{\varepsilon,\gamma} e^{\frac{s_p}{2} \widehat{\triangle}^{q}} \chi_3 e^{\frac{s_p}{2} \widehat{\triangle}^{q}} \mathcal{D}_2^{\varepsilon,\gamma} \chi_2
+ 
\chi_1\mathcal{D}_1^{\varepsilon,\gamma} e^{\frac{s_p}{2} \widehat{\triangle}^{q}} (1-\chi_3) e^{\frac{s_p}{2} \widehat{\triangle}^{q}} \mathcal{D}_2^{\varepsilon,\gamma} \chi_2
$$
where $\mathcal{D}_1^{\varepsilon,\gamma}$ and $\mathcal{D}_2^{\varepsilon,\gamma}$ are compositions of $e^{s_j \widehat{\triangle}^{q}}$ and of $\mathcal{A}_{j}^{\varepsilon,\gamma} - \mathcal{A}_{j}$. Using the exponential estimates \eqref{upper_estim_kernel} for heat kernels and their derivatives (see Appendix \ref{sec:global_estimates_sR}), and combining with the previous reasonings, we infer that, not only the operator $e^{\frac{s_p}{2} \widehat{\triangle}^{q}} \mathcal{D}_2^{\varepsilon,\gamma} \chi_2$ maps $H^j(\R^n)$ to $H^\alpha_\beta(\R^n)$ for any $\alpha$ and $\beta$ arbitrarily large, but also, its range is ``essentially concentrated" in $\widehat{B}^q(0,R/\varepsilon^{\gamma/2})$, in the sense that
$$
\Vert (1-\chi_3) e^{\frac{s_p}{2} \widehat{\triangle}^{q}} \mathcal{D}_2^{\varepsilon,\gamma} \chi_2 \Vert_{L( H^j(\R^n), H^\alpha_\beta(\R^n))} \leq C \, e^{-C/\vert\varepsilon\vert^\gamma} \qquad\forall\varepsilon\in[-\varepsilon_0,\varepsilon_0]
$$
with $C=\Cst(j,k,t_0,t_1,N)$. 
The exponential term comes from the fact that the support of $1-\chi_3$ is far from the support of $\chi_2$, at a distance of the order of $1/\vert\varepsilon\vert^{\gamma/2}$.
By a similar argument, we have
$$
\Vert ( \mathcal{A}_{j_p}^{\varepsilon,\gamma} - \mathcal{A}_{j_p}) \, e^{\frac{s_p}{2} \widehat{\triangle}^{q}} \chi_3 e^{\frac{s_p}{2} \widehat{\triangle}^{q}} \mathcal{D}_2^{\varepsilon,\gamma} \chi_2\Vert_{L( H^j(\R^n), H^\alpha_\beta(\R^n))} \leq C \, e^{-C/\vert\varepsilon\vert^\gamma} \qquad\forall\varepsilon\in[-\varepsilon_0,\varepsilon_0].
$$
The exponential term comes from the fact that the support of $\mathcal{A}_i^{\varepsilon,\gamma} - \mathcal{A}_i$ is far from the support of $\chi_3$, at a distance of the order of $1/\vert\varepsilon\vert^\gamma$.

All in all, combining with the arguments already used in the proof of Lemma \ref{lem_Ci_locsmoothing}, we obtain the lemma.
\end{proof}

Proposition \ref{prop_expansion_semigroup_epsgam} is proved.
\end{proof}


\subsection{Asymptotic expansion in $\varepsilon$ of the heat kernel}\label{taking_Skernels}
Taking the Schwartz kernels in the expansion \eqref{expansion_exp_triangle_eps}, more precisely, considering their densities with respect to the measure $\widehat{\mu}^{q}$, and recalling that $\widehat{e}^{q} = e_{\widehat{\triangle}^{q},\widehat{\mu}^{q}}$, we get that
\begin{equation}\label{expansion_kernel_eps_temp}
e_{\triangle^{\varepsilon,\gamma},\widehat{\mu}^{q}}(t,x,x') = \widehat{e}^{q}(t,x,x') + \sum_{i=1}^N \varepsilon^i [\mathcal{C}_i(t)]_{\widehat{\mu}^{q}}(x,x') +\mathrm{o}\big(\vert\varepsilon\vert^N\big) 
\end{equation}
at any order $N$, as $\varepsilon\rightarrow 0$, in $C^\infty((0,+\infty)\times\R^n\times\R^n)$ topology. For instance, we have
$$
[\mathcal{C}_1(t)]_{\widehat{\mu}^{q}}(x,x') = \int_0^t \int_{\R^n} \widehat{e}^{q}(t-s,x,z) \, ((\mathcal{A}_1)_x \widehat{e}^{q} ) (s,z,x') \, dz\, ds
$$
and all other (smooth) functions $[\mathcal{C}_i(t)]_{\widehat{\mu}^{q}}(x,x')$ can be expressed as well with convolutions. Now, using the formula \eqref{kernel_change_measure} in Appendix \ref{appendix_Schwartz}, we have
$$
e^{\varepsilon,\gamma}(t,x,x') = e_{\triangle^{\varepsilon,\gamma},\nu^{\varepsilon,\gamma}}(t,x,x') = \frac{1}{h^{\varepsilon,\gamma}(x')} e_{\triangle^{\varepsilon,\gamma},\widehat{\mu}^{q}}(t,x,x')
$$
where $h^{\varepsilon,\gamma} = \frac{d \nu^{\varepsilon,\gamma}}{d\widehat{\mu}^{q}}$ is the (smooth) density of $\nu^{\varepsilon,\gamma}$ with respect to $\widehat{\mu}^{q}$. By Lemma \ref{lem_epsgamma}, $h^{\varepsilon,\gamma}$ converges to $1$ uniformly on $\R^n$, and it depends smoothly on $\varepsilon$ in $C^\infty$ topology, hence $h^{\varepsilon,\gamma} = 1 + \varepsilon h^1 + \cdots + \varepsilon^Nh^N+\mathrm{o}\big(\vert\varepsilon\vert^N\big)$ at any order $N$, in $C^\infty$ topology. Using \eqref{expansion_kernel_eps_temp}, we conclude that
\begin{equation}\label{expansion_kernel_epsgamma}
e^{\varepsilon,\gamma}(t,x,x') = \widehat{e}^{q}(t,x,x') + \sum_{i=1}^N \varepsilon^i f_i^{q}(t,x,x') +\mathrm{o}\big(\vert\varepsilon\vert^N\big) 
\end{equation}
at any order $N$, as $\varepsilon\rightarrow 0$, in $C^\infty((0,+\infty)\times\R^n\times\R^n)$ topology, where the functions $f_i^{q}$ are smooth on $(0,+\infty)\times\R^n\times\R^n$.

Recall that 
$e^\varepsilon = e_{\triangle^\varepsilon,\nu^\varepsilon}$ 
where 
$\triangle^\varepsilon$ 
is defined by 
\eqref{def_triangleepsilon}. 

\begin{lemma}\label{lem_eepsgamma_eeps}
We have
$$
e^{\varepsilon,\gamma}(t,x,x') = e^\varepsilon(t,x,x') + \mathrm{O}(\vert\varepsilon\vert^\infty)
$$
as $\varepsilon\rightarrow 0$ in $C^\infty((0,+\infty)\times\R^n\times\R^n)$ topology.
\end{lemma}

\begin{proof}
Let $\beta>0$ small. By \eqref{homog_eepsgamma}, we have
$$
e^{\varepsilon,\gamma}(t,x,x') = \vert\varepsilon\vert^{\beta\mathcal{Q}(q)} e^{\varepsilon^{1-\beta},\gamma-\beta} ( \varepsilon^{2\beta} t, \delta_{\varepsilon^\beta}(x), \delta_{\varepsilon^\beta}(x') )  .
$$
Since $\triangle^{\varepsilon^{1-\beta},\gamma-\beta}$ coincides with $\triangle^{\varepsilon^{1-\beta}}$ on every compact, we infer from the localization result, Theorem \ref{theo:localheat} (Section \ref{sec_localheat}), that
$$
\vert\varepsilon\vert^{\beta\mathcal{Q}(q)} e^{\varepsilon^{1-\beta},\gamma-\beta} ( \varepsilon^{2\beta} t, \delta_{\varepsilon^\beta}(x), \delta_{\varepsilon^\beta}(x') )
= \vert\varepsilon\vert^{\beta\mathcal{Q}(q)} e^{\varepsilon^{1-\beta}}( \varepsilon^{2\beta} t, \delta_{\varepsilon^\beta}(x), \delta_{\varepsilon^\beta}(x') ) + \mathrm{O}(\vert\varepsilon\vert^\infty)
$$
as $\varepsilon\rightarrow 0$, on every compact subset of $(0,+\infty)\times\R^n\times\R^n$. 
Besides, it follows from the homogeneity property \eqref{homog_deltaeps} that
$e^\varepsilon(t,x,x') = \lambda^{\mathcal{Q}(q)} e^{\varepsilon/\lambda}(\lambda^2 t,\delta_\lambda(x),\delta_\lambda(x'))$ for $\lambda>0$, and taking $\lambda=\varepsilon^{2\beta}$, we obtain
$$
\vert\varepsilon\vert^{\beta\mathcal{Q}(q)} e^{\varepsilon^{1-\beta}}( \varepsilon^{2\beta} t, \delta_{\varepsilon^\beta}(x), \delta_{\varepsilon^\beta}(x') )
= e^\varepsilon(t,x,x').
$$
The lemma follows.
\end{proof}

We note again, in the above proof, the crucial role of the parameter $\gamma$ in our construction. Here, it has been instrumental to be able to apply the localization result, and thus show that the small-time asymptotics of $e^{\varepsilon,\gamma}$ and of $e^\varepsilon$ coincide at the infinite order.

Using \eqref{expansion_kernel_epsgamma} and Lemma \ref{lem_eepsgamma_eeps}, we finally obtain that
\begin{equation*}
e^\varepsilon(t,x,x') = \widehat{e}^{q}(t,x,x') + \sum_{i=1}^N \varepsilon^i f_i^{q}(t,x,x') +\mathrm{o}\big(\vert\varepsilon\vert^N\big) 
\end{equation*}
at any order $N$, as $\varepsilon\rightarrow 0$, in $C^\infty((0,+\infty)\times\R^n\times\R^n)$ topology.

\subsection{End of the proof}\label{end_of_the_proof_main_thm}
The end of the proof is now similar to the proof of Theorem \ref{main_thm_weak}. Thanks to repeated applications of the localization result, Theorem \ref{theo:localheat} (hypoelliptic Kac's principle), we finally obtain
$$
\vert\varepsilon\vert^{\mathcal{Q}(q)} e_{\triangle,\mu}(\varepsilon^2t,\delta_\varepsilon(x),\delta_\varepsilon(x')) = \widehat{e}^{q}(t,x,x') + \sum_{i=1}^N \varepsilon^i f_i^{q}(t,x,x') +\mathrm{o}\big(\vert\varepsilon\vert^N\big) 
$$
at any order $N$, as $\varepsilon\rightarrow 0$, in $C^\infty$ topology, which is exactly \eqref{complete_expansion}.

Let us next establish the homogeneity property for the (smooth) functions $f^{q}_i$.

\begin{lemma}\label{lem_fi}
For every $i\in\N^*$ and every $\varepsilon\neq 0$, we have
$$
f^{q}_i(t,x,x') = \varepsilon^{-i} \vert\varepsilon\vert^{\mathcal{Q}(q)} f^{q}_i(\varepsilon^2 t,\delta_\varepsilon(x),\delta_\varepsilon(x'))  
$$
for all $(t,x,x')\in(0,+\infty)\times\R^n\times\R^n$.
\end{lemma}

\begin{proof}
Given any fixed $\varepsilon\neq 0$, by \eqref{complete_expansion}, we have on the one part
$$
\vert s\varepsilon\vert^{\mathcal{Q}(q)} e_{\triangle,\mu} ( s^2\varepsilon^2 t, \delta_{s\varepsilon}(x), \delta_{s\varepsilon}(x') ) = \widehat{e}^{q}(t,x,x') + \sum_{i=1}^N (s\varepsilon)^i f^{q}_i(t,x,x') +\mathrm{o}(s^N) 
$$
and on the other part
\begin{multline*}
\vert s\varepsilon\vert^{\mathcal{Q}(q)} e_{\triangle,\mu} ( s^2\varepsilon^2 t, \delta_{s\varepsilon}(x), \delta_{s\varepsilon}(x') ) = \vert \varepsilon\vert^{\mathcal{Q}(q)} \widehat{e}^{q}(\varepsilon^2 t,\delta_\varepsilon(x),\delta_\varepsilon(x')) \\
+ \vert \varepsilon\vert^{\mathcal{Q}(q)} \sum_{i=1}^N s^i f^{q}_i(\varepsilon^2 t,\delta_\varepsilon(x),\delta_\varepsilon(x')) +\mathrm{o}(s^N) 
\end{multline*}
for every $s\in\R$ with $\vert s\vert$ sufficiently small and all $t>0$ and $x,x'\in\R^n$. The result follows, since $\widehat{e}^{q}(t,x,x') = \vert \varepsilon\vert^{\mathcal{Q}(q)} \widehat{e}^{q}(\varepsilon^2 t,\delta_\varepsilon(x),\delta_\varepsilon(x'))$ (see \eqref{homog_nilp}).
\end{proof}

\begin{remark}
Applying Lemma \ref{lem_fi} with $\varepsilon=-1$, we obtain the ``oddness" property 
$$
f^{q}_{2j-1}(t,x,x') = -f^{q}_{2j-1}(t,\delta_{-1}(x),\delta_{-1}(x')) \qquad\forall j\in\N^*\qquad\forall (t,x,x')\in(0,+\infty)\times\R^n\times\R^n .
$$
In particular, we have $f^{q}_{2j-1}(t,0,0) = -f^{q}_{2j-1}(t,0,0)$ and thus $f^{q}_{2j-1}(t,0,0)=0$ for every $j\in\N^*$ and for every $t>0$. Note that, to obtain this property, it has been necessary to consider dilations $\delta_\varepsilon$ with $\varepsilon<0$, whereas dilations are most often considered only with $\varepsilon>0$ in the existing literature. We realized this fact in a discussion with Davide Barilari, whom we thank once again.
\end{remark}

\paragraph{Case $q$ regular.}
The case is treated exactly as in Section \ref{sec_proof_weak_main_thm}. We do not repeat the argument.

\part{Appendix}\label{part_appendix}

\appendix

\section{Schwartz kernels, heat kernels}\label{appendix_Schwartz}
Let $M$ be a smooth manifold. We set $\mathcal{D}(M)=C_c^\infty(M)$ and we denote by $\mathcal{D}'(M)$ the  space of distributions on $M$, i.e., the topological dual of $\mathcal{D}(M)$ endowed with the weak topology.
Let $\mu$ be a smooth measure (density) on $M$.

\paragraph{Schwartz kernels.}
According to the Schwartz kernel theorem, there is a linear bijection between $\mathcal{D}'(M\times M)$ and the set of bilinear continuous functionals on $\mathcal{D}(M)\times \mathcal{D}(M)$.
Given a linear continuous mapping $A: \mathcal{D}(M)\rightarrow \mathcal{D}'(M)$, the Schwartz kernel of $A$ is the unique distribution $[A]\in \mathcal{D}'(M\times M)$ defined by $\langle Af,g\rangle_{\mathcal{D}'(M),\mathcal{D}(M)} = \langle [A],g\otimes f\rangle$ for all $f,g\in C_c^\infty(M)$, where $\langle\cdot,\cdot\rangle$ is the duality bracket.

When $[A]\in C^0(M\times M)$, identifying the distribution bracket by an integral with respect to the measure $\mu\otimes\mu$ and denoting by $[A]_\mu$ the density function, we have the familiar formula 
$$
Af(q) = \int_M [A]_\mu(q,q') f(q')\, d\mu(q') \qquad \forall q\in M\qquad \forall f\in\mathcal{D}(M).
$$
We stress that, although the density function $[A]_\mu$ depends on $\mu$, given any $q\in M$, the absolutely continuous measure $[A]_\mu(q,\cdot)\, d\mu(\cdot)$ depends only on $A$: it does not depend on the smooth measure $\mu$, in the sense that $[A]_\mu(q,\cdot)\, d\mu(\cdot) = [A]_\nu(q,\cdot)\, d\nu(\cdot)$ for any other smooth measure $\nu$ on $M$.

Actually, in geometric terms, $[A]$ is a continuous section of the bundle $\pi_2^*(\Omega_M)$ on $M\times M$, where $\Omega_M$ is the line bundle of smooth measures (densities) on $M$ and $\pi_2:M\times M\rightarrow M$ is the projection defined by $\pi_2(q,q')=q'$.

Similarly, the diagonal part $[A]_\mu(q,q)\, d\mu(q)$ is an absolutely continuous measure, which does not depend on $\mu$. Denoting by $\mathcal{M}_f$ the operator of multiplication by $f$, we have
$$
\mathrm{Tr}(A \mathcal{M}_f) = \int_M [A]_\mu(q,q) f(q)\, d\mu(q) \qquad\forall f\in \mathcal{D}(M).
$$

\paragraph{Hilbert-Schmidt norm.}
When $[A]\in L^2(M\times M,\mu\otimes\mu)$, the operator $A\in L(L^2(M,\mu))$ is Hilbert-Schmidt and the Hilbert-Schmidt norm of $A$ is 
$$
\Vert A\Vert_{HS} = (\mathrm{Tr}(A^*A))^{1/2} = \Vert [A]_\mu\Vert_{L^2(M\times M,\mu\otimes\mu)} 
$$
and we recall that 
\begin{equation*}
\begin{split}
& \Vert A\Vert_{L(L^2(M,\mu))} \leq \Vert A\Vert_{HS} \\
& \Vert AB\Vert_{HS} \leq \Vert A\Vert_{L(L^2(M,\mu))} \Vert B\Vert_{HS} \\
& \Vert AB\Vert_{HS} \leq \Vert A\Vert_{HS} \Vert B\Vert_{L(L^2(M,\mu))} 
\end{split}
\end{equation*}
where $A$ and $B$ are bounded operators on $L^2(M,\mu)$, with $A$ or $B$ Hilbert-Schmidt according to the inequality under consideration.

\paragraph{Action of pseudo-differential operators on Schwartz kernels.}
Given any pseudo-differential operators $T_1$ and $T_2$ on $M$, we have $[T_1AT_2^*]_\mu = (T_1)_q (T_2)_{q'}[A]_\mu$ (where $T_2^*$ is the transpose in $L^2(M,\mu)$), i.e.,
$$
T_1AT_2^* f(q) = \int_M (T_1)_q (T_2)_{q'}[A]_\mu (q,q') f(q')\, d\mu(q') .
$$


\paragraph{Heat kernels.}
Let $A:D(A)\rightarrow L^2(M,\mu)$ be a densely defined operator on $L^2(M,\mu)$, generating a strongly continuous semigroup $(e^{tA})_{t\geq 0}$. 
For every $t>0$, the heat kernel $e_A(t)$ associated with $A$ is the measure on $M$ defined as the Schwartz kernel of $e^{tA}$, i.e., $e_A(t)=[e^{tA}]$. 
Of course, it does not depend on $\mu$.


When this measure has a density $[e^{tA}]_\mu$ with respect to $\mu$ which is locally integrable, we define the heat kernel $e_{A,\mu}(t,\cdot,\cdot)$ associated with $A$ and with the measure $\mu$ by $e_{A,\mu}(t,q,q') = [e^{tA}]_\mu(q,q')$. This means that 
$$
u(t,q) = (e^{tA}f)(q) = \int_M e_{A,\mu}(t,q,q')f(q')\, d\mu(q')
$$
is the unique solution to $\partial_t u-Au=0$ for $t>0$, $u(0,\cdot)=f(\cdot)$, for every $f\in C_c^\infty(M)$.
In other words, we have
$$
e_A(t)(q,q') = [e^{tA}](q,q') = e_{A,\mu}(t,q,q')\, d\mu(q') \qquad\forall t>0\qquad \forall q,q'\in M.
$$
As said above, this expression depends only on $A$, not on the smooth measure $\mu$.

\medskip

Extending $e_{A,\mu}$ by $0$ for $t<0$, for any fixed $q'\in M$ the mapping $(t,q)\mapsto e_{A,\mu}(t,q,q')$ is also solution of $(\partial_t-A)e_{A,\mu}(\cdot,\cdot,q')=\delta_{(0,q')}$ in the sense of distributions, where the distribution pairing is considered with respect to the measure $dt\times d\mu(q)$ on $\R\times M$. 

We gather hereafter some useful facts.

\begin{itemize}
\item
Let $\varphi:M\rightarrow M$ be a diffeomorphism, representing a change of variable in the manifold $M$. We have $\varphi^*\mu = \vert J_\mu(\varphi)\vert\mu$, where $J_\mu(\varphi)$ is the Jacobian of $\varphi$ with respect to $\mu$, and where $\varphi^*\mu$ is the pullback of $\mu$ under $\varphi$.
Then
\begin{equation}\label{formulas_kernel}
\begin{split}
e_{\varphi^*A\varphi_*,\mu}(t,q,q') &= \vert J_\mu(\varphi)(q')\vert\, e_{A,\mu}(t,\varphi(q),\varphi(q'))  \\
e_{A,\varphi^*\mu}(t,q,q') &= \frac{1}{\vert J_\mu(\varphi)(q')\vert}\, e_{A,\mu}(t,q,q')  \\
e_{\varphi^*A\varphi_*,\varphi^*\mu}(t,q,q') &= e_{A,\mu}(t,\varphi(q),\varphi(q'))
\end{split}
\end{equation}
for all $t>0$ and $(q,q')\in M^2$.
Note that the last one follows from the two first ones, in which we have replaced $A$ with $\varphi^*A\varphi_*$ in the second one.
The two first formulas in \eqref{formulas_kernel} are not symmetric, but there is no contradiction there: indeed if $A$ is selfadjoint in $L^2(M,\mu)$ then $e_{A,\mu}$ is symmetric, but $A$ need not be selfadjoint in $L^2(M,\varphi^*\mu)$ and thus $e_{A,\varphi^*\mu}$ need not be symmetric.

Actually, we have
$$
e_{\varphi^*A\varphi_*,\mu}(t,q,q')\, d(\varphi^*\mu)(q') = e_{\varphi^*A\varphi_*,\nu}(t,q,q')\, d\nu(q') 
$$
for any other smooth measure $\nu$ on $M$.

\item As a particular case, given any $\lambda>0$, we have
$\lambda\, e_{A,\lambda\mu} = e_{A,\mu} $.

\item
Given any $\varepsilon>0$, the kernel associated with $\varepsilon^2 A$ and with the measure $\nu$ is 
\begin{equation*}
e_{\varepsilon^2 A, \nu}(t,q,q') = e_{A,\nu}(\varepsilon^2 t,q,q')
\end{equation*}
for all $t>0$ and $(q,q')\in M^2$.

\item
We assume that $\mu=h\nu$ with $h$ a positive smooth function on $M$ (density of $\mu$ with respect to $\nu$). Then
$h(q')e_{A,\mu}(t,q,q')=e_{A,\nu}(t,q,q')$
for all $(t,q,q')\in(0,+\infty)\times M\times M$, or equivalently,
\begin{equation}\label{kernel_change_measure}
e_{A,\mu}(t,q,q')\, d\mu(q') = e_{A,\nu}(t,q,q')\, d\nu(q').
\end{equation}


\end{itemize}

\section{Subelliptic estimates and smoothing properties for hypoelliptic heat semigroups}\label{sec:uniform_subelliptic_estimates}
This section can be read independently of the rest.

Let $d\in\N^*$. We set $\Lambda = \left(\mathrm{id}-\sum_{i=1}^d \partial_i^2\right)^{1/2}$. Denoting by $\Vert u\Vert_{L^2(\R^d)}$ the $L^2$ norm of a function $u\in L^2(\R^d)$ and by $\langle\cdot,\cdot\rangle_{L^2(\R^d)}$ the corresponding inner product, we recall that the Hilbert space $H^s(\R^d)$ is equipped with the norm $\Vert u\Vert_{H^s(\R^d)} = \Vert \Lambda^s u\Vert_{L^2(\R^d)}$ and with the inner product $\langle u,v\rangle_{H^s(\R^d)} = \langle u,\Lambda^{2s}v\rangle_{L^2(\R^d)}$.

Let $p\in\N^*$ and let $\mathcal{K}$ be a compact set (in our applications, we will take either $\mathcal{K}=[-\varepsilon_0,\varepsilon_0]$ or $\mathcal{K}=[-\varepsilon_0,\varepsilon_0]\times K$ for some $\varepsilon_0>0$ and for some compact subset $K$ of $M$). For every $\tau\in\mathcal{K}$, let $Y_0^\tau,Y_1^\tau,\ldots,Y_p^\tau$ be smooth vector fields on $\R^d$ and let $\mathbb{V}^\tau$ be a smooth function on $\R^d$, all of them depending continuously on $\tau$ in $C^\infty$ topology. 
We set
\begin{equation} \label{def_Ltau}
L^\tau = \sum_{i=1}^p (Y_i^\tau)^2 + Y_0^\tau - \mathbb{V}^\tau . 
\end{equation}

\subsection{Local estimates}\label{sec:local_estimates}
Let $U$ be an arbitrary open subset of $\R^d$.
Of course, all local estimates hereafter could be settled as well on a $d$-dimensional manifold.

\subsubsection{Uniform local subelliptic estimates}\label{sec:uniform_local_subelliptic_estimates}

\begin{theorem}\label{thm_subelliptic_uniform}
We assume that the Lie algebra $\mathrm{Lie}(Y_0^\tau,Y_1^\tau,\ldots,Y_p^\tau)$ generated by the vector fields is equal to $\R^d$ at any point of $U$, with a degree of nonholonomy $r$ that is uniform with respect to $\tau\in\mathcal{K}$ (\emph{uniform weak H\"ormander condition on $U$}).
Then there exists $\sigma>0$ such that, for every $s\in\R$, for all smooth functions $\zeta$ and $\zeta'$ compactly supported in $U$ with $\zeta'=1$ on the support of $\zeta$, 
$$
\Vert \zeta u\Vert_{H^{s+\sigma}(\R^d)} \leq \Cst(s,\zeta,\zeta') \left( \Vert \zeta' L^\tau u\Vert_{H^s(\R^d)} + \Vert \zeta' u \Vert_{H^s(\R^d)} \right) \qquad \forall u\in C^\infty(\R^d)\qquad \forall \tau\in\mathcal{K} .
$$
\end{theorem}

This is a parameter-dependent version of the famous local subelliptic H\"ormander estimate (see \cite{Ho-67}). It is straightfoward to obtain by following the classical proofs, but for completeness, we provide the main steps (this is also useful in view of deriving global subelliptic estimates, further). The proof given in \cite{Ho-67} gives an optimal gain of regularity, which is $\sigma=2/r$. Here, we rather follow the simpler proof given by Kohn in \cite{Kohn1973} (see also \cite{HelfferNier}), which does not give the optimal gain of regularity and yields $\sigma=1/2^{2r-1}$. 

Note that, in the case where $\overline{U}$ is compact, Theorem \ref{thm_subelliptic_uniform} implies that there exists $C>0$ such that
$\Cst\, \Lambda^{\sigma} \leq C\,\mathrm{id} - L^\tau - (L^\tau)^* \leq \Cst\, \Lambda^2$,
for every $\tau\in\mathcal{K}$, where the inequalities are written for positive selfadjoint operators on $L^2(U)$.

Taking $s\in\{s,s+\sigma,\ldots,s+k\sigma\}$, we also infer from the theorem that, for every $k\in\N$, for all smooth functions $\zeta$ and $\zeta'$ compactly supported in $U$ with $\zeta'=1$ on the support of $\zeta$, we have
\begin{multline}\label{thm_subelliptic_uniform_cor}
\Vert \zeta u\Vert_{H^{k\sigma}(\R^d)} \leq \Cst(k,\zeta,\zeta') \left( \Vert \zeta' (L^\tau)^k u\Vert_{L^2(\R^d)} + \cdots+\Vert \zeta' L^\tau u\Vert_{L^2(\R^d)} + \Vert \zeta' u \Vert_{L^2(\R^d)} \right)  \\
\forall u\in C^\infty(\R^d)\qquad \forall \tau\in\mathcal{K} .
\end{multline}

Another useful consequence is the following.
For every $k\in\Z$, we define the Hilbert space $\mathcal{D}^\tau_k$ as the completion for the norm $\Vert u\Vert_{\mathcal{D}^\tau_k} = \Vert (C\,\mathrm{id} - L^\tau - (L^\tau)^*)^k u\Vert_{L^2(\R^d)}$ of the set of smooth functions on $\R^d$ of compact support.
For $k=1$, the set $\mathcal{D}^\tau_1=\{ u\in L^2(\R^d)\ \mid\ L^\tau u\in L^2(\R^d)\}$ is the maximal domain of the operator $L^\tau$ on $L^2(\R^d)$.
It follows from \eqref{thm_subelliptic_uniform_cor} that, for every $k\in\N$, for all smooth functions $\zeta$ and $\zeta'$ compactly supported in $U$ with $\zeta'=1$ on the support of $\zeta$, we have
\begin{equation*}
\Cst(k,\zeta,\zeta') \Vert \zeta u \Vert_{H^{k\sigma}(\R^d)}  \leq \Vert \zeta' u \Vert_{\mathcal{D}^\tau_k }  \leq \Cst(k,\zeta,\zeta') \Vert \zeta' u \Vert_{H^{2k}(\R^d)} \qquad \forall u\in C^\infty(\R^d)\qquad \forall \tau\in\mathcal{K} .
\end{equation*}
By duality, we have the converse inequalities for $k$ negative.

\begin{remark}
Note that, in the statement of Theorem \ref{thm_subelliptic_uniform}, we use cut-off functions $\zeta$ and $\zeta'$ (as in \cite{Kohn1973}), in order to obtain the subelliptic estimates for any smooth function $u$ on $\R^d$, and not only for any smooth function of compact support on $U$. This is because, in Appendix \ref{sec:localregheat}, we are going to apply these estimates to solutions $u(t)=e^{t L^\tau}f$ of the heat equation $(\partial_t-L^\tau)u=0$, which are not of compact support.
\end{remark}

\begin{proof}[Proof of Theorem \ref{thm_subelliptic_uniform}.]
We only recall the main steps of the proof, without providing all details.
We modify the function $\mathbb{V}^\tau$ and the vector fields $Y_j^\tau$ (and thus $L^\tau$) outside of $U$ so that their coefficients and all their derivatives are uniformly bounded on $\R^d$, uniformly with respect to $\tau$. Localization is performed by bracketting (see the localization lemma in \cite{Kohn1973}).

First of all, we note that $(Y_j^\tau)^*=-Y_j^\tau-(\mathrm{div}\, Y_j^\tau)$, 
where the dual is taken in $L^2(\R^d)$. It follows that, given any smooth function $g$ on $\R^d$, we have
$$
\vert\langle Y_j^\tau u,gu\rangle_{L^2(\R^d)}\vert = \frac{1}{2} \langle u, (g\,\mathrm{div}\, Y_j^\tau+Y_j^\tau g)u\rangle_{L^2(\R^d)}\vert \leq \Cst(g) \Vert u\Vert_{L^2(\R^d)}^2  
\qquad \forall u\in C_c^\infty(\R^d)\qquad \forall \tau\in\mathcal{K} .
$$
We easily infer that
\begin{multline*}
\sum_{i=1}^p \Vert Y_j^\tau u\Vert_{L^2(\R^d)}^2 \leq \Cst\left( \left\vert\langle L^\tau u,u\rangle_{L^2(\R^d)}\right\vert+\Vert u\Vert_{L^2(\R^d)}^2\right) \leq \Cst \left( \Vert L^\tau u\Vert_{L^2(\R^d)}^2+\Vert u\Vert_{L^2(\R^d)}^2 \right) \\
\forall u\in C_c^\infty(\R^d) \qquad \forall \tau\in\mathcal{K} .
\end{multline*}
For every $\sigma\in(0,1]$, let $\mathcal{P}_\sigma$ be the set of pseudo-differential operators $P$ of order $0$ satisfying
$$
\Vert Pu\Vert_{H^\sigma(\R^d)} \leq \Cst(\sigma) \left( \Vert L^\tau u\Vert_{L^2(\R^d)} + \Vert u \Vert_{L^2(\R^d)}  \right) 
\qquad \forall u\in C_c^\infty(\R^d) \qquad \forall \tau\in\mathcal{K} .
$$
for every smooth function $u$ on $\R^d$ of compact support. Note that $0\leq \sigma_1\leq \sigma_2$ implies $\mathcal{P}_{\sigma_2}\subset\mathcal{P}_{\sigma_1}$. The proof consists of proving that:
\begin{enumerate}
\item $\mathcal{P}_\sigma$ is stable by taking the adjoint, for every $\sigma\in(0,1/2]$;
\item $\mathcal{P}_\sigma$ is a left and right ideal in the ring of pseudo-differential operators of order $0$;
\item $Y_j^\tau\Lambda^{-1}\in\mathcal{P}_\sigma$ for every $j\in\{0,1,\ldots,p\}$, for every $\sigma\in(0,1/2]$;
\item if $P\in \mathcal{P}_\sigma$ then $[Y_j^\tau,P]\in \mathcal{P}_{\sigma/4}$, for every $j\in\{0,1,\ldots,p\}$ and for every $\sigma\in(0,1/2]$.
\end{enumerate}
Each of the steps is quite straightforward to establish, following \cite{HelfferNier,Kohn1973} and using elementary properties of brackets and of pseudo-differential operators.
Then, noting that
$
[Y_j^\tau,Y_k^\tau]\Lambda^{-1} = [Y_j^\tau,Y_k^\tau\Lambda^{-1}] - Y_k^\tau\Lambda^{-1} \Lambda[Y_j^\tau,\Lambda^{-1}],
$
using that $Y_k^\tau\Lambda^{-1}\in\mathcal{P}_{1/2}$ by the third property and thus that $[Y_j^\tau,Y_k^\tau\Lambda^{-1}]\in\mathcal{P}_{1/8}$, using that $\Lambda[Y_j^\tau,\Lambda^{-1}]$ is a pseudo-differential operator of order $0$ and using the first and the third properties, we have $Y_k^\tau\Lambda^{-1} \Lambda[Y_j^\tau,\Lambda^{-1}]\in\mathcal{P}_{1/2}$, and we infer that $[Y_j^\tau,Y_k^\tau]\Lambda^{-1}\in\mathcal{P}_{1/8}$. By recurrence, we get that $Y_I^\tau\Lambda^{-1}\in\mathcal{P}_{1/2^{2k-1}}$, where $Y_I^\tau$ is any Lie bracket of the vector fields $Y_j^\tau$ of length $k$.

Using the uniform H\"ormander condition, it follows that $\mathcal{P}_{1/2^{2r-1}}$ coincides with the ring of pseudo-differential operators of order $0$. 
Indeed, we have proved that $\Lambda^{-1} \partial_{x_i} \in \mathcal{P}_{1/2^{2k-1}}$ for every $i\in\{1,\ldots,d\}$; using the properties of ideal and of stability by taking the adjoint, we have then $\Lambda^{-1} \partial_{x_i}^2 \Lambda^{-1} \in \mathcal{P}_{1/2^{2k-1}}$ and thus $\Lambda^{-2} \partial_{x_i}^2 = \Lambda^{-1} \partial_{x_i}^2 \Lambda^{-1} +  \Lambda^{-1} [ \Lambda^{-1} \partial_{x_i}^2 \Lambda^{-1}, \Lambda ]  \in \mathcal{P}_{1/2^{2k-1}}$, and finally, $\mathrm{id} \in \mathcal{P}_{1/2^{2k-1}}$.
The theorem is thus proved for $s=0$.

We then prove the result for any $s\in\R$ by applying the inequality established for $s=0$ to $\Lambda^s u$ and by managing the bracket $[L^\tau,\Lambda^s]$, using in particular the fact that $[(Y_i^\tau)^2,\Lambda^s] = 2[Y_i^\tau,\Lambda^s]Y_i^\tau+[Y_i^\tau,[Y_i^\tau,\Lambda^s]]$.
\end{proof}


\subsubsection{Application: uniform local smoothing property for heat semigroups}\label{sec:localregheat}
Throughout the section, we assume that the operator $L^\tau$ defined by \eqref{def_Ltau} on the domain $\mathcal{D}^\tau_1$ generates a strongly continuous semigroup $(e^{tL^\tau})_{t\geq 0}$ on $L^2(\R^d)$, for every $\tau\in\mathcal{K}$, satisfying the following uniform estimate: for all positive real numbers $t_0<t_1$ we have 
\begin{equation}\label{uniformestimatesemigroup0}
\Vert e^{tL^\tau} \Vert_{L(L^2(\R^d))}\leq\Cst(t_0,t_1) \qquad\forall t\in[t_0,t_1]\qquad \forall \tau\in\mathcal{K}.
\end{equation}
This is the case if $\mathbb{V}^\tau$ is uniformly bounded below on $\R^d$ and if either $L^\tau$ is selfadjoint for any $\tau$ or the vector fields $Y_0^\tau,Y_1^\tau,\ldots,Y_p^\tau$ are bounded on $\R^d$ as well as their first derivatives uniformly with respect to $\tau$. Indeed, under these assumptions, reasoning as in the proof of Lemma \ref{lem_semigroup}, there exists $C\geq 0$ such that the operator $L^\tau-C\,\mathrm{id}$ on $L^2(\R^d)$ is closed and dissipative, as well as its adjoint, and thus the strongly continuous semigroup $(e^{t(L^\tau-C\,\mathrm{id})})_{t\geq 0}$ that it generates is a contraction semigroup, and then $(e^{tL^\tau})_{t\geq 0}$ is a strongly continuous semigroup, with uniform norm estimates.

We denote by $e^\tau$ the 
heat kernel associated with $L^\tau$ for the Lebesgue measure on $\R^d$.
The function $e^\tau$ is defined on $(0,+\infty)\times\R^d\times\R^d$ and depends on three variables $(t,x,y)$. In what follows, the notation $\partial_1$ (resp., $\partial_2$, $\partial_3$) denotes the partial derivative with respect to $t$ (resp., to $x$, to $y$). The indices $x$ in $\chi_1$ and $y$ in $\chi_2$ below mean that, when $e^\tau$ is taken at $(t,x,y)$, $\chi_1$ is taken at $x$ and $\chi_2$ is taken at $y$.

\begin{corollary}\label{cor_heat_uniform}
We assume that the Lie algebra $\mathrm{Lie}(Y_1^\tau,\ldots,Y_p^\tau)$ generated by the vector fields $Y_i^\tau$, $i=1,\ldots,p$, is equal to $\R^d$ at any point of $U$, with a degree of nonholonomy $r$ that is uniform with respect to $\tau\in\mathcal{K}$ (\emph{uniform strong H\"ormander condition on $U$}).
Note that this assumption is more restrictive\footnote{This more restrictive assumption ensures that $L^\tau-\partial_t$ is hypoelliptic. 
Otherwise hypoellipticity may fail: take $Y_0=\partial_{x_1}$ and all $Y_i$, $i\geq 1$, not depending on $x_1$: then $(Y_0,Y_1,\ldots,Y_n)$ satisfies the H\"ormander condition, but not $(\partial_t+Y_0,Y_1,\ldots,Y_n)$.} than in Appendix \ref{sec:uniform_local_subelliptic_estimates}, because we have excluded $Y_0^\tau$.

For all $\chi_1,\chi_2\in C_c^\infty(U)$, for all positive real numbers $0<t_0<t_1$, for all $s,s'\in\R$, for all $(k,\alpha,\beta)\in\N\times\N^d\times\N^d$, we have
\begin{equation}\label{cor_heat_uniform_e}
\left\Vert (\chi_1)_x (\chi_2)_y \partial_1^k \partial_2^\alpha \partial_3^\beta \, e^\tau(\cdot,\cdot,\cdot) \right\Vert_{L^\infty((t_0,t_1)\times\R^d\times\R^d)}    \leq   \Cst(\chi_1,\chi_2,t_0,t_1,k,\alpha,\beta)
\qquad \forall \tau\in\mathcal{K}
\end{equation}
and
$$
\Vert \chi_1 e^{t L^\tau}\chi_2 f\Vert_{H^s(\R^d)}\leq \Cst(\chi_1,\chi_2,t_0,t_1,s,s') \Vert f\Vert_{H^{s'}(\R^d)}
\qquad\forall t\in[t_0,t_1]\quad \forall f\in C_c^\infty(\R^d)\quad \forall \tau\in\mathcal{K}.
$$
\end{corollary}

The corollary says that the family $(e^\tau)_{\tau\in\mathcal{K}}$ is bounded in $C^{\infty}((0,+\infty)\times\R^n\times\R^n)$, uniformly with respect to $\tau$ (for the Fr\'echet topology defined by seminorms on an exhaustive sequence of compact subsets of $(0,+\infty)\times\R^n\times\R^n$).
It also states that, with the above notations, $\chi_1 e^{t L^\tau} \chi_2 \in L(H^{s'}(\R^d),H^s(\R^d))$ for every $t>0$, where the norm of this operator can be bounded by a positive constant depending on $\chi_1$ and $\chi_2$ but not depending on $\tau$.

\begin{proof}
Defining on $\R\times\R^d\times\R^d$ the vector fields 
$$
\tilde Y_0^\tau=\begin{pmatrix} -2\\ Y_0^\tau\\ (Y_0^\tau)^*\end{pmatrix},\qquad
\tilde Y_i^\tau=\begin{pmatrix} 0 \\ Y_i^\tau\\ (Y_i^\tau)^*\end{pmatrix},\quad i=1,\ldots,p,
$$
we consider the operator 
$$
P^\tau = (L^\tau)_x + (L^\tau)^*_{x'} - 2 \partial_t=\sum_{i=1}^p (\tilde Y_i^\tau)^2+\tilde Y_0^\tau - \mathbb{V}^\tau 
$$
on $L^2(\R\times\R^d\times\R^d)$, where $(L^\tau)_x = L^\tau\otimes\mathrm{id}$ and $(L^\tau)^*_{x'}=\mathrm{id}\otimes (L^\tau)^*$ are differential operators acting on functions $g=g(x,x')$.
The uniform weak H\"ormander condition is satisfied on $U$ for the $p+1$ vector fields $(\tilde Y_0^\tau,\tilde Y_1^\tau,\ldots,\tilde Y_p^\tau)$ on $\R^{1+2d}$. 
Applying Theorem \ref{thm_subelliptic_uniform} to the heat kernel $e^\tau$, which satisfies $P^\tau e^\tau=0$ on $(0,+\infty)\times U\times U$, 
we obtain that $\Vert\zeta e^\tau\Vert_{H^{s+\sigma}(\R^{1+2d})} \leq\Cst(s,\zeta,\zeta')\Vert \zeta' e^\tau\Vert_{H^s(\R^{1+2d})}$, and by recurrence, 
$$
\Vert\zeta e^\tau\Vert_{H^{s+k\sigma}(\R^{1+2d})} \leq \Cst(s,k,\zeta,\zeta') \Vert \zeta' e^\tau\Vert_{H^s(\R^{1+2d})} \qquad\forall\tau\in\mathcal{K}\quad \forall k\in\N^* \quad \forall s\in\R
$$
for all smooth functions $\zeta$ and $\zeta'$ on $(0,+\infty)\times U\times U$ with compact support, satisfying $\zeta'=1$ on the support of $\zeta$. In what follows, we take $\zeta=\chi_0\otimes\chi_1\otimes\chi_2$ and $\zeta'=\chi_0'\otimes\chi_1'\otimes\chi_2'$.
In order to obtain \eqref{cor_heat_uniform_e}, we need an initialization, which is given in the following lemma.

\begin{lemma}\label{lemma_init}
There exists $s\in\R$ such that $\Vert (\chi_0')_t (\chi_1')_x (\chi_2')_{x'} e^\tau\Vert_{H^s(\R^{1+2d})}<+\infty$.
\end{lemma}

\begin{proof}[Proof of Lemma \ref{lemma_init}.]
Recall that $\Lambda = (\mathrm{id}-\sum_{i=1}^d\partial_i^2)^{1/2}$. Given any $j\in\Z$, we define on $\R^d$ the elliptic pseudo-differential operator $\Lambda'_{j,j}$ of order $j$ by $\Lambda'_{j,j} u(x) = \langle x\rangle^j \Lambda^j u(x)$ for every $u\in C^\infty(\R^d)$, where $\langle x\rangle = (1+\Vert x\Vert_2^2)^{1/2}$ (Japanese bracket), with $\Vert\cdot\Vert_2$ the Euclidean norm in $\R^d$. 
There exists $m\in\N^*$ such that $\Lambda'_{-m,-m}$ is Hilbert-Schmidt (as an operator on $L^2(\R^d)$), i.e., $\left\Vert \Lambda'_{-m,-m} \right\Vert_{HS}<+\infty$. 
Indeed, the Schwartz kernel of $\Lambda'_{-m,-m}$ is
$$
[\Lambda'_{-m,-m}](x,y) = \frac{1}{(2\pi)^n}\frac{1}{(1+\Vert x\Vert_2^2)^{m/2}} \int \frac{e^{i(x-y).\xi}}{(1+\Vert \xi\Vert_2^2)^{m/2}}\, d\xi  
$$
which is the product of $(1+\Vert x\Vert_2^2)^{-m/2}$ and of the Fourier transform taken at $y-x$ of the function $\xi\mapsto (1+\Vert \xi\Vert_2^2)^{-m/2}$, and it is therefore square-integrable when $m$ is chosen large enough.
Using the general facts (concerning Schwartz kernels) recalled in Appendix \ref{appendix_Schwartz} with $A = e^{t L^\tau}$, $T_1=\Lambda'_{-m,-m} \chi_1'$ and $T_2 = \chi_2'$, we obtain that
\begin{multline*}
\left\Vert (\chi_0)_t(\Lambda'_{-m,-m})_x(\chi_1')_x(\chi_2')_{x'}e^\tau \right\Vert_{L^2(\R^{1+2d})}
= \left\Vert (\chi_0)_t \Lambda'_{-m,-m} \chi_1' e^{tL^\tau}\chi_2' \right\Vert_{HS} \\
\leq \left\Vert \Lambda'_{-m,-m} \right\Vert_{HS} \left\Vert (\chi_0)_t e^{t L^\tau}\right\Vert_{L(L^2(\R^d))}
\leq \Cst(t_0,t_1)
\end{multline*}
when $\supp(\chi_0)\subset[t_0,t_1]$, where we have used the uniform estimate \eqref{uniformestimatesemigroup0} to derive the latter inequality.
\end{proof}

We have therefore proved \eqref{cor_heat_uniform_e}. To get the estimate on the semigroup, it suffices to write that, for all $j,k\in\Z$,
$$
\left\Vert \Lambda^j \chi_1 e^{tL^\tau}\chi_2\Lambda^k\right\Vert_{L(L^2(\R^d))} \leq \left\Vert \Lambda^j \chi_1 e^{tL^\tau}\chi_2\Lambda^k\right\Vert_{HS} = \left\Vert \Lambda^j_x(\chi_1)_x \Lambda^k_{x'} (\chi_2)_{x'} e^\tau(t,\cdot,\cdot)\right\Vert_{L^2(\R^d\times\R^d)}
$$
and then to apply \eqref{cor_heat_uniform_e}.
\end{proof}

\subsection{Global estimates}\label{sec:global_estimates}
In the sequel, we use the notation $\langle x\rangle = (1+\Vert x\Vert_2^2)^{1/2}$ for $x\in\R^d$ (Japanese bracket), where $\Vert x\Vert_2$ is the Euclidean norm of $x$.
Throughout the section, we assume that there exists $n\in\N^*$ such that the function $\mathbb{V}^\tau$ and all coefficients of the vector fields $Y_0^\tau,Y_1^\tau,\ldots,Y_p^\tau$, as well as all their derivatives, are bounded by multiples of $\langle x\rangle^n$.
In other words, we assume that the growth at infinity of the vector fields $Y_i^\tau$ and of $\mathbb{V}^\tau$, as well as their derivatives, is at most polynomial.

The objective of this section is to establish global subelliptic estimates and global smoothing properties of the corresponding heat semigroup and heat kernel.
We follow \cite{EckmannHairer}, where alternative global subelliptic H\"ormander estimates can be found (see \cite[Theorems 3.1 and 4.1]{EckmannHairer}) but are not sufficient for our needs.

Given any $\alpha\in\R$ and $\beta\in\R$, we define on $L^2(\R^d)$ the elliptic selfadjoint pseudo-differential operator $\Lambda_{\alpha,\beta}=\frac{1}{2}(\Lambda_{\alpha,\beta}'+(\Lambda_{\alpha,\beta}')^*)$ of order $\alpha$, with 
\begin{equation*}
\Lambda_{\alpha,\beta}' u(x) = \langle x\rangle^\beta \Big( \mathrm{id}-\sum_{i=1}^d\partial_{x_i}^2 \Big)^{\alpha/2} u(x) = \langle x\rangle^\beta \Lambda^\alpha u(x) .
\end{equation*}
and we define the weighted Hilbert space $H^\alpha_\beta(\R^d)$ as the completion of $C_c^\infty(\R^d)$ for the norm 
$$
\Vert u\Vert_{H^\alpha_\beta(\R^d)} = \Vert \Lambda_{\alpha,\beta} u\Vert_{L^2(\R^n)} .
$$
The inner product is $\langle u,v\rangle_{H^\alpha_\beta(\R^d)} = \langle \Lambda_{\alpha,\beta}u,\Lambda_{\alpha,\beta}v\rangle_{L^2(\R^n)} $.
We have the following properties (see also \cite{EckmannHairer}):
\begin{itemize}
\item $\Lambda_{\alpha,\beta} \in L(H^{\alpha'}_{\beta'}(\R^d), H^{\alpha'-\alpha}_{\beta'-\beta}(\R^d))$.
\item We have the continuous embedding $H^{\alpha'}_{\beta'}(\R^d) \hookrightarrow H^{\alpha}_{\beta}(\R^d)$ whenever $\alpha'\geq\alpha$ and $\beta'\geq\beta$, with a compact embedding if both inequalities are strict. 
\item The dual of $H^{\alpha}_{\beta}(\R^d)$ with respect to the pivot space $L^2(\R^d)=H^0_0(\R^d)$ is $H^{-\alpha}_{-\beta}(\R^d)$.
\end{itemize}

We denote by $(Y_i^\tau)_{i\geq 0}$ the family of vector fields consisting of the vector fields $Y_0^\tau,Y_1^\tau,\ldots,Y_p^\tau$ completed with all their successive Lie brackets.

\subsubsection{Uniform global subelliptic estimates}\label{sec:uniform_global_subelliptic_estimates}
\begin{definition}
Following \cite{EckmannHairer}, we say that the \emph{uniform polynomial weak H\"ormander condition} is satisfied if there exist nonnegative integers $N_0$ and $N_1$ such that
\begin{equation}\label{UPH_nondegen}
\Vert y\Vert_2^2 \leq \Cst \langle x\rangle^{2N_0} \sum_{i=0}^{N_1} \langle Y_i^\tau(x),y\rangle^2
\qquad \forall(x,y)\in\R^d\times\R^d\qquad\forall\tau\in\mathcal{K}.
\end{equation}
\end{definition}

\begin{theorem}\label{thm_subelliptic_uniform_global}
Under the uniform polynomial weak H\"ormander condition, there exist $\sigma>0$ and $N\in\N$ such that
$$
\Vert u\Vert_{H^{s+\sigma}_\beta(\R^d)} \leq \Cst(s,\beta) \left( \Vert L^\tau u\Vert_{H^s_{\beta+N}(\R^d)} + \Vert u\Vert_{H^s_{\beta+N}(\R^d)} \right) 
\qquad \forall s,\beta\in\R \qquad \forall u\in C_c^\infty(\R^d).
$$
\end{theorem}

For every $k\in\Z$, for every $\beta\in\R$, we define the Hilbert space $\mathcal{D}^\tau_{k,\beta}$ as the completion of $C_c^\infty(\R^d)$ for the norm 
$$
\Vert u\Vert_{\mathcal{D}^\tau_{k,\beta}} = \Vert (C\,\mathrm{id} - L^\tau - (L^\tau)^*)^k u\Vert_{H^0_\beta(\R^d)} 
$$
where $C>0$ is a fixed constant such that $(C\,\mathrm{id} - L^\tau - (L^\tau)^*)^k$ is a positive selfadjoint operator on $L^2(\R^d)$ (of maximal domain $\mathcal{D}^\tau_{k,0}$).
When $\beta=0$, we denote simply $\mathcal{D}^\tau_k=\mathcal{D}^\tau_{k,\beta}$.
For $k=1$, the set $\mathcal{D}^\tau_1=\{ u\in L^2(\R^d)\ \mid\ L^\tau u\in L^2(\R^d)\}$ is the maximal domain of the operator $L^\tau$ on $L^2(\R^d)$.

Note that the weight $\beta$ in the iterated domains will be used only in the proof of Corollary \ref{cor_heat_uniform_global} further.

By iteration, Theorem \ref{thm_subelliptic_uniform_global} implies that 
\begin{multline*}
\Vert u\Vert_{H^{s+k\sigma}_{\beta-kN}(\R^d)} \leq \Cst(s,k,\beta) \left( \Vert (L^\tau)^k u\Vert_{H^s_\beta(\R^d)} + \cdots + \Vert L^\tau u\Vert_{H^s_\beta(\R^d)} + \Vert u\Vert_{H^s_\beta(\R^d)} \right)  \\
\forall s,\beta\in\R\qquad \forall k\in\N^* \qquad \forall u\in C^\infty(\R^d) \qquad \forall\tau\in\mathcal{K}
\end{multline*}
and therefore we have the continuous embeddings
\begin{equation}\label{inclusions_global}
H^{2k}_{\beta+2kN}(\R^d) \hookrightarrow \mathcal{D}^\tau_{k,\beta} \hookrightarrow H^{k\sigma}_{\beta-kN}(\R^d)
\qquad\forall k\in\N\qquad\forall\beta\in\R\qquad \forall\tau\in\mathcal{K}
\end{equation}
with constants depending on $k$ but not on $\tau$. By duality, we have
\begin{equation}\label{inclusions_global_neg}
H^{-k\sigma}_{\beta+kN}(\R^d) \hookrightarrow \mathcal{D}^\tau_{-k,\beta} \hookrightarrow H^{-2k}_{\beta-2kN}(\R^d) \qquad\forall k\in\N\qquad\forall\beta\in\R\qquad \forall\tau\in\mathcal{K} .
\end{equation}

\begin{remark}\label{rem_variant_thm_subelliptic_uniform_global}
Actually, by bracketting (see the sketch of proof of the theorem hereafter), we obtain the following slight variant of Theorem \ref{thm_subelliptic_uniform_global}: for all smooth functions $\zeta$ and $\zeta'$ on $\R^n$, which are globally bounded as well as all their derivatives, satisfying $\zeta'=1$ on the support of $\zeta$, we have
\begin{multline*}
\Vert \zeta u\Vert_{H^{s+\sigma}_{\beta-N}(\R^d)} \leq \Cst(s,\beta,\zeta,\zeta') \left( \Vert \zeta' L^\tau u\Vert_{H^s_\beta(\R^d)} + \Vert \zeta' u\Vert_{H^s_\beta(\R^d)} \right)  \\
\forall s,\beta\in\R \qquad \forall u\in C^\infty(\R^d) \qquad \forall\tau\in\mathcal{K} .
\end{multline*}
Note that the functions $\zeta$ and $\zeta'$ are not assumed here to be of compact support. This fact will be useful in Appendix \ref{sec:globalregheat} to localize in time only (but not in space).
\end{remark}

\begin{proof}[Proof of Theorem \ref{thm_subelliptic_uniform_global}.]
For every $m\in\R$ and every $\delta\in\R$, we define the following class of global symbols $\mathcal{S}^{m,\delta}$, by considering their growth in $x$: a symbol $a:\R^d\times\R^d\setminus \{0\}\rightarrow \C$ (of order $m$) is a smooth function having an asymptotic expansion $a\sim \sum _{j=0}^{+\infty} a_{m-j}$, where $a_{m-j}$ is smooth and such that, for any $k\in\N$, the function $a_k=a-\sum _{j=0}^{k-1}a_{m-j}$ satisfies
$$
\vert \partial_x^\alpha \partial_\xi^\beta a_k (x,\xi)\vert \leq  \Cst(\alpha,\beta) \langle x\rangle^{\delta-\vert\alpha\vert} (1+\Vert\xi\Vert_2)^{m-k-\vert\beta\vert}  \qquad\forall \alpha,\beta\in\N^d.
$$
We denote by $\Psi^{m,\delta}$ the set of pseudo-differential operators on $\R^d$ whose symbol belongs to $\mathcal{S}^{m,\delta}$. We have the following useful property: if $a\in\mathcal{S}^{m,\delta}$ then $\mathrm{Op}(a) \in L(H^\alpha_\beta(\R^d),H^{\alpha-m}_{\beta-\delta}(\R^d))$. Here, $\mathrm{Op}$ is the usual left (for instance) quantization operator.

Now, if we consider in $\mathcal{S}^{m,\delta}$ a family of symbols $(a^\tau)_{\tau\in\mathcal{K}}$, with the above constants $\Cst(\alpha,\beta)$ being uniform with respect to $\tau$, then we have the uniform estimate
$$
\Vert \mathrm{Op}(a^\tau)\Vert_{L\big(H^\alpha_\beta(\R^d),H^{\alpha-m}_{\beta-\delta}(\R^d)\big)}\leq \Cst(\alpha,\beta) \qquad\forall \tau\in\mathcal{K}.
$$
Thanks to this general remark, we obtain that $Y^\tau_i \in L(H^\alpha_\beta(\R^d),H^{\alpha-1}_{\beta-N}(\R^d))$, that $\Lambda^{-1}Y^\tau_i \in L(H^\alpha_\beta(\R^d),H^\alpha_{\beta-N}(\R^d))$ and that $[Y^\tau_i,\Lambda^\gamma] \in L(H^\alpha_\beta(\R^d),H^{\alpha-\gamma}_{\beta-N}(\R^d))$, with uniform constants in the estimates (see also \cite[Lemma 3.2]{EckmannHairer} for other useful properties of the spaces $H^\alpha_\beta(\R^d)$).

These preliminaries being done, we follow the steps of the proof of Theorem \ref{thm_subelliptic_uniform},  keeping track of constants. It is then straightforward to establish that
\begin{equation*}
\begin{split}
& \sum_{j=1}^p \Vert Y_j^\tau u\Vert_{L^2(\R^d)}^2 \leq \Cst \left( \left\vert\langle L^\tau u,u\rangle_{L^2(\R^d)}\right\vert + \langle x\rangle^{2N} \Vert u\Vert_{L^2(\R^d)}^2 \right)  \\
& \Vert u\Vert_{H^\sigma(\R^d)} \leq \Cst \langle x\rangle^N \left( \Vert L^\tau u\Vert_{L^2(\R^d)} +  \Vert u\Vert_{L^2(\R^d)} \right) \\
& \qquad\qquad\qquad\qquad\qquad\qquad\qquad\qquad \forall u\in C_c^\infty(B(x,1))\qquad \forall x\in\R^d \qquad\forall \tau\in\mathcal{K}
\end{split}
\end{equation*}
where $N\in\N$ is a sufficiently large integer. 
Then, noting that, roughly, $\Vert u\Vert_{H^\alpha_\beta(\R^d)} \simeq \langle x\rangle^\beta\Vert u\Vert_{H^\alpha(\R^d)}$ for every $u\in C_c^\infty(B(x,1))$ and for all $\alpha,\beta\in\R$, using a partition of unity (see \cite[Proof of Theorem 4.1]{EckmannHairer} for details), we obtain
$$
\Vert u\Vert_{H^\sigma(\R^d)} \leq \Cst  \left( \Vert L^\tau u\Vert_{H^0_N(\R^d)} +  \Vert u\Vert_{H^0_N(\R^d)} \right) \\
\qquad \forall u\in C_c^\infty(\R^d) \qquad\forall \tau\in\mathcal{K} .
$$
Using the brackets $[\Lambda_{s,\beta},Y_i^\tau]$, $[\Lambda_{s,\beta},L^\tau]$, increasing the integer $N$ if necessary (but $N$ remains uniform with respect to $s$ and $\beta$), we then obtain
\begin{equation*}
\sum_{j=1}^p \Vert Y_j^\tau u\Vert_{H^s_\beta(\R^d)}^2 \leq \Cst(s,\beta) \left( \left\vert\langle L^\tau u,u\rangle_{H^s_\beta(\R^d)}\right\vert + \Vert u\Vert_{H^s_{\beta+N}(\R^d)}^2 \right)  
\qquad \forall u\in C_c^\infty(\R^d) \qquad\forall \tau\in\mathcal{K} 
\end{equation*}
and we then establish the global estimate of Theorem \ref{thm_subelliptic_uniform_global}. We do not provide any details.
\end{proof}

\begin{remark}\label{rem_polynomial_Hormander}
It is interesting to note that the uniform polynomial weak H\"ormander condition \eqref{UPH_nondegen} is satisfied for any $m$-tuple $(Y_0,Y_1,\ldots,Y_p)$ of polynomial vector fields that satisfy the weak H\"ormander condition at every point with a uniform degree of nonholonomy. 

Indeed, defining as above the family $(Y_i)_{i\geq 0}$ by completing the $(p+1)$-tuple $(Y_0,\ldots,Y_p)$ with their iterated Lie brackets, by the assumption, there exists a large enough integer $N_1$ such that the span of the finite family $(Y_i)_{0\leq i\leq N_1}$ is equal to $\R^n$ at every point. Defining the matrix
$$
P(x) = \begin{pmatrix}
Y_0(x) & Y_1(x) & \cdots & Y_{N_1}(x)
\end{pmatrix} \qquad \forall x\in\R^n,
$$
we have to prove that there exists $N_0\in\N$ such that
\begin{equation}\label{coerciv_N0}
y^\top P(x) P(x)^\top y \geq \Cst \langle x\rangle^{-2N_0} \Vert y\Vert_2^2\qquad \forall x,y\in\R^n.
\end{equation}
By \cite[Lemma 2]{Ho-58}, it is known that, for every polynomial function $Q$ on $\R^n$ such that $Q(x)>0$ for every $x\in\R^n$, there exists an integer $N\in\N$ such that $Q(x)\geq \frac{\Cst}{\langle x\rangle^{2N}}$. In other words, the decay at infinity of any polynomial is polynomial, it cannot be arbitrarily small. This non-obvious fact follows from a \L{}ojasiewicz inequality combined with an inversion argument.
The existence of $N_0$ such that \eqref{coerciv_N0} is satisfied follows.
\end{remark}

\subsubsection{Application: uniform global smoothing property of heat semigroups}\label{sec:globalregheat}
We say that the \emph{uniform polynomial strong H\"ormander assumption} is satisfied if there exist nonnegative integers $N_0$ and $N_1$ such that
\begin{equation}\label{UPH_nondegen_1}
\Vert y\Vert_2^2 \leq \Cst\, \langle x\rangle^{2N_0} \sum_{i=1}^{N_1} \langle Y_i^\tau(x),y\rangle^2
\qquad \forall (x,y)\in\R^d\times\R^d\qquad\forall \tau\in\mathcal{K} .
\end{equation}
Compared with \eqref{UPH_nondegen}, in \eqref{UPH_nondegen_1} we have excluded the vector field $Y_0^\tau$: the sum starts at $i=1$.
By a similar argument as the one elaborated in Remark \ref{rem_polynomial_Hormander}, note that \eqref{UPH_nondegen_1} is satisfied for polynomial vector fields satisfying the strong H\"ormander condition at every point, with a uniform degree of nonholonomy.

Throughout this section, we assume that the operator $L^\tau: D(L^\tau)=\mathcal{D}^\tau_1\rightarrow L^2(\R^d)$ generates a strongly continuous semigroup $(e^{tL^\tau})_{t\geq 0}$, for every $\tau\in\mathcal{K}$, satisfying the uniform estimate: for all positive real numbers $t_0<t_1$, we have 
\begin{equation}\label{13:05}
\Vert e^{tL^\tau}\Vert_{L(L^2(\R^d))}\leq \Cst(t_0,t_1) \qquad\forall t\in[t_0,t_1] .
\end{equation}

We denote by $e^\tau$ 
the heat kernel associated with $L^\tau$ for the Lebesgue measure on $\R^d$.
The function $e^\tau$ is defined on $(0,+\infty)\times\R^d\times\R^d$ and depends on three variables $(t,x,y)$. In what follows, the notation $\partial_1$ (resp., $\partial_2$, $\partial_3$) denotes the partial derivative with respect to $t$ (resp., to $x$, to $y$).

\begin{corollary}\label{cor_heat_uniform_global}
Under the uniform polynomial strong H\"ormander assumption, there exist $s_0>0$ and $k_0\in\N^*$ such that, for all $s\geq s_0$ and $s'\geq s_0$, for all positive real numbers $0<t_0<t_1$,
$$
\Vert (\Lambda_{s,-k_0s})_x (\Lambda_{s',-k_0s'})_{y}\, e^\tau(\cdot,\cdot,\cdot) \Vert_{L^2((t_0,t_1)\times\R^d\times\R^d)}\leq \Cst(s,s',t_0,t_1) \qquad \forall \tau\in\mathcal{K}
$$
and
$$
\Vert e^{t L^\tau} \Vert_{L\big( H^{-s'}_{k_0s'}(\R^d), H^{s}_{-k_0s}(\R^d) \big)}
\leq \Cst(s,s',t_0,t_1) 
\qquad \forall t\in[t_0,t_1]\qquad \forall \tau\in\mathcal{K} 
$$
\end{corollary}

This result is a global smoothing property for hypoelliptic heat kernels in Sobolev spaces with polynomial weight.

\begin{proof}
We follows the same lines as in the proof of Corollary \ref{cor_heat_uniform}, by considering the operator $P^\tau$.
It follows from Remark \ref{rem_variant_thm_subelliptic_uniform_global} that, for all $s,\beta\in\R$, for every $k\in\N$, we have $\Vert \zeta e^\tau\Vert_{H^{s+k\sigma}_{\beta-kN}(\R^{1+2d})} \leq \Cst(s,k,\beta,\zeta,\zeta') \Vert \zeta' e^\tau\Vert_{H^{s}_\beta(\R^{1+2d})}$, for all smooth functions $\zeta$ and $\zeta'$ on $(0,+\infty)\times\R^d\times\R^d$ satisfying $\zeta'=1$ on the support of $\zeta$, and such that $\zeta$ and $\zeta'$ and all their derivatives are bounded by constants. 
We choose $\zeta=(\chi_0)_t$ and $\zeta'=(\chi_0')_t$ only depending on $t$, in order to localize in time over $t\in[t_0,t_1]$.

To initialize the bootstrap argument, we first observe that, like in the proof of Corollary \ref{cor_heat_uniform}, using \eqref{13:05}, there exists $m\in\N^*$ such that $\Vert (\chi_0')_t (\Lambda_{-m,-m})_x e^\tau\Vert_{L^2(\R^{1+2d})}\leq\Cst(t_0,t_1)$ (assuming that $\supp(\chi_0')\subset[t_0,t_1]$), hence $\Vert (\chi_0')_t e^\tau\Vert_{H^{-m}_{-m}(\R^{1+2d})}\leq\Cst(t_0,t_1)$.
Therefore, taking $s=\beta=-m$, we get that $\Vert (\chi_0)_t e^\tau\Vert_{H^{-m+k\sigma}_{-m-kN}(\R^{1+2d})}\leq \Cst(k,t_0,t_1)$ for every $k\in\N^*$. The result follows.
\end{proof}

\section{Global smoothing properties of sR heat semigroups}
\label{sec:global_sR}
This section can be read independently of the rest.

\subsection{Global smoothing properties in iterated domains}\label{sec:global_estimates_sR}
Throughout this section, we consider a selfadjoint sR Laplacian $\trianglesR=-\sum_{i=1}^mX_i^*X_i$ (see Section \ref{sec_sRLaplacian}) on $L^2(\R^n)$, where $X_1,\ldots,X_m$ are smooth vector fields on $\R^n$, satisfying the H\"ormander condition at every point of $\R^n$. We denote by $(e^{t\trianglesR})_{t\geq 0}$ the associated sR heat semigroup, and by $e$ the sR heat kernel.

We show how to use the Kannai transform to establish global smoothing properties of the sR heat semigroup in the scale of Sobolev spaces associated with $\trianglesR$ (iterated domains $D((\mathrm{id}-\trianglesR)^j)$, for $j\in\Z$). This technique is well known (see \cite{CheegerGromovTaylor_1982}) and, conveniently combined with complex analysis (e.g., Phragmen-Lindel\"of principle, see \cite{CoulhonSikora_PLMS2008}), leads to exponential estimates of heat kernels. However, we recall some of the precise arguments and we give statements that we have not been able to find in this form in the existing literature.

Since the Kannai transform is based on the spectral theorem, it requires to consider a selfadjoint operator. This is why, in this section, we only consider selfadjoint sR Laplacians.

\subsubsection{Rough global smoothing properties in iterated domains}
Starting from the inequality $(1+\lambda^2)^k e^{-t\lambda^2} \leq \frac{\Cst(k,t_1)}{t^k}$ for all $\lambda\geq 0$, $t_1>0$, $t\in(0,t_1]$ and $k\in\Z$, we infer from the spectral theorem applied to the selfadjoint operator $\trianglesR$ that
$\Vert (\mathrm{id}-\trianglesR)^k e^{t\trianglesR}\Vert_{L(L^2(\R^n))} \leq \frac{\Cst(k,t_1)}{t^k}$.
Hence, given any $t>0$ and any $j,k\in\Z$, the operator $e^{t\trianglesR}$ is bounded as an operator from $\mathcal{D}_j=D((\mathrm{id}-\trianglesR)^{j})$ to $\mathcal{D}_k=D((\mathrm{id}-\trianglesR)^{k})$, and we obtain the following result.

\begin{lemma}
Given any $t>0$ and any $j,k\in\Z$, the operator $e^{t\trianglesR}$ is bounded as an operator from $\mathcal{D}_j=D((\mathrm{id}-\trianglesR)^{j})$ to $\mathcal{D}_k=D((\mathrm{id}-\trianglesR)^{k})$, and
\begin{equation*}
\Vert e^{t\trianglesR} \Vert_{L( \mathcal{D}_j, \mathcal{D}_k)} \leq 
\Bigg\{\begin{array}{ll}
\displaystyle\Cst(k-j,t_1)\frac{1}{t^{2(k-j)}} \qquad & \textrm{if}\ k\geq j , \\[2mm]
\displaystyle\Cst(k) \qquad & \textrm{if}\ k\leq j ,
\end{array}
\qquad \forall t_1>0\qquad\forall t\in(0,t_1] .
\end{equation*}
\end{lemma}
This property is completely general and is even true without any H\"ormander condition on the vector fields $X_i$: it is satisfied for any selfadjoint nonpositive operator.

We establish in Section \ref{sec:global_estimates_sR} much finer global smoothing properties, thanks to the Kannai transform that we recall in the next section.

\subsubsection{Kannai transform}
Hereafter, we recall a classical argument developed in \cite{CheegerGromovTaylor_1982}, consisting of obtaining estimates of the heat semigroup (in the Riemannian case), via the Kannai transform, 
Here, we are in the subelliptic case and we consider a sR Laplacian $\trianglesR$ on $\R^n$: this is a nonpositive selfadjoint operator on $L^2(\R^n)$, of domain $D(\trianglesR)$.
We define the sR wave operator $\cos\left(t\sqrt{-\trianglesR}\right)$ on $L^2(\R^n)$ as follows: $u(t,\cdot)=\cos\left(t\sqrt{-\trianglesR}\right)\delta_y$ is the solution to
\begin{equation*}
\begin{split}
& (\partial_{tt}-\trianglesR)u=0, \\
& u(0,\cdot)=\delta_y,\ \partial_t u(0,\cdot)=0.
\end{split}
\end{equation*}
For every $t\in\R$ we have
\begin{equation}\label{cosinf1}
\left\Vert \cos\left(t\sqrt{-\trianglesR}\right)\right\Vert_{L(L^2(\R^n))}\leq 1 .
\end{equation}
Given any smooth even real-valued function $F\in L^2(\R)$, its Fourier transform is defined by
$\hat F(s) = \int_\R F(\lambda) \cos(\lambda s)\, ds$
and we have
$$
F(\lambda) = \frac{1}{2\pi}\int_\R \hat F(s) \cos(\lambda s)\, ds = \frac{1}{\pi}\int_0^{+\infty} \hat F(s) \cos(\lambda s)\, ds .
$$
As a consequence of the spectral theorem, we have
$$
F\left(\sqrt{-\trianglesR}\right) = \frac{1}{2\pi}\int_\R \hat F(s) \cos\left(s\sqrt{-\trianglesR}\right)\, ds = \frac{1}{\pi}\int_0^{+\infty} \hat F(s) \cos\left(s\sqrt{-\trianglesR}\right) ds .
$$
In particular, fixing some arbitrary $t\geq 0$ and taking $F(\lambda)=e^{-t\lambda^2}$, we have $\hat F(s) =  \sqrt{\frac{\pi}{t}} e^{-s^2/4t}$ and thus
$$
e^{-t\lambda^2} = \frac{1}{2\sqrt{\pi t}} \int_\R e^{-s^2/4t} \cos(s\lambda)\, ds = \frac{1}{\sqrt{\pi t}} \int_0^{+\infty} e^{-s^2/4t} \cos(s\lambda)\, ds
$$
for every $\lambda\in\R$, and therefore
$$
e^{t\trianglesR} = \frac{1}{2\sqrt{\pi t}} \int_\R e^{-s^2/4t} \cos\left(s\sqrt{-\trianglesR}\right)\, ds = \frac{1}{\sqrt{\pi t}} \int_0^{+\infty} e^{-s^2/4t} \cos\left(s\sqrt{-\trianglesR}\right) ds
$$
for every $t>0$. This formula is usually called transmutation, or Kannai transform (see \cite{Kannai}).

We also have
$$
e^{-t(1+\lambda^2)} = \frac{1}{2\sqrt{\pi t}} \int_\R e^{-t-s^2/4t} \cos(s\lambda)\, ds = \frac{1}{\sqrt{\pi t}} \int_0^{+\infty} e^{-t-s^2/4t} \cos(s\lambda)\, ds .
$$
Derivating $k$ times with respect to $t$, for some $k\in\N$, one gets
$$
(1+\lambda^2)^ke^{-t(1+\lambda^2)} = \int_0^{+\infty} \frac{d^k}{dt^k}\left( \frac{1}{\sqrt{\pi t}} e^{-t-s^2/4t}\right) \cos(s\lambda)\, ds 
$$
i.e.,
$$
(1+\lambda^2)^ke^{-t\lambda^2} = \int_0^{+\infty} \frac{P_{2k}(s,t)}{t^{2k+1/2}} e^{-s^2/4t} \cos(s\lambda)\, ds 
$$
where $P_{2k}$ is a polynomial of the two variables $s$ and $t$, of degree $2k$ (actually, depending on $s^2$ and of $t$), and thus
\begin{equation}\label{fourierjk}
(\mathrm{id}-\trianglesR)^ke^{t\trianglesR}
= \int_0^{+\infty} \frac{P_{2k}(s,t)}{t^{2k+1/2}} e^{-s^2/4t} \cos\left(s\sqrt{-\trianglesR}\right) ds \qquad
\forall t>0\qquad \forall k\in\N\qquad\forall j\in\Z .
\end{equation}

\subsubsection{Exponential estimates, using the finite speed propagation property}
Let us establish an additional estimate, by using the finite speed propagation property for the sR wave equation. The finite speed propagation property for sR wave equations has been established in \cite{Me-84}; actually, it follows from the finite speed propagation property for usual Riemannian wave equations, by considering a Riemannian $\varepsilon$-regularization of the sR Laplacian (see \cite{Ge}) and then passing to the limit, using the fact that the $\varepsilon$-regularized Riemannian distance converges to the sR distance as $\varepsilon$ tends to $0$ (see also \cite[Section 1.4.D]{Gromov} and \cite[Appendix A4]{Varopoulos}).

Let $x\in\R^n$ and let $0<r_1<r_2$ be arbitrary. In what follows, we denote by $\BsR(x,r_i)$ the sR ball in $\R^n$ of center $x$ and of radius $r_i$, for the sR distance $\dsR$ associated with $\trianglesR$ (we assume that the H\"ormander condition is satisfied).

Given any smooth function $f$ on $\R^n$ such that $\supp(f)\subset\BsR(x,r_1)$, let us estimate $\Vert (\mathrm{id}-\trianglesR)^{k}e^{t\trianglesR} f\Vert_{L^2(\R^n\setminus \BsR(x,r_2))}$. 
Since $\supp(f)\subset \BsR(x,r_1)$, we have, by the finite speed propagation property, 
$$
\supp\left( \cos\left(s\sqrt{-\trianglesR}\right) f\right) \subset \BsR(x,r_1+s)\qquad\forall s\geq 0.
$$
In other words, a minimal time $s\geq r_2-r_1$ is required in order to transport information from $\BsR(x,r_1)$ to $\R^n\setminus \BsR(x,r_2)$.
Therefore, using \eqref{cosinf1} and \eqref{fourierjk} we get that
\begin{equation}\label{estimate_glob}
\begin{split}
\left\Vert (\mathrm{id}-\trianglesR)^{k}e^{t\trianglesR} f\right\Vert_{L^2(\R^n\setminus \BsR(x,r_2))}
&= \left\Vert \int_{r_2-r_1}^{+\infty} \frac{P_{2k}(s,t)}{t^{2k+1/2}} e^{-s^2/4t} \cos\left(s\sqrt{-\trianglesR}\right)f\, ds \right\Vert_{L^2(\R^n)} \\
&\leq \int_{r_2-r_1}^{+\infty} \frac{\vert P_{2k}(s,t)\vert}{t^{2k+1/2}} e^{-s^2/4t}\, ds\, \Vert f\Vert_{L^2(\R^n)} \\
&\leq \Cst(k)\frac{1+t^{2k}}{t^{2k}} (1+(r_2-r_1)^{2k}) e^{-(r_2-r_1)^2/4t} \Vert f\Vert_{L^2(\R^n)} \\[2mm]
\forall t>0 \qquad & \forall k\in\N\qquad \forall f\in C_c^\infty(\R^n) \ \textrm{s.t.}\ \supp(f)\subset \BsR(x,r_1) .
\end{split}
\end{equation}
Note that $r_2-r_1 = \dsR(\BsR(x,r_1),\R^n\setminus \BsR(x,r_2))$. To obtain the latter inequality, we have used the inequalities (proved by recurrence) 
\begin{multline*}
\int_a^{+\infty} e^{-s^2/4t}\, ds \leq \Cst \, \sqrt{t}\, e^{-a^2/4t}, \\
\int_a^{+\infty} s^j e^{-s^2/4t}\, ds \leq \Cst(j) t (1+t^{(j-1)/2})(1+a^{j-1}) e^{-a^2/4t}\qquad \forall j\in\N^*\qquad \forall a>0\qquad \forall t>0.
\end{multline*}

A consequence of this analysis is the following.

\begin{proposition}
Let $\chi$ and $\chi'$ be smooth functions on compact support on $\R^n$, such that $\chi'=1$ on $\supp(\chi)$. Then, for every $k\in\N$ and every $t>0$, the operator $(1-\chi') (\mathrm{id}-\trianglesR)^{k}e^{t\trianglesR} \chi$ is bounded on $L^2(\R^n)$, and
\begin{equation*}
\left\Vert (1-\chi') (\mathrm{id}-\trianglesR)^{k}e^{t\trianglesR} \chi \right\Vert_{L(L^2(\R^n))}
\leq \Cst(k,\chi,\chi')\frac{1+t^{2k}}{t^{2k}} e^{-\Cst(\chi,\chi')/t}  \qquad\forall t>0  .
\end{equation*}
The same property holds for the operator $(\mathrm{id}-\trianglesR)^{k}(1-\chi')e^{t\trianglesR} \chi$ and (taking the adjoint and by selfadjointness) for the operators $\chi (\mathrm{id}-\trianglesR)^{k}e^{t\trianglesR} (1-\chi')$ and $(\mathrm{id}-\trianglesR)^{k} \chi e^{t\trianglesR} (1-\chi')$.
\end{proposition}

\begin{proof}
Using a partition of unity, we reduce the proof to the case where $\supp(\chi)\subset\BsR(\bar x,r)$ and $\chi'=1$ on $\BsR(\bar x,2r)$ for some $\bar x\in\R^n$ and some $r>0$, so that $\supp(\chi')\subset\R^n\setminus\BsR(\bar x,2r)$. Then \eqref{estimate_glob} gives the conclusion.
\end{proof}

Note that all estimates that we have obtained are uniform with respect to $t>0$.


\subsection{Global smoothing properties for nilpotent sR Laplacians} 
\label{sec:global_estimates_sR_weight}
In this section, we consider the nilpotent sR Laplacian $\widehat{\triangle}^q$, associated with the nilpotentization $(\widehat{M}^q,\widehat{D}^q,\widehat{g}^q)$ of the sR structure $(M,D,g)$ at $q$.
Recall that $\widehat{M}^q$ is a homogeneous space of a Carnot group (see Section \ref{sec_def_nilpotentization}), but is not a Carnot group whenever $q$ is singular.

We will establish global smoothing properties for the semi-group $(e^{t\widehat{\triangle}^q})_{t\geq 0}$ in iterated domains and Sobolev spaces with polynomial weight.

Note that, since we deal in this section with polynomial vector fields, we could apply the smoothing properties established in Appendix \ref{sec:global_estimates} (more precisely, see Corollary \ref{cor_heat_uniform_global}) by deriving global subelliptic estimates in Sobolev spaces with polynomial weight. But here, in this specific context, the global smoothing properties that we are going to establish are much stronger.

\subsubsection{Reminders on exponential estimates}
There are various ways to establish exponential estimates for heat kernels. 
Such estimates have been proved in \cite{Varopoulos} by means of large deviation theory, using the probabilistic interpretation of the sR heat kernel (see also \cite{KusuokaStroock} for long time estimates). We also refer the reader to \cite{Sac-84} and \cite{Je-Sa-86} for estimates of the sR heat kernel and of its derivatives, established, however, on a compact manifold.

In \cite{CheegerGromovTaylor_1982} (see also \cite{Me-84}) the authors show how to combine the Kannai transform (i.e., the finite propagation speed of waves) with the Harnack principle and with a classical Moser iteration argument.
In \cite{Saloff-Coste_IMRN1992}, the author combines the Kannai transform with the Harnack parabolic principle in the hypoelliptic case.
In \cite{CoulhonSikora_PLMS2008}, the authors use an elegant Phragmen-Lindel\"of argument to obtain off-diagonal exponential estimates.

Finally, we quote the paper \cite{Maheux_JGA1998}, in which it is shown that all known exponential estimates remain valid as well when considering a sR Laplacian on a homogeneous space (thus, extending results of \cite{Varopoulos} that were established on Lie groups).

According to these references, denoting by $\widehat{e}^q$ the heat kernel of the sR Laplacian $\widehat{\triangle}^q=\sum_{i=1}^m (\widehat{X}_i^q)^2$, for every $\varepsilon>0$, for all $i_1,\ldots,i_s\in\{1,\ldots,m\}$ with $s\in\N^*$, we have
\begin{equation}\label{upper_estim_kernel}
\left\vert \partial_t^m X_q^I \widehat{e}^q(t,x,y) \right\vert \leq \frac{\Cst(\varepsilon,I)}{t^{m+\vert I\vert/2}} \frac{1}{\mathrm{Vol}(\widehat{B}^q(y,\sqrt{t}))} \exp \left( \frac{-\widehat{\mathrm{d}}^q(x,y)^2}{4(1+\varepsilon)t} \right) 
\qquad \forall t>0\qquad\forall x,y\in\widehat{M}^q\simeq\R^n
\end{equation}
where $X^I = \widehat{X}_{i_1}\cdots \widehat{X}_{i_s}$, $I=(i_1,\ldots,i_s)$ and $\vert I\vert=s$.
In other words, the heat kernel $\widehat{e}^q$ and its derivatives are exponentially decreasing off the diagonal.
Note that $\mathrm{Vol}(\widehat{B}^q(y,\sqrt{t}))=t^{\mathcal{Q}(y)/2}$.
By symmetry of the heat kernel, one can obtain the inequality \eqref{upper_estim_kernel} with $t^{\mathcal{Q}(y)/2}$ replaced by $t^{\mathcal{Q}(x)/2}$, or even by $t^{\mathcal{Q}(x)/4}t^{\mathcal{Q}(y)/4}$. Note anyway that all these inequalities are different, because $\mathcal{Q}(x)\neq \mathcal{Q}(y)$ in general, unless $\widehat{M}^q$ is a Carnot group (which is the case when $q$ is regular).

\subsubsection{Global smoothing properties}
As in Appendix \ref{sec:uniform_global_subelliptic_estimates}, given any $j\in\Z$ and any $\alpha\in\R$, we define $\widehat{\mathcal{D}}^q_{j,\alpha}$ as the completion of $C_c^\infty(\R^n)$ for the norm 
$$
\Vert u\Vert_{\widehat{\mathcal{D}}^q_{j,\alpha}} = \Vert (\mathrm{id} - \widehat{\triangle}^q)^j u\Vert_{H^0_\alpha(\R^n)} = \Vert \langle x\rangle^\alpha(\mathrm{id} - \widehat{\triangle}^q)^j u\Vert_{L^2(\R^n)} .
$$
Here, we recall that the Japanese bracket is defined by $\langle x\rangle = (1+\Vert x\Vert_2^2)^{1/2}$ where $\Vert x\Vert_2^2=x_1^2+\cdots+x_n^2$ is the Euclidean norm.

It is useful to define the \emph{sR Japanese bracket} by $\langle x\rangle_{\mathrm{sR}} = (1+\Vert x\Vert_{\mathrm{sR}}^2)^{1/2}$, where the sR pseudo-norm is defined by $\Vert x\Vert_{\mathrm{sR}} = \sum_{i=1}^n \vert x_i\vert^{1/w_i}$ for every $x=(x_1,\ldots,x_n)\in\R^n$ in privileged coordinates around $0$ (here, $w_i=w_i(q)$). Accordingly, we define $\widehat{\mathcal{D}}_{j,\alpha}^{q,\mathrm{sR}}$ as the completion of $C_c^\infty(\R^n)$ for the norm 
$$
\Vert u\Vert_{\widehat{\mathcal{D}}_{j,\alpha}^{q,\mathrm{sR}}} = \Vert \langle x\rangle_{\mathrm{sR}}^\alpha(\mathrm{id} - \widehat{\triangle}^q)^j u\Vert_{L^2(\R^n)} .
$$
Setting $r=w_n$ (degree of nonholonomy at $0$), we have the inequality
\begin{equation*}
\Cst\, \langle x\rangle^{1/r} \leq \langle x\rangle_{\mathrm{sR}} \leq \Cst\, \langle x\rangle\qquad \forall x\in\R^n
\end{equation*}
from which it follows that
\begin{equation}\label{equivDjalpha}
\Cst \Vert u\Vert_{\widehat{\mathcal{D}}^q_{j,\alpha/r}} \leq \Vert u\Vert_{\widehat{\mathcal{D}}_{j,\alpha}^{q,\mathrm{sR}}} \leq \Cst \Vert u\Vert_{\widehat{\mathcal{D}}^q_{j,\alpha}}\qquad\forall j\in\Z\qquad\forall\alpha>0\qquad\forall u\in C_c^\infty(\R^n).
\end{equation}

\begin{proposition}\label{prop_globreg_poidsfixe}
Given any $t>0$, any $j,k\in\Z$ and any real number $\beta\geq 1$, the operator $e^{t\widehat{\triangle}^q}$ is bounded as an operator from $\widehat{\mathcal{D}}^{q,\mathrm{sR}}_{j,\beta}$ to $\widehat{\mathcal{D}}^{q,\mathrm{sR}}_{k,\beta}$, 
and
\begin{equation}\label{estim_domainhat}
\Vert e^{t\widehat{\triangle}^q}\Vert_{L\big(\widehat{\mathcal{D}}_{j,\beta}^{q,\mathrm{sR}},\widehat{\mathcal{D}}_{k,\beta}^{q,\mathrm{sR}}\big)} \leq 
\left\{\begin{array}{ll}
\displaystyle \frac{\Cst(j,k,\beta,t_1)}{t^{k-j}} \qquad & \textrm{if}\ k\geq j , \\[2mm]
\displaystyle\Cst(k,\beta,t_1) \qquad &  \textrm{if}\ k\leq j ,
\end{array}\right.
\qquad \forall t\in(0,t_1]\qquad\forall t_1>0 .
\end{equation}
\end{proposition}

This result states a global smoothing property in the iterated domains with polynomial weight, by keeping the same weight $\beta\geq 1$. 
In particular, for $j=k$, we have
\begin{equation}\label{15:34}
\left\Vert e^{t\widehat{\triangle}^q} f\right\Vert_{\mathcal{D}_{k,\beta}^{\mathrm{sR}}}
\leq \Cst(k,\beta)  \left\Vert f\right\Vert_{\mathcal{D}_{k,\beta}^{\mathrm{sR}}} \qquad
\forall t\in[0,1] \qquad \forall k\in\Z\qquad\forall\beta\geq 1\qquad \forall f\in C_c^\infty(\R^n) .
\end{equation}
This implies that $(e^{t\widehat{\triangle}^q})_{t\geq 0}$ is a $C_0$ semigroup in the Hilbert space $\mathcal{D}_{k,\beta}^{\mathrm{sR}}$, for every $k\in\Z$ and every $\beta\geq 1$, which may not be a semigroup of contractions. Note that, for $k=0$, $\mathcal{D}_{k,\beta}^{\mathrm{sR}}$ is the space $L^2(\R^n)$ with polynomial weight $\langle x\rangle_{\mathrm{sR}}^\beta$ and indeed, for $\beta\geq 1$, $\widehat{\triangle}^q$ may not be dissipative in this Hilbert space.

\begin{proof}
Since $(\mathrm{id} - \widehat{\triangle}^q)^j e^{t\widehat{\triangle}^q} = e^{t\widehat{\triangle}^q} (\mathrm{id} - \widehat{\triangle}^q)^j$, it suffices to prove \eqref{estim_domainhat} for $j=0$. We proceed in two steps: we first establish \eqref{estim_domainhat} for $t=1$, and then we show how to obtain the estimate for every $t\in(0,t_1]$ by homogeneity.

Let $\beta\geq 1$ and $k\in\N$ be arbitrary. We follow and adapt a classical argument (Schur test): by the Cauchy-Schwarz inequality, we have
\begin{multline}\label{ineg_schur}
\left( \langle x\rangle_{\mathrm{sR}}^\beta (\mathrm{id} - \widehat{\triangle}^q)^k e^{\widehat{\triangle}^q} f(x) \right)^2
= \left( \int_{\R^n} \frac{\langle x\rangle_{\mathrm{sR}}^\beta}{\langle y\rangle_{\mathrm{sR}}^\beta} (\mathrm{id} - \widehat{\triangle}^q)_x^k\, \widehat{e}^q(1,x,y) \langle y\rangle_{\mathrm{sR}}^\beta f(y)\, dy \right)^2 \\
\leq \int_{\R^n} \left( \frac{\langle x\rangle_{\mathrm{sR}}}{\langle y\rangle_{\mathrm{sR}}} \right)^{2\beta} \left\vert (\mathrm{id} - \widehat{\triangle}^q)_x^k\, \widehat{e}^q(1,x,y) \right\vert dy   \int_{\R^n} \left\vert (\mathrm{id} - \widehat{\triangle}^q)_x^k\, \widehat{e}^q(1,x,y) \right\vert \langle y\rangle_{\mathrm{sR}}^{2\beta} f(y)^2\, dy 
\end{multline}
for every $f\in C^\infty_c(\R^n)$.

We claim that $\Cst\, \widehat{d}^q(0,x)\leq \Vert x\Vert_{\mathrm{sR}} \leq \Cst\, \widehat{d}^q(0,x)$ for every $x\in\R^n$ (see \cite[Lemma 2.1 p. 28]{Jean_2014}). Indeed, the distance $\widehat{d}^q(0,\cdot)$ is continuous (by the H\"ormander condition) and thus reaches its bounds on the compact set $\{x\in\R^n\ \mid\ \Vert x\Vert_{\mathrm{sR}}=1\}$; the inequality follows by homogeneity.
We infer that
\begin{equation}\label{inegcrochetjap}
\frac{\langle x\rangle_{\mathrm{sR}}}{\langle y\rangle_{\mathrm{sR}}} \leq \Cst \frac{1+\widehat{d}^q(0,x)}{1+\widehat{d}^q(0,y)} \leq \Cst (1+\widehat{d}^q(x,y)) \qquad\forall x,y\in\R^n.
\end{equation}
It follows from \eqref{upper_estim_kernel} and \eqref{inegcrochetjap} that
\begin{multline*}
\int_{\R^n} \left( \frac{\langle x\rangle_{\mathrm{sR}}}{\langle y\rangle_{\mathrm{sR}}} \right)^{2\beta} \left\vert (\mathrm{id} - \widehat{\triangle}^q)_x^k\, \widehat{e}^q(1,x,y) \right\vert dy
\leq \Cst(k) \int_{\R^n} (1+\widehat{d}^q(x,y))^{2\beta} e^{-\widehat{d}^q(x,y)^2/8} \, dy \\
\leq \Cst(k) \int_{\R^n} (1+\widehat{d}^q(x,y))^{2\beta} e^{-\widehat{d}^q(x,y)^2/8} \, dy 
\leq \Cst(k,\beta) \int_{\R^n} e^{-\widehat{d}^q(x,y)^2/16} \, dy =\Cst(k,\beta) 
\end{multline*}
and hence, plugging into \eqref{ineg_schur}, integrating with respect to $x$ and using \eqref{upper_estim_kernel} again, we obtain
\begin{multline}\label{estim_for_t1}
\big\Vert e^{\widehat{\triangle}^q} f \big\Vert_{\widehat{\mathcal{D}}^{q,\mathrm{sR}}_{k,\beta}}^2 
= \int_{\R^n} \left( \langle x\rangle_{\mathrm{sR}}^\beta (\mathrm{id} - \widehat{\triangle}^q)^k e^{\widehat{\triangle}^q} f(x) \right)^2 dx \\
\leq \Cst(k,\beta) \int_{\R^n} e^{-\widehat{d}^q(x,y)^2/8} \, dx   \int_{\R^n} \langle y\rangle_{\mathrm{sR}}^{2\beta} f(y)^2\, dy
= \Cst(k,\beta) \big\Vert f \big\Vert_{\widehat{\mathcal{D}}^{q,\mathrm{sR}}_{0,\beta}}^2 .
\end{multline}
We have obtained \eqref{estim_domainhat} for $t=1$.

To obtain \eqref{estim_domainhat} for every $t\in(0,t_1]$, we make the changes of variable $x=\varphi(z)$ and $y=\varphi(z)$ in \eqref{estim_for_t1}, with $\varphi=\delta_{1/\sqrt{t}}$. Noting that $\varphi^* \widehat{\triangle}^q \varphi_* = t \widehat{\triangle}^q$, $\varphi^* e^{\widehat{\triangle}^q} \varphi_* = e^{t \widehat{\triangle}^q}$, $\varphi^*dx=\vert\det(\varphi)\vert\, dx=\frac{1}{t^{\mathcal{Q}(q)/2}}\, dx$, and $\Vert\varphi(x)\Vert_{\mathrm{sR}} = \frac{1}{\sqrt{t}} \Vert x\Vert_{\mathrm{sR}}$, we get from \eqref{estim_for_t1} that
\begin{multline*}
\frac{1}{t^{\mathcal{Q}(q)/2}}\int_{\R^n} \left( \left(1+ \frac{1}{t} \Vert z\Vert_{\mathrm{sR}}^2 \right)^{2\beta} (\mathrm{id} - t\widehat{\triangle}^q)^k e^{t\widehat{\triangle}^q} (\varphi^*f)(z) \right)^2 dz \\
\leq \Cst(k,\beta) \frac{1}{t^{\mathcal{Q}(q)/2}} \int_{\R^n} \left(1+ \frac{1}{t} \Vert z\Vert_{\mathrm{sR}}^2 \right)^{2\beta} (\varphi^*f)(z)^2\, dz
\end{multline*}
from which \eqref{estim_domainhat} is easily inferred.
\end{proof}

Using \eqref{equivDjalpha} and the embeddings \eqref{inclusions_global} and \eqref{inclusions_global_neg}, we infer the following result. 

\begin{corollary}
Given any $t>0$, any $j,k\in\Z$ and any real number $\beta\geq 1$, the operator $e^{t\widehat{\triangle}^q}$ is bounded as an operator from $H^{j\sigma}_{\beta-jN}(\R^n)$ when $j\leq 0$ (resp., $H^{2j}_{\beta+2jN}(\R^n)$ when $j\geq 0$) to $H^{2k}_{\beta+2kN}(\R^n)$ when $k\leq 0$ (resp., $H^{k\sigma}_{\beta-kN}(\R^n)$ when $k\geq 0$).
Moreover, for all $t_0,t_1\in(0,+\infty)$, their norms are uniformly bounded with respect to $t\in[t_0,t_1]$.
\end{corollary}

\begin{remark}
Given any $t>0$, it is not true that $e^{t\widehat{\triangle}^q}$ maps continuously any $H^\alpha_\beta(\R^n)$ to any $H^{\alpha'}_{\beta'}(\R^n)$. Indeed, otherwise, using Hilbert-Schmidt norms we would infer that its heat kernel $e(t,\cdot,\cdot)$ is bounded in the Schwartz space $\mathcal{S}(\R^n\times\R^n)$, which is not true around the diagonal.
\end{remark}

\paragraph{Acknowledgment.}
This work was supported by the grant ANR-15-CE40-0018 \emph{SRGI} (Sub-Riemannian Geometry and Interactions).

\small
\bibliographystyle{plain}
\bibliography{bib_sR}

\end{document}